\title{\vspace{-0.7cm}Combinatorial theorems in sparse random sets}
\author{D. Conlon\thanks{Mathematical Institute, Oxford OX2 6GG, UK.
E-mail: {\tt david.conlon@maths.ox.ac.uk}. Research supported by a
Royal Society University Research Fellowship.} \and W. T. Gowers
\thanks{Department of Pure Mathematics and Mathematical Statistics,
Wilberforce Road, Cambridge CB3 0WB, UK. Email: {\tt
W.T.Gowers@dpmms.cam.ac.uk}. Research supported by a Royal Society 2010 Anniversary Research Professorship.}}

\documentclass[11pt]{article}
\usepackage{amsfonts,amssymb,amsmath,latexsym,amsthm}

\def\qed{\ifvmode\mbox{ }\else\unskip\fi\hskip 1em plus 10fill$\Box$}

\def\E{\mathbb{E}}
\def\Z{\mathbb{Z}}
\def\R{\mathbb{R}}

\def\N{\mathbb{N}}
\def\P{\mathbb{P}}

\def\V{\mathbb{V}}

\def\r{\mathbf{r}}

\def\a{\alpha}
\def\b{\beta}
\def\g{\gamma}
\def\d{\delta}
\def\D{\Delta}
\def\e{\epsilon}
\def\l{\lambda}
\def\r{\rho}

\def\sp#1{\langle #1\rangle}

\newtheorem{theorem}{Theorem}[section]

\newtheorem{lemma}[theorem]{Lemma}

\newtheorem{corollary}[theorem]{Corollary}
\newtheorem{conjecture}[theorem]{Conjecture}

\newtheorem*{assumption*}{Main assumption}
\newtheorem*{properties*}{Four key properties}

\newtheorem{definition}[theorem]{Definition}
\newtheorem*{notation*}{Notation}

\newtheorem*{condition1*}{Condition 1}
\newtheorem*{condition2*}{Condition 2}

\date{}
\begin{document}
\maketitle

\begin{abstract}
We develop a new technique that allows us to show in a unified
way that many well-known combinatorial theorems, including
Tur\'an's theorem, Szemer\'edi's theorem and Ramsey's theorem,
hold almost surely inside sparse random sets. For instance,
we extend Tur\'an's theorem to the random setting by 
showing that for every $\e > 0$ and every positive integer $t \geq 3$ there 
exists a constant $C$ such that, if $G$ is a random graph 
on $n$ vertices where each edge is chosen independently 
with probability at least $C n^{-2/(t+1)}$, then, with probability 
tending to 1 as $n$ tends to infinity, every subgraph of $G$ with at least 
$\left(1 - \frac{1}{t-1} + \e\right) e(G)$ edges
contains a copy of $K_t$. This is sharp up to the constant $C$. We also show 
how to prove sparse analogues of structural results, giving two main applications,
a stability version of the random Tur\'an theorem stated above and a sparse hypergraph 
removal lemma. Many similar results have recently been obtained independently
in a different way by Schacht and by Friedgut, R\"{o}dl and Schacht.
\end{abstract}

\section{Introduction}

In recent years there has been a trend in combinatorics towards
proving that certain well-known theorems, such as Ramsey's
theorem, Tur\'an's theorem and Szemer\'edi's theorem, have 
``sparse random'' analogues. For instance, the first non-trivial
case of Tur\'an's theorem asserts that a subgraph of $K_n$ with
more than $\lfloor n/2\rfloor \lceil n/2\rceil$ edges must contain
a triangle. A sparse random analogue of this theorem is the assertion
that if one defines a random subgraph $G$ of $K_n$ by choosing each
edge independently at random with some very small probability $p$,  
then with high probability every subgraph $H$ of $G$ 
such that $|E(H)|\geq \left(\frac{1}{2}+\e\right)|E(G)|$ will contain a triangle.
Several results of this kind have been proved, and in some 
cases, including this one, the exact bounds on what $p$ one
can take are known up to a constant factor.

The greatest success in this line of research has been with
analogues of Ramsey's theorem \cite{R30}. Recall that Ramsey's
theorem (in one of its many forms) states that, for every graph $H$ 
and every natural number $r$, there exists $n$ such that if the
edges of the complete graph $K_n$ are coloured with $r$ colours,
then there must be a copy of $H$ with all its edges of the same
colour. Such a copy of $H$ is called \textit{monochromatic}. 

Let us say that a graph $G$ is $(H,r)$-\textit{Ramsey} if,
however the edges of $G$ are coloured with $r$ colours, there
must be a monochromatic copy of $H$. After
efforts by several researchers \cite{FR86, LRV92, RR93, RR94,
RR95}, most notably R\"{o}dl and Ruci\'nski, the following
impressive theorem, a ``sparse random version'' of Ramsey's 
theorem, is now known. We write $G_{n,p}$ for the standard 
binomial model of random graphs, where each edge in an $n$-vertex graph is chosen
independently with probability $p$. We also write $v_H$ and
$e_H$ for the number of vertices and edges, respectively, in
a graph $H$. 

\begin{theorem} \label{RelativeRamsey}
Let $r \geq 2$ be a natural number and let $H$ be a graph that
is not a forest consisting of stars and paths of length $3$. Then there exist positive constants $c$ and $C$ such that
\[\lim_{n \rightarrow \infty} \mathbb{P} (G_{n,p} \mbox{ is $(H,r)$-Ramsey}) =
\left\{ \begin{array}{ll} 0, \mbox{ if } p < c n^{-1/m_2(H)},\\ 1,
\mbox{ if } p > C n^{-1/m_2(H)},
\end{array} \right.\] where
\[m_2(H) = \max_{K \subset H, v_K \geq 3} \frac{e_K - 1}{v_K - 2}.\]
\end{theorem}

That is, given a graph $G$ that is not a disjoint union of stars and paths of length $3$,
there is a threshold at approximately $p =n^{-1/m_2(H)}$ where the
probability that the random graph $G_{n,p}$ is $(H,r)$-Ramsey
changes from 0 to~1. 

This theorem comes in two parts: the statement that above the
threshold the graph is almost certainly $(H,r)$-Ramsey and the
statement that below the threshold it almost certainly is not.
We shall follow standard practice and call these the 
1-\textit{statement} and the 0-\textit{statement}, respectively.

There have also been some efforts towards proving sparse random
versions of Tur\'an's theorem, but these have up to now been less
successful. Tur\'an's theorem \cite{T41}, or rather its
generalization, the Erd\H{o}s-Stone-Simonovits theorem (see for
example \cite{B78}), states that if $H$ is some fixed graph, then any
graph with $n$ vertices that contains more than
\[\left(1 - \frac{1}{\chi(H) - 1} + o(1)\right) \binom{n}{2}\]
edges must contain a copy of $H$. Here, $\chi(H)$ is the
chromatic number of $H$. 

Let us say that a graph $G$ is $(H,\e)$-\textit{Tur\'an} if 
every subgraph of $G$ with at least 
\[\left(1 - \frac{1}{\chi(H) - 1} + \epsilon\right) e(G)\]
edges contains a copy of $H$. One may then ask for
the threshold at which a random graph becomes
$(H,\e)$-Tur\'an. The conjectured answer \cite{HKL95, HKL96, KLR97}
is that the threshold is the same as it is for the corresponding
Ramsey property.

\begin{conjecture} \label{RelativeTuran}
For every $\epsilon > 0$ and every graph $H$ there exist positive
constants $c$ and $C$ such that
\[\lim_{n \rightarrow \infty} \mathbb{P} (G_{n,p} \mbox{ is
$(H,\epsilon)$-Tur\'an}) = 
\left\{ \begin{array}{ll} 0, \mbox{ if } p < c n^{-1/m_2(H)},\\ 1,
\mbox{ if } p > C n^{-1/m_2(H)},
\end{array} \right.\] where
\[m_2(H) = \max_{K \subset H, v_K \geq 3} \frac{e_K - 1}{v_K - 2}.\]
\end{conjecture}

A difference between this conjecture and Theorem \ref{RelativeRamsey}
is that the 0-statement in this conjecture is very simple to prove. To see
this, suppose that $p$ is such that the expected number of copies of
$H$ in $G_{n,p}$  is significantly less than the expected number 
of edges in $G_{n,p}$. Then, since the number of copies of $H$ and the 
number of edges are both concentrated around their expectations, 
we can almost always remove a small number of edges
from $G_{n,p}$ and get rid of all copies of $H$, which proves that $G_{n,p}$
is not $(H,\e)$-Tur\'an. The expected number of copies of $H$
(if we label the vertices of $H$) is approximately $n^{v_H}p^{e_H}$,
while the expected number of edges in $G_{n,p}$ is approximately
$pn^2$. The former becomes less than the latter when $p=n^{-(v_H-2)/(e_H-1)}$.

A further observation raises this bound. Suppose, for example, that $H$ is 
a triangle with an extra edge attached to one of its vertices. It is clear that the
real obstacle to finding copies of $H$ is finding triangles: it is not
hard to add edges to them. More generally, if $H$ has a subgraph $K$
with $\frac{e_K-1}{v_K-2}>\frac{e_H-1}{v_H-2}$, then we can increase
our estimate of $p$ to $n^{-(v_K-2)/(e_K-1)}$, since if we can get rid
of copies of $K$ then we have got rid of copies of $H$. Beyond this
extra observation, there is no obvious way of improving the bound for
the 0-statement, which is why it is the conjectured upper bound as well.

An argument along these lines does not work at all for the Ramsey property,
since if one removes a few edges in order to eliminate all copies of $H$
in one colour, then one has to give them another colour. Since the set of
removed edges is likely to look fairly random, it is not at all clear 
that this can be done in such a way as to eliminate all monochromatic
copies of $H$.

Conjecture \ref{RelativeTuran} is known to be true for some graphs, for example
$K_3$, $K_4$, $K_5$ (see \cite{FR86, KLR97, GSS04},
respectively) and all cycles (see \cite{F94, HKL95, HKL96}), but it is open in general. Some partial results towards the general conjecture, where the $1$-statement is proved
with a weaker exponent, have been given by Kohayakawa, R\"{o}dl
and Schacht \cite{KRS04} and Szab\'o and Vu \cite{SV03}. The paper of
Szab\'o and Vu contains the best known upper bound in the case where 
$H$ is the complete graph $K_t$ for some $t\geq 6$; the bound they 
obtain is $p=n^{-1/(t - 1.5)}$, whereas the conjectured best possible
bound is $p=n^{-2/(t+1)}$ (since $m_2(K_t)$ works out to be $(t+1)/2$).
Thus, there is quite a significant gap. The full conjecture has also been proved to be a consequence of the 
so-called K\L R conjecture \cite{KLR97} of Kohayakawa, \L uczak and 
R\"{o}dl. This conjecture, regarding the number of $H$-free graphs of a
certain type, remains open, except in a few special cases \cite{GKRS07, GPSST07, GSS04, KLR96}.\footnote{The full K\L R conjecture was subsequently established by Balogh, Morris and Samotij \cite{BMS14} and by Saxton and Thomason \cite{ST14} (see also \cite{CGSS14}). Their methods also allow one to give alternative proofs for many of the results in this paper. We refer the reader to \cite{C14} for a more complete overview.}

As noted in \cite{K97, KLR97}, the K\L R conjecture would also imply the following structural
result about $H$-free graphs which contain nearly the extremal number of edges. The analogous
result in the dense case, due to Simonovits \cite{S68}, is known as the stability theorem. Roughly speaking, it says that if an $H$-free graph contains almost $\left(1 - \frac{1}{\chi(H) - 1}\right) \binom{n}{2}$ edges, then it must be very close to being $(\chi(H) - 1)$-partite.

\begin{conjecture} \label{RelativeStab}
Let $H$ be a graph with $\chi(H) \geq 3$ and let 
\[m_2(H) = \max_{K \subset H, v_K \geq 3} \frac{e_K - 1}{v_K - 2}.\]
Then, for every $\d > 0$, there exist positive constants $\e$ and $C$ such that 
if $G$ is a random graph on $n$ vertices, where each edge is chosen independently 
with probability $p$ at least $C n^{-1/m_2(H)}$, then, with probability 
tending to 1 as $n$ tends to infinity, every $H$-free subgraph of $G$ with at least $\left(1 - \frac{1}{\chi(H)-1} - \e\right) e(G)$ edges may be made $(\chi(H)-1)$-partite by removing at most $\d p n^2$ edges.
\end{conjecture}

Another example where some success has been achieved is Szemer\'edi's 
theorem \cite{Sz75}. This celebrated theorem states
that, for every positive real number $\delta$ and every natural
number $k$, there exists a positive integer $n$ such that every 
subset of the set $[n]= \{1, 2, \cdots, n\}$ of size at least $\delta n$
contains a $k$-term arithmetic progression. The particular
case where $k = 3$ had been proved much earlier by Roth
\cite{R53}, and is accordingly known as Roth's theorem. A sparse
random version of Roth's theorem was proved by Kohayakawa, 
\L uczak and R\"{o}dl \cite{KLR96}. To state the theorem, let 
us say that a subset $I$ of the integers is $\delta$-\textit{Roth} 
if every subset of $I$ of size $\delta |I|$ contains a $3$-term arithmetic 
progression. We shall also write $[n]_p$ for a random
set in which each element of $[n]$ is chosen independently with
probability $p$.

\begin{theorem} \label{RelativeRoth}
For every $\delta > 0$ there exist positive constants $c$ and $C$ such that
\[\lim_{n \rightarrow \infty} \mathbb{P} ([n]_p \mbox{ is $\delta$-Roth}) =
\left\{ \begin{array}{ll} 0, \mbox{ if } p < c n^{-1/2},\\ 1,
\mbox{ if } p > C n^{-1/2}.
\end{array} \right.\]
\end{theorem}

Once again the 0-statement is trivial (as it tends to be for density theorems): 
if $p=n^{-1/2}/2$, then the expected number of $3$-term progressions in $[n]_p$ is less than $n^{1/2}/8$, while the expected number 
of elements of $[n]_p$ is $n^{1/2}/2$. Therefore, one can almost
always remove an element from each progression and still be left
with at least half the elements of $[n]_p$. 

For longer progressions, the situation has been much less satisfactory.
Let us define
a set $I$ of integers to be $(\delta,k)$-\textit{Szemer\'edi} if every
subset of $I$ of cardinality at least $\delta|I|$ contains a $k$-term arithmetic
progression.
Until recently, hardly anything was known at all about which
random sets were $(\d,k)$-Szemer\'edi. However, that changed
with the seminal paper of Green and Tao \cite{GT08}, who, on the way
to proving that the
primes contain arbitrarily long arithmetic progressions, showed
that every \textit{pseudo}random set is $(\d,k)$-Szemer\'edi, if 
``pseudorandom" is defined in an appropriate way. Their definition
of pseudorandomness is somewhat complicated, but it is straightforward
to show that quite sparse random sets are pseudorandom in their
sense. From this the following result follows, though we are not
sure whether it has appeared explicitly in print.

\begin{theorem}
For every $\d>0$ and every $k \in \N$ there exists a function $p=p(n)$
tending to zero with $n$ such that 
\[\lim_{n \rightarrow \infty} \mathbb{P} ([n]_p \mbox{ is $(\d,k)$-Szemer\'edi}) = 1.\]
\end{theorem}
 
The approach of Green and Tao depends heavily on the use of
a set of norms known as \textit{uniformity norms}, introduced in \cite{G01}. In order
to deal with $k$-term arithmetic progressions, one must use a uniformity
norm that is based on a count of certain configurations that can be thought
of as $(k-1)$-dimensional parallelepipeds. These configurations have $k$ 
degrees of freedom (one for each dimension and one because the
parallelepipeds can be translated) and size $2^{k-1}$. A simple argument 
(similar to the arguments for the 0-statements in the density theorems above)
shows that the best bound
that one can hope to obtain by their methods is therefore at most
$p=n^{-k/2^{k-1}}$. This is far larger than the bound that arises in the
obvious 0-statement for Szemer\'edi's theorem: the same argument that 
gives a bound of $cn^{-1/2}$ for the Roth property gives a bound of 
$cn^{-1/(k-1)}$ for the Szemer\'edi property. However, even $p=n^{-k/2^{k-1}}$ is not
the bound that they actually obtain, because they need in addition a 
``correlation condition" that is not guaranteed by the smallness of the 
uniformity norm. This means that the bound they obtain is of the form 
$n^{-o(1)}$.

The natural conjecture is that the obvious bound for the 0-statement is 
in fact correct, so it is far stronger than the bound of Green and Tao.

\begin{conjecture} \label{RelativeSzem}
For every $\delta > 0$ and every positive integer $k \geq 3$, there exist positive constants $c$ and $C$ such that
\[\lim_{n \rightarrow \infty} \mathbb{P} ([n]_p \mbox{ is $(\delta,k)$-Szemer\'edi}) =
\left\{ \begin{array}{ll} 0, \mbox{ if } p < c n^{-1/(k-1)},\\ 1,
\mbox{ if } p > C n^{-1/(k-1)}.
\end{array} \right.\]
\end{conjecture}

One approach to proving Szemer\'edi's theorem is known as the hypergraph removal lemma. Proved independently by Nagle, R\"{o}dl, Schacht and Skokan \cite{NRS06, RS04} and by the second author \cite{G07} (see also~\cite{T06}), this theorem states that for every $\d > 0$ and every positive integer $k \geq 2$ there exists a constant $\e > 0$ such that if $G$ is a $k$-uniform hypergraph containing at most $\e n^{k+1}$ copies of the complete $k$-uniform hypergraph $K_{k+1}^{(k)}$ on $k+1$ vertices, then it may be made $K_{k+1}^{(k)}$-free by removing at most $\d n^k$ edges. Once this theorem is known, Szemer\'edi's theorem follows as an easy consequence. The question of whether an analogous result holds within random hypergraphs was posed by \L uczak \cite{L06}. For $k = 2$, this follows from the work of Kohayakawa, \L uczak and R\"odl \cite{KLR96}.

\begin{conjecture} \label{RelativeRemoval}
For every $\d > 0$ and every integer $k \geq 2$ there exist constants $\e > 0$ and $C$ such that, if $H$ is a random $k$-uniform hypergraph on $n$ vertices where each edge is chosen independently with probability $p$ at least $C n^{-1/k}$, then, with probability 
tending to 1 as $n$ tends to infinity, every subgraph of $H$ containing at most $\e p^{k+1} n^{k+1}$ copies of the complete $k$-uniform hypergraph $K_{k+1}^{(k)}$ on $k+1$ vertices may be made $K_{k+1}^{(k)}$-free by removing at most $\d p n^k$ edges.
\end{conjecture}

\subsection{The main results of this paper}

In the next few sections we shall give a very general method for proving sparse
random versions of combinatorial theorems. This method allows one 
to obtain sharp bounds for several theorems, of
which the principal (but by no means only) examples are positive answers to the conjectures we 
have just mentioned. This statement comes with one caveat. When dealing with graphs and 
hypergraphs, we shall restrict our attention to those which are well-balanced in the following 
sense. Note that most graphs of interest, including complete graphs and cycles, satisfy this condition.

\begin{definition} 
A $k$-uniform hypergraph $K$ is said to be strictly $k$-balanced if, for every subgraph $L$ of $K$,
\[\frac{e_K - 1}{v_K - k} > \frac{e_L - 1}{v_L - k}.\]
\end{definition}

The main results we shall prove in this paper (in the order in which we 
discussed them above, but not the order in which we shall prove them) are 
as follows. The first is a sparse random version of Ramsey's theorem.
Of course, as we have already mentioned, this is known: however, our
theorem applies not just to graphs but to hypergraphs, where the
problem was wide open apart from a few special cases \cite{RR98, RRS07}.
As we shall see, our methods apply just
as easily to hypergraphs as they do to graphs. We write $G_{n,p}^{(k)}$
for a random $k$-uniform hypergraph on $n$ vertices, where each
hyperedge is chosen independently with probability $p$. If $K$ is some
fixed $k$-uniform hypergraph, we say that a hypergraph is $(K,r)$-Ramsey
if every $r$-colouring of its edges contains a monochromatic
copy of $K$.

\begin{theorem} \label{ApproxRamsey}
Given a natural number $r$ and a strictly $k$-balanced $k$-uniform hypergraph $K$, there
exists a positive constant $C$ such that
\[\lim_{n \rightarrow \infty} \mathbb{P} (G_{n,p}^{(k)} \mbox{ is
$(K,r)$-Ramsey}) = 1, \mbox{ if } p > C n^{-1/m_k(K)},\] 
where $m_k(K) = (e_K - 1)/(v_K - k)$.
\end{theorem}

One problem that the results of this paper leave open is to establish
a corresponding 0-statement for Theorem \ref{ApproxRamsey}. The above
bound is the threshold below which the number of copies of $K$ 
becomes less than the number of hyperedges, so the results for
graphs make it highly plausible that the 0-statement holds when
$p<cn^{-1/m_k(K)}$ for small enough $c$. However, the example of stars, for which 
the threshold is lower than expected, shows that we cannot take
this result for granted.

We shall also prove Conjecture \ref{RelativeTuran} for strictly 2-balanced graphs. In particular, it holds for complete graphs.

\begin{theorem} \label{ApproxTuran}
Given $\epsilon > 0$ and a strictly 2-balanced graph $H$, there exists a positive constant $C$ 
such that
\[\lim_{n \rightarrow \infty} \mathbb{P} (G_{n,p} \mbox{ is
$(H,\epsilon)$-Tur\'an}) = 1, \mbox{ if } p > C n^{-1/m_2(H)},\] 
where $m_2(H) = (e_H - 1)/(v_H - 2)$.  
\end{theorem}

A slightly more careful application of our methods also allows us to prove its structural counterpart, Conjecture \ref{RelativeStab}, for strictly $2$-balanced graphs.

\begin{theorem} \label{ApproxStab}
Given a strictly $2$-balanced graph $H$ with $\chi(H) \geq 3$ and a constant $\d > 0$, there exist positive constants $C$ and $\e$ such that in the random graph $G_{n,p}$ chosen with probability $p \geq C n^{-1/m_2(H)}$, where $m_2(H) = (e_H - 1)/(v_H - 2)$, the following holds with probability tending to 1 as $n$ tends to infinity. Every $H$-free subgraph of $G_{n,p}$ with at least $\left(1 - \frac{1}{\chi(H) - 1} - \e\right) e(G)$ edges may be made $(\chi(H)-1)$-partite by removing at most $\d p n^2$ edges. 
\end{theorem}

We also prove Conjecture \ref{RelativeSzem}, obtaining bounds for the Szemer\'edi property that are essentially best possible.

\begin{theorem} \label{ApproxSzem}
Given $\delta > 0$ and a natural number $k \geq 3$, there exists
a constant $C$ such that
\[\lim_{n \rightarrow \infty} \mathbb{P} ([n]_p \mbox{ is $(
\delta, k)$-Szemer\'edi}) = 1, \mbox{ if } p > C n^{-1/(k-1)}.\]
\end{theorem}

Our final main result is a proof of Conjecture \ref{RelativeRemoval}, the sparse hypergraph removal lemma. As we have mentioned, the dense hypergraph removal lemma implies Szemer\'edi's theorem, but it turns out that the sparse hypergraph removal lemma does not imply Theorem \ref{ApproxSzem}. The difficulty is this. When we prove Szemer\'edi's theorem using the removal lemma, we first pass to a hypergraph to which the removal lemma can be applied. Unfortunately, in the sparse case, passing from the sparse random set to the corresponding hypergraph gives us a sparse hypergraph with dependencies between its edges, whereas in the sparse hypergraph removal lemma we assume that the edges of the sparse random hypergraph are independent. While it is likely that this problem can be overcome, we did not, in the light of Theorem \ref{ApproxSzem}, see a strong reason for doing so.

In addition to these main results, we shall discuss other density
theorems, such as Tur\'an's theorem for hypergraphs (where,
even though the correct bounds are not known in the dense
case, we can obtain the threshold at which the bounds in the
sparse random case will be the same), the multidimensional Szemer\'edi theorem of Furstenberg and Katznelson \cite{FK78} and the Bergelson-Leibman theorem \cite{BL96} concerning polynomial configurations in dense sets. In the colouring case, we shall discuss Schur's theorem \cite{S16} as a further example. Note that many similar results have also been obtained by a different method by Schacht \cite{S09} and by Friedgut, R\"{o}dl and Schacht \cite{FRS09}.

\subsection {A preliminary description of the argument}

The basic idea behind our proof is to use a \textit{transference 
principle} to deduce sparse random versions of density and
colouring results from their dense counterparts. To oversimplify
slightly, a transference principle in this context is a statement 
along the following lines. Let $X$ be a structure such as the 
complete graph $K_n$ or the set $\{1,2,\dots,n\}$, and let $U$
be a sparse random subset of $X$. Then, for every subset
$A\subset U$, there is a subset $B\subset X$ that has similar
properties to $A$. In particular, the density of $B$ is approximately
the same as the relative density of $A$ in $U$, and the number
of substructures of a given kind in $A$ is an appropriate multiple
of the number of substructures of the same kind in $B$.

Given a strong enough principle of this kind, one can prove a sparse 
random version of Szemer\'edi's theorem, say, as follows. Let $A$
be a subset of $[n]_p$ of relative density $\d$. Then there exists a 
subset $B$ of $[n]$ of size approximately $\d n$ such that the number of
$k$-term progressions in $B$ is approximately $p^{-k}$ times
the number of $k$-term progressions in $A$. From Szemer\'edi's
theorem it can be deduced that the number of $k$-term progressions 
in $B$ is at least $c(\d)n^2$, so the number of $k$-term progressions
in $A$ is at least $c(\d)p^kn^2/2$. Since the size of $A$
is about $pn$, we have roughly $pn$ degenerate progressions. Hence, 
there are non-degenerate progressions within $A$ as long as
$p^k n^2$ is significantly larger than $pn$, that is, as long as $p$ is 
at least $Cn^{-1/(k-1)}$ for some large $C$.

It is very important to the success of the above argument that
a dense subset of $[n]$ should contain not just one progression
but several, where ``several" means a number that is within a 
constant of the trivial upper bound of $n^2$. The other combinatorial
theorems discussed above have similarly ``robust" versions and again 
these are essential to us. Very roughly,
our general theorems say that a typical combinatorial theorem that
is robust in this sense will have a sparse random version with an upper
bound for the probability threshold that is very close to a natural lower bound that is trivial for density
theorems and often true, even if no longer trivial, for Ramsey theorems.

It is also very helpful to have a certain degree of 
homogeneity. For instance, in order to prove the sparse version of 
Szemer\'edi's theorem we use the fact that it is equivalent to 
the sparse version of Szemer\'edi's theorem in $\Z_n$, where
we have the nice property that for every $k$ and every $j$ with
$1\leq j\leq k$, every element $x$ appears in the
$j$th place of a $k$-term arithmetic progression in 
exactly $n$ ways (or $n-1$ if you discount the degenerate
progression with common difference 0). It will also be convenient to
assume that $n$ is prime, since in this case we know that for every pair of points
$x, y$ in $\Z_n$ there is exactly one arithmetic progression of length $k$ that
starts with $x$ and ends in $y$. This simple homogeneity property will prove 
useful when we come to do our probabilistic estimates. 

The idea of using a transference principle to obtain sparse random
versions of robust combinatorial statements is not what is
new about this paper. In fact, this was exactly the strategy
of Green and Tao in their paper on the primes, and could be said
to be the main idea behind their proof (though of course it took
many further ideas to get it to work). 
Since it is difficult to say what is new about our argument without going into
slightly more detail, we postpone further discussion for now. However,
there are three further main ideas involved and we shall highlight them
as they appear.

In the next few sections, we shall find a very general set of criteria under which one may  transfer combinatorial statements to the sparse random setting. In Sections \ref{ProbI}-\ref{ProbIII}, we shall show how to prove that these criteria hold. Section \ref{summary} is a brief summary of the general results, both conditional and unconditional, that have been proved up to that point. In Section \ref{Applications}, we show how these results may be applied to prove the various theorems promised in the introduction. In Section \ref{Conclusion}, we conclude by briefly mentioning some questions that are still open.

\subsection{Notation}

We finish this section with some notation and terminology that we shall need throughout the
course of the paper. By a {\it measure} on a finite set $X$ we shall mean a non-negative function from 
$X$ to $\R$. Usually our measures will have average value $1$, or very close to $1$. The {\it characteristic measure} $\mu$ of a subset
$U$ of $X$ will be the function defined by $\mu(x) = |X|/|U|$ if
$x \in U$ and $\mu(x) = 0$ otherwise. 

Often our set $U$ will be a random subset of $X$ with each element of $X$ chosen with probability $p$, the choices being independent. In this case, we shall use the shorthand $U=X_p$, just as we wrote $[n]_p$ for a random subset of $[n]$ in the statement of the sparse random version of Szemer\'edi's theorem earlier. When $U=X_p$ it is more convenient to consider the measure $\mu$ that is equal to $p^{-1}$ times the characteristic function of $U$. That is, $\mu(x)=p^{-1}$ if $x\in U$ and 0 otherwise. To avoid confusion, we shall call this the \textit{associated measure} of $U$. Strictly speaking, we should not say this, since it depends not just on $U$ but on the value of $p$ used when $U$ was chosen, but this will always be clear from the context so we shall not bother to call it the associated measure of $(U,p)$. 

For an arbitrary function $f$ from $X$
to $\R$ we shall write $\E_xf(x)$ for $|X|^{-1}\sum_{x\in X}f(x)$. Note that if $\mu$ is the characteristic measure of a set $U$, then $\E_x\mu(x)=1$ and $\E_x\mu(x)f(x)=\E_{x\in U}f(x)$ for any function $f$. If $U=X_p$ and $\mu$ is the associated measure of $U$, then we can no longer say this. However, we can say that the \textit{expectation} of $\E_x\mu(x)$ is 1. Also, with very high probability the cardinality of $U$ is roughly $p|X|$, so with high probability $\E_x\mu(x)$ is close to $1$. More generally, if $|f(x)|\leq 1$ for every $x\in X$, then with high probability $\E_x\mu(x)f(x)$ is close to $\E_{x\in U}f(x)$. We also take expectations over several variables: if it is clear from the context that $k$ variables $x_1,\dots,x_k$ range over finite
sets $X_1,\dots,X_k$, respectively, then $\E_{x_1,\dots,x_k}$ will be shorthand
for $|X_1|^{-1}\dots|X_k|^{-1}\sum_{x_1\in X_1}\dots\sum_{x_k\in X_k}$. If
the range of a variable is not clear from the context then we shall specify it.

We define an inner product for real-valued functions on $X$ by the formula 
$\sp{f,g}=\E_xf(x)g(x)$, and we define the $L_p$ norm by $\|f\|_p=(\E_x|f(x)|^p)^{1/p}$.
In particular, $\|f\|_1=\E_x |f(x)|$ and $\|f\|_\infty=\max_x|f(x)|$.

Let $\|.\|$ be a norm on the space $\R^X$. The {\it dual norm}
$\|.\|^*$ of $\|.\|$ is a norm on the collection of linear
functionals $\phi$ acting on $\R^X$ given by
\[\|\phi\|^* = \sup\{|\sp{f,\phi}| : \|f\| \leq 1\}.\]
It follows trivially from this definition that $|\sp{f,\phi}|\leq\|f\|\|\phi\|^*$. 
Almost as trivially, it follows that if $|\sp{f,\phi}|\leq 1$ whenever
$\|f\|\leq\eta$, then $\|\phi\|^*\leq\eta^{-1}$, a fact that will be used
repeatedly.

\section{Transference principles} \label{transfers}

As we have already mentioned, a central notion in this paper is that of 
transference. Roughly speaking,
a \textit{transference principle} is a theorem that states that every function $f$ in 
one class can be replaced by a function $g$ in another, more convenient
class in such a way that the properties of $f$ and $g$ are similar. 

To understand this concept and why it is useful, let us look at the sparse 
random version of Szemer\'edi's theorem that we shall prove. Instead of 
attacking this directly, it is convenient to prove a functional 
generalization of it. The statement we shall prove is the following.

\begin{theorem}\label{sparseszem}
For every positive integer $k$ and every $\d>0$ there are positive constants 
$c$ and $C$ with the following property. Let $p \geq Cn^{-1/(k-1)}$, let 
$U$ be a random subset of $\Z_n$ where each element is chosen 
independently with probability $p$ and let $\mu$ be the associated measure
of $U$. Then, with probability tending to $1$ as $n$ tends to infinity, every 
function $f$ such that $0\leq f\leq\mu$ and $\E_xf(x)\geq\d$ satisfies the inequality
\begin{equation*}
\E_{x,d}f(x)f(x+d)\dots f(x+(k-1)d)\geq c.
\end{equation*}
\end{theorem}

To understand the normalization, it is a good exercise (and an easy one) to check
that with high probability $\E_{x,d}\mu(x)\mu(x+d)\dots\mu(x+(k-1)d)$ is close
to 1, so that the conclusion of Theorem \ref{sparseszem} is stating that 
$\E_{x,d}f(x)f(x+d)\dots f(x+(k-1)d)$ is within a constant of its trivial maximum.
(If $p$ is smaller than $n^{-1/(k-1)}$ then this is no longer true: the main 
contribution to $\E_{x,d}\mu(x)\mu(x+d)\dots\mu(x+(k-1)d)$ comes from
the degenerate progressions where $d=0$.)

Our strategy for proving this theorem is to ``transfer" the function $f$ from the sparse set
$U$ to $\Z_n$ itself and then to deduce the conclusion from the following robust functional
version of Szemer\'edi's theorem, which can be proved by a simple averaging
argument due essentially to Varnavides \cite{V59}. 

\begin{theorem} \label{FuncSzem}
For every $\d > 0$ and every positive integer $k$ there is a constant $c > 0$
such that, for every positive integer $n$, every function $g: \Z_n \rightarrow [0, 1]$ with 
$\E_x g(x) \geq \d$ satisfies the inequality
\[\E_{x,d} g(x) g(x+d) \dots g(x+(k-1)d) \geq c.\]
\end{theorem}

\noindent Note that in this statement we are no longer talking about dense subsets of $\Z_n$, but rather 
about $[0,1]$-valued functions defined on $\Z_n$ with positive expectation. It will be important in what follows
that any particular theorem we wish to transfer has such an equivalent functional formulation. 
As we shall see in Section \ref{CondMain}, all of the theorems that we consider do have such formulations.

Returning to transference principles, our aim is to find a function
$g$ with $0\leq g\leq 1$ for which we can prove that $\E_xg(x)\approx\E_xf(x)$
and that 
\begin{equation*}
\E_{x,d}g(x)g(x+d)\dots g(x+(k-1)d)\approx\E_{x,d}f(x)f(x+d)\dots f(x+(k-1)d).
\end{equation*}
We can then argue as follows: if $\E_xf(x)\geq\d$, then $\E_xg(x)\geq\d/2$; by
Theorem \ref{FuncSzem} it follows that $\E_{x,d}g(x)g(x+d)\dots g(x+(k-1)d)$ is bounded 
below by a constant $c$; and this implies that  $\E_{x,d}f(x)f(x+d)\dots f(x+(k-1)d)\geq c/2$.

In the rest of this section we shall show how the Hahn-Banach theorem can
be used to prove general transference principles. This was first demonstrated
by the second author in \cite{G08}, and independently (in a slightly different
language) by Reingold, Trevisan, Tulsiani and Vadhan \cite{RTTV08}, and 
leads to simpler proofs than the method used by Green and 
Tao. The first transference principle we shall prove is particularly appropriate for
density theorems: this one was shown in \cite{G08} but for convenience we
repeat the proof. Then we shall prove a modification of it for use with
colouring theorems. 

Let us begin by stating the finite-dimensional Hahn-Banach theorem in
its separation version.

\begin{lemma}\label{hb:real}
Let $K$ be a closed convex set in $\mathbb{R}^n$ and let $v$ be a vector
that does not belong to $K$. Then there is a real number $t$ and a linear 
functional $\phi$ such that $\phi(v)>t$ and such that $\phi(w)\leq t$ for every $w\in K$.
\end{lemma}

The reason the Hahn-Banach theorem is useful to us is that one often wishes to prove 
that one function is a sum of others with certain properties, and often the sets 
of functions that satisfy those properties
are convex (or can easily be made convex). For instance, we shall want to
write a function $f$ with $0\leq f\leq\mu$ as a sum $g+h$ with $0\leq g\leq 1$ 
and with $h$ small in a certain norm. The following lemma, an almost
immediate consequence of Lemma \ref{hb:real}, tells us what happens 
when a function \textit{cannot} be decomposed in this way. We implicitly use 
the fact that every linear functional on $\R^Y$ has the
form $f\mapsto\sp{f,\phi}$ for some $\phi$.

\begin{lemma}\label{2convexsets:real}
Let $Y$ be a finite set and let $K$ and $L$ be two subsets of $\R^Y$ that
are closed and convex and that contain 0. Suppose that $f\notin K+L$. Then
there exists a function $\phi\in\R^Y$ such that $\sp{f,\phi}>1$ and such
that $\sp{g,\phi}\leq 1$ for every $g\in K$ and $\sp{h,\phi}\leq 1$ for
every $h\in L$.
\end{lemma}

\begin{proof}
By Lemma \ref{hb:real} there is a function $\phi$ and a real number $t$
such that $\sp{f,\phi}>t$
and such that $\sp{g+h,\phi}\leq t$ whenever $g\in K$ and $h\in L$.
Setting $h=0$ we deduce that $\sp{g,\phi}\leq t$ for every $g\in K$,
and setting $g=0$ we deduce that $\sp{h,\phi}\leq t$ for every $h\in L$.
Setting $g=h=0$ we deduce that $t\geq 0$. Dividing through by
$t$ (or by $\frac{1}{2}\sp{f,\phi}$ if $t=0$) we see that we may take $t$ to be 1.
\end{proof}

Now let us prove our two transference principles, beginning with the
density one. In the statement of the theorem
below we write $\phi_+$ for the positive part of $\phi$.

\begin{lemma}\label{transfer:density}
Let $\e$ and $\eta$ be positive real numbers, let $\mu$ and $\nu$ be 
non-negative functions defined on a finite set $X$ and let $\|.\|$ be
a norm on $\R^X$. Suppose that 
$\sp{\mu-\nu,\phi_+}\leq\e$ whenever $\|\phi\|^*\leq\eta^{-1}$. Then
for every function $f$ with $0\leq f\leq\mu$ there exists a function $g$
with $0\leq g\leq\nu$ such that $\|(1+\e)^{-1}f-g\|\leq\eta$.
\end{lemma}

\begin{proof}
If we cannot approximate $(1+\e)^{-1}f$ in this way, then we cannot write 
$(1+\e)^{-1}f$ as a sum $g+h$ with $0\leq g\leq\nu$ and $\|h\|\leq\eta$. 
Now the sets $K=\{g:0\leq g\leq\nu\}$ and $L=\{h:\|h\|\leq\eta\}$ are 
closed and convex and they both contain 0. It follows from Lemma 
\ref{2convexsets:real}, with $Y = X$, that there is a function $\phi$ with the following 
three properties.
\begin{itemize}
\item $\sp{(1+\e)^{-1}f,\phi}>1$;
\item $\sp{g,\phi}\leq 1$ whenever $0\leq g\leq\nu$;
\item $\sp{h,\phi}\leq 1$ whenever $\|h\|\leq\eta$.
\end{itemize}
From the first of these properties we deduce that $\sp{f,\phi}>1+\e$. From
the second we deduce that $\sp{\nu,\phi_+}\leq 1$, since the function $g$
that takes the value $\nu(x)$ when $\phi(x)\geq 0$ and $0$ otherwise
maximizes the value of $\sp{g,\phi}$ over all $g\in K$. And from the third
property we deduce immediately that $\|\phi\|^*\leq\eta^{-1}$.

But our hypothesis implies that $\sp{\mu,\phi_+}\leq\sp{\nu,\phi_+}+\e$.
It therefore follows that
\begin{equation*}
1+\e<\sp{f,\phi}\leq\sp{f,\phi_+}\leq\sp{\mu,\phi_+}\leq\sp{\nu,\phi_+}+\e\leq 1+\e,
\end{equation*}
which is a contradiction.
\end{proof}

Later we shall apply Lemma \ref{transfer:density} with $\mu$ the 
associated measure of a sparse random set and $\nu$ the constant 
measure 1. 

The next transference principle is the one that we shall use for obtaining
sparse random colouring theorems. It may seem strange that the condition
we obtain on $g_1+\dots+g_r$ is merely that it is less than $\nu$ (rather
than equal to $\nu$). However, we also show that $f_i$ and $g_i$ are
close in a certain sense, and in applications that will imply that $g_1+\dots+g_r$
is indeed approximately equal to $\nu$ (which will be the constant measure 1).
With a bit more effort, one could obtain equality from the Hahn-Banach method,
but this would not make life easier later, since the robust versions of Ramsey 
theorems hold just as well when you colour almost everything as they do 
when you colour everything.

\begin{lemma}\label{transfer:colouring}
Let $\e$ and $\eta$ be positive real numbers, let $r$ be a positive integer,
let $\mu$ and $\nu$ be non-negative functions defined on a finite set 
$X$ and let $\|.\|$ be a norm on $\R^X$. Suppose that 
$\sp{\mu-\nu,(\max_{1\leq i\leq r}\phi_i)_+}\leq\e$ whenever 
$\phi_1,\dots,\phi_r$ are functions with $\|\phi_i\|^*\leq\eta^{-1}$ for each $i$. 
Then for every sequence of $r$ functions $f_1,\dots,f_r$ with $f_i \geq 0$ 
for each $i$ and $f_1+\dots+f_r\leq\mu$ there exist functions $g_1,\dots,g_r$ 
with $g_i \geq 0$ for each $i$ and $g_1+\dots+g_r\leq\nu$ such that 
$\|(1+\e)^{-1}f_i-g_i\|\leq\eta$ for each $i$.
\end{lemma}

\begin{proof}
Suppose that the result does not hold for the $r$-tuple $(f_1,\dots,f_r)$.
Let $K$ be the closed convex set of all $r$-tuples of functions $(g_1,\dots,g_r)$ 
such that $g_i \geq 0$ for each $i$ and $g_1+\dots+g_r\leq\nu$, and let $L$ be
the closed convex set of all $r$-tuples $(h_1,\dots,h_r)$ such that
$\|h_i\|\leq \eta$ for each $i$. Then both $K$ and $L$ contain 0 and our 
hypothesis is that $(1+\e)^{-1}(f_1,\dots,f_r)\notin K+L$. Therefore, Lemma
\ref{2convexsets:real}, with $Y = X^r$, gives us an $r$-tuple of functions $(\phi_1,\dots,\phi_r)$ with 
the following three properties.
\begin{itemize}
\item $\sum_{i=1}^r \sp{(1+\e)^{-1}f_i,\phi_i}>1$;
\item $\sum_{i=1}^r \sp{g_i,\phi_i}\leq 1$ whenever
$g_i \geq 0$ for each $i$ and $g_1+\dots+g_r \leq \nu$;
\item $\sum_{i=1}^r \sp{h_i,\phi_i}\leq 1$ whenever $\|h_i\|\leq\eta$
for each $i$.
\end{itemize}
The first of these conditions implies that $\sum_{i=1}^r\sp{f_i,\phi_i}>1+\e$.
In the second condition, let us choose the functions $g_i$ as follows. For each
$x$, pick an $i$ such that $\phi_i(x)$ is maximal. If $\phi_i(x)\geq 0$, then set 
$g_i(x)$ to be $\nu(x)$, and otherwise set $g_i(x)=0$. For each $j\ne i$, set
$g_j(x)$ to be zero. Then $\sum_{i=1}^rg_i(x)\phi_i(x)$ is equal to 
$\nu(x)\max_i\phi_i(x)$ if this maximum is non-negative, and 0 otherwise.
Therefore, $\sum_{i=1}^r\sp{g_i,\phi_i}=\sp{\nu,(\max_i\phi_i)_+}$.
Thus, it follows from the second condition that 
$\sp{\nu,(\max_i\phi_i)_+}\leq 1$. Let us write $\phi$ for $\max_i\phi_i$. 
The third condition implies that $\|\phi_i\|^*\leq\eta^{-1}$ for each $i$.

Using this information together with our hypothesis about $\mu-\nu$, we
find that
\begin{equation*}
1+\e<\sum_{i=1}^r\sp{f_i,\phi_i}\leq\sum_{i=1}^r\sp{f_i,\phi_+}\leq\sp{\mu,\phi_+}
\leq\sp{\nu,\phi_+}+\e\leq1+\e,
\end{equation*}
a contradiction.
\end{proof}

\section{The counting lemma} \label{clemma}

We now come to the second main idea of the paper, and perhaps
the main new idea. Lemmas \ref{transfer:density} and \ref{transfer:colouring} will be very useful to us,
but as they stand they are rather abstract: in order to make use of them we
need to find a norm $\|.\|$ such that if $\|f-g\|$ is small then $f$
and $g$ behave similarly in a relevant way. Several norms have
been devised for exactly this purpose, such as the uniformity
norms mentioned earlier, and also ``box norms'' for multidimensional
structures and ``octahedral norms'' for graphs and hypergraphs.
It might therefore seem natural to try to apply Lemmas \ref{transfer:density} and \ref{transfer:colouring}
to these norms. However, as we have already commented in the case of uniformity
norms, if we do this then we cannot obtain sharp bounds: except in
a few cases, these norms are related to counts of configurations that 
are too large to appear non-degenerately in very sparse random sets. 

We are therefore forced to adopt a different approach. Instead of
trying to use an off-the-shelf norm, we use a bespoke norm, designed
to fit perfectly the problem at hand. Notice that Lemmas \ref{transfer:density} and \ref{transfer:colouring}
become harder to apply as the norm $\|.\|$ gets bigger, since then
the dual norm $\|.\|^*$ gets smaller and there are more functions
$\phi$ with $\|\phi\|^*\leq\eta^{-1}$, and therefore more functions
of the form $\phi_+$ for which one must show that $\sp{\mu-\nu,\phi_+}\leq\e$ 
(and similarly for $(\max_{1\leq i\leq r}\phi_i)_+$ with colouring problems). 
Therefore, we shall try to make our norm as small as possible,
subject to the condition we need it to satisfy: that $f$ and $g$
behave similarly if $\|f-g\|$ is small. 

Thus, our norm will be defined by means of a universal construction.
As with other universal constructions, this makes the norm easy
to define but hard to understand concretely. However, we can get 
away with surprisingly little understanding of its detailed behaviour,
as will become clear later. An advantage of this abstract approach
is that it has very little dependence on the particular problem that is being
studied: it is for that reason that we have ended up with a very general result.

Before we define the norm, let us describe the general set-up
that we shall analyse. We shall begin with a finite set $X$ and 
a collection $S$ of ordered subsets of $X$, each of size $k$. 
Thus, any element $s \in S$ may be expressed in the form 
$s = (s_1, \dots, s_k)$.

Here are two examples. When we apply our results to Szemer\'edi's
theorem, we shall take $X$ to be $\Z_n$, and $S$ to be the set
of ordered $k$-tuples of the form $(x,x+d,\dots,x+(k-1)d)$,
and when we apply it to Ramsey's theorem or Tur\'an's theorem for $K_4$, we shall
take $X$ to be the edge set of the complete graph $K_n$ and $S$ to be the set
of ordered sextuples of pairs of the form 
$(x_1x_2,x_1x_3,x_1x_4,x_2x_3,x_2x_4,x_3x_4)$,
where $x_1$, $x_2$, $x_3$ and $x_4$ are vertices of $K_n$. 
Depending on the particular circumstance, we shall choose whether to include or 
ignore degenerate configurations. For example, for Szemer\'edi's theorem, it is
convenient to include the possibility that $d = 0$, but for theorems involving $K_4$, we
restrict to configurations where $x_1$, $x_2$, $x_3$ and $x_4$ are all distinct. In practice, 
it makes little difference, since the number of degenerate configurations is never very large.

In both these two examples, the collection $S$ of ordered subsets of $X$ has 
some nice homogeneity properties, which we shall assume for our general 
result because it makes the proofs cleaner, even if one sometimes has to 
work a little to show that these properties may be assumed. 

\begin{definition}
Let $S$ be a collection of ordered $k$-tuples $s=(s_1,\dots,s_k)$ of 
elements of a finite set $X$, and let us write $S_j(x)$ for the set of all 
$s$ in $S$ such that $s_j=x$. We shall say that $S$ is \emph{homogeneous}
if for each $j$ the sets $S_j(x)$ all have the same size.
\end{definition}

\noindent We shall assume throughout that our sets of ordered $k$-tuples are
homogeneous in this sense. Note that this assumption does not hold for
arithmetic progressions of length $k$ if we work in
the set $[n]$ rather than the set $\Z_n$. However, sparse random 
Szemer\'edi for $\Z_n$ implies sparse random Szemer\'edi for $[n]$,
so this does not bother us. Similar observations can be used to 
convert several other problems into equivalent ones for which
the set $S$ is homogeneous. Moreover, such observations will easily 
accommodate any further homogeneity assumptions that we have to
introduce in later sections.

The functional version of a combinatorial theorem about the
ordered sets in $S$ will involve expressions such as
\[\E_{s\in S}f(s_1)\dots f(s_k).\]
Thus, what we wish to do is define a norm $\|.\|$ with the property 
that
\[\E_{s\in S}f(s_1)\dots f(s_k)-\E_{s\in S}g(s_1)\dots g(s_k)\]
can be bounded above in terms of $\|f-g\|$ whenever $0 \leq f \leq \mu$ 
and $0 \leq g \leq \nu$. This is what we mean by saying that $f$ and
$g$ should behave similarly when $\|f-g\|$ is small.

The feature of the problem that gives us a simple and natural norm
is the $k$-linearity of the expression $\E_{s\in S}f(s_1)\dots f(s_k)$,
which allows us to write the above difference as
\begin{equation*}
\sum_{j=1}^k\E_{s\in S}g(s_1)\dots g(s_{j-1})(f-g)(s_j)f(s_{j+1})\dots f(s_k).
\end{equation*}
Because we are assuming that the sets $S_j(x)$ all have the same size, we
can write any expression of the form $\E_{s\in S}h_1(s_1)\dots h_k(s_k)$
as 
\begin{equation*}
\E_{x\in X}h_j(x)\E_{s\in S_j(x)}h_1(s_1)\dots h_{j-1}(s_{j-1})h_{j+1}(s_{j+1})\dots h_k(s_k).
\end{equation*}
It will be very convenient to introduce some terminology and notation for expressions 
of the kind that are beginning to appear. 

\begin{definition}
Let $X$ be a finite set and let $S$ be a homogeneous collection of ordered subsets of $X$, 
each of size $k$. Then, given $k$ functions $h_1,\dots,h_k$ from $X$ to $\R$, their 
\emph{$j$th convolution} is defined to be the function
\begin{equation*}
*_j(h_1, \dots, h_k)(x)=\E_{s\in S_j(x)}h_1(s_1)\dots h_{j-1}(s_{j-1})h_{j+1}(s_{j+1})\dots h_k(s_k).
\end{equation*}
\end{definition}
We call this a convolution because in the special case where $S$ is the set of 
arithmetic progressions of length 3 in $\Z_N$, we obtain convolutions in the
conventional sense. Using this notation and the observation made above, we can rewrite
\begin{equation*}
\E_{s\in S}g(s_1)\dots g(s_{j-1})(f-g)(s_j)f(s_{j+1})\dots f(s_k)
\end{equation*}
as $\sp{f-g,*_j(g \dots, g, f, \dots, f)}$ (where it is understood that there are $j-1$ occurrences of $g$ and $k-j$ occurrences of $f$), and from that we obtain the identity
\begin{equation*}
\E_{s\in S}f(s_1)\dots f(s_k)-\E_{s\in S}g(s_1)\dots g(s_k)=\sum_{j=1}^k\sp{f-g,*_j(g \dots, g, f, \dots, f)}.
\end{equation*}
This, together with the triangle inequality, gives us the following lemma.

\begin{lemma} \label{uselesslemma}
Let $X$ be a finite set and let $S$ be a homogeneous collection of ordered subsets 
of $X$, each of size $k$. Let $f$ and $g$ be two functions defined on $X$. Then
\begin{equation*}
|\E_{s\in S}f(s_1)\dots f(s_k)-\E_{s\in S}g(s_1)\dots g(s_k)|\leq
\sum_{j=1}^k|\sp{f-g,*_j(g,\dots,g,f,\dots,f)}|.
\end{equation*}
\end{lemma}

It follows that if $f - g$ has small inner product with all functions
of the form $*_j(g \dots, g, f, \dots, f)$, then $\E_{s \in S} f(s_1)\dots f(s_k)$ 
and $\E_{s \in S}g(s_1)\dots g(s_k)$ are close. It is tempting, therefore, to define a norm 
$\|.\|$ by taking $\|h\|$ to be the maximum value of $|\sp{h, \phi}|$ over all functions 
$\phi$ of the form $*_j(g,\dots,g,f,\dots,f)$ for which $0\leq g\leq \nu$ and $0\leq f\leq\mu$.
If we did that, then we would know that 
$|\E_{s\in S}f(s_1)\dots f(s_k)-\E_{s\in S}g(s_1)\dots g(s_k)|$ was small
whenever $\|f-g\|$ was small, which is exactly the property we need our norm
to have. Unfortunately, this definition leads to difficulties. To see why we need to look
in more detail at the convolutions.

Any convolution $*_j(g,\dots, g,f,\dots,f)$ is bounded above by $*_j(\nu, \dots, 
\nu, \mu, \dots, \mu)$. For the sake of example, let us consider the case of 
Szemer\'edi's theorem. Taking $\nu = 1$, we see that the $j$th convolution is 
bounded above by the function
\[P_j(x) = \E_d \mu(x+d) \dots \mu(x+(k-j)d).\]
Up to normalization, this counts the number of progressions of length $k-j+1$ beginning at $x$.
If $j > 1$, probabilistic estimates imply that, at the critical probability $p = C n^{-1/(k-1)}$, $P_j$ is, with high probability, $L_{\infty}$-bounded (that is, the largest value of the function is bounded by some
absolute constant). However, functions of the form $*_1(f,\dots,f)$ with $0\leq f\leq\mu$ 
are almost always unbounded. This makes it much more difficult to control their inner 
products with $\mu-1$, and we need to do that if we wish to apply the abstract transference 
principle from the previous section.

For graphs, a similar problem arises. The $j$th convolution will count, up to normalization, the
number of copies of some subgraph of the given graph $H$ that are rooted on a particular edge. If we assume that the graph is balanced, as we are doing, then, at probability $p = C n^{-1/m_2(H)}$, this count will be $L_{\infty}$-bounded for any proper subgraph of $H$. However, for $H$ itself, we do not have this luxury and the function 
$*_1(f,\dots,f)$ is again likely to be unbounded. 

If we were prepared to increase the density of the random set by a polylogarithmic factor, 
we could ensure that even $*_1(f,\dots,f)$ was bounded and this problem would go away. 
Thus, a significant part of the complication of this paper is due to our wish to get a bound that
is best possible up to a constant.

There are two natural ways of getting around the difficulty if we are not prepared to 
sacrifice a polylogarithmic factor. One is to
try to exploit the fact that although $*_1(f,\dots,f)$ is not
bounded, it typically takes large values very infrequently, so it is
``close to bounded'' in a certain sense. The other is to replace
$*_1(f,\dots,f)$ by a modification of the function that has 
been truncated at a certain maximum. It seems likely that both approaches 
can be made to work: we have found it technically easier to go for the second.
The relevant definition is as follows.

\begin{definition}
Let $X$ be a finite set and let $S$ be a homogeneous collection of ordered subsets of $X$, 
each of size $k$. Then, given $k$ non-negative functions $h_1,\dots,h_k$ from $X$ to $\R$, their \emph{$j$th capped convolution} $\circ_j (h_1, \dots, h_k)$ is defined by
\begin{equation*}
\circ_j (h_1, \dots, h_k)(x)=\min\{*_j (h_1, \dots, h_k)(x),2\}.
\end{equation*}
\end{definition}

Unlike with ordinary convolutions, there is no obvious way of controlling the 
difference between $\E_{s \in S}f(s_1)\dots f(s_k)$ and $\E_{s \in S}g(s_1)\dots
g(s_k)$ in terms of the inner product between $f-g$ and suitably
chosen capped convolutions. So instead we shall look at a quantity
that is related in a different way to the number of substructures 
of the required type. Roughly speaking, this quantity counts the number of substructures, 
but does not count too many if they start from the same point.

A natural quantity that fits this description is
$\sp{f,\circ_1(f,f,\dots,f)}$, and this is indeed closely related to
the quantity we shall actually consider.  However, there is an
additional complication, which is that it is very convenient to think of our random set $U$ as
a union of $m$ random sets $U_1,\dots,U_m$, and of a function defined
on $U$ as an average $m^{-1}(f_1+\dots+f_m)$ of functions with $f_i$
defined on $U_i$. More precisely, we shall take $m$ independent random sets 
$U_1,\dots,U_m$, each distributed as $X_p$. (Recall that $X_p$ stands for a
random subset of $X$ where the elements are chosen independently with
probability $p$.) Writing $\mu_1,\dots,\mu_m$ for
their associated measures, for each 
$i$ we shall take a function $f_i$ such that $0\leq f_i\leq\mu_i$. Our
assertion will then be about the average $f=m^{-1}(f_1+\dots+f_m)$.
Note that $0\leq f\leq\mu$, where $\mu=m^{-1}(\mu_1+\dots+\mu_m)$,
and that every function $f$ with $0\leq f\leq\mu$ can be expressed as
an average of functions $f_i$ with $0\leq f_i\leq\mu_i$. Note also that
if $U=U_1\cup\dots\cup U_m$ then $\mu$ is neither the characteristic
measure of $U$ nor the associated measure of $U$. However, provided $p$ is 
fairly small, it is close to both with high probability, and this is all that matters.

Having chosen $f$ in this way, the quantity we shall then look at is 
\begin{equation*}
\sp{f,m^{-(k-1)}\sum_{i_2,\dots,i_k}\circ_1(f_{i_2},\dots,f_{i_k})}
=\E_{i_1,\dots,i_k\in\{1,\dots,m\}}\sp{f_{i_1},\circ_1(f_{i_2},\dots,f_{i_k})}.
\end{equation*}
In other words, we expand the expression $\sp{f,*_1(f,f,\dots,f)}$ in terms
of $f_1,\dots,f_m$ and then do the capping term by term. 

Central to our approach is a
``counting lemma'', which is an easy corollary of the following result, which 
keeps track of the errors that are introduced by our ``capping". (To understand
the statement, observe that if we replaced the capped convolutions $\circ_j$
by their ``genuine" counterparts $*_j$, then the two quantities that we
are comparing would become equal.) In the next lemma, we assume that
a homogeneous set $S$ of ordered $k$-tuples has been given.

\begin{lemma}\label{precounting}
Let $\eta>0$, let $m\geq 2k^3/\eta$ and let $\mu_1,\dots,\mu_m$ be 
non-negative functions defined on $X$ with $\|\mu_i\|_1 \leq 2$ for all $i$.
Suppose that $\|*_1(\mu_{i_2},\dots,\mu_{i_k})
-\circ_1(\mu_{i_2},\dots,\mu_{i_k})\|_1\leq\eta$
whenever $i_2,\dots,i_k$ are distinct integers between 1 and $m$, and
also that $*_j(1,1,\dots,1,\mu_{i_{j+1}},\dots,\mu_{i_k})$ is uniformly
bounded above by 2 whenever $j\geq 2$ and $i_{j+1},\dots, i_k$ are distinct. For each $i$, let $f_i$ be a function
with $0\leq f_i\leq\mu_i$, let $f=\E_i f_i$ and let $g$ be a function 
with $0\leq g\leq 1$. Then
\begin{equation*}
\E_{i_1,\dots,i_k\in\{1,\dots,m\}}\sp{f_{i_1},\circ_1(f_{i_2},\dots,f_{i_k})}
-\sp{g,*_1(g,g,\dots,g)}
\end{equation*}
differs from 
\begin{equation*}
\sum_{j=1}^k
\sp{f-g,\E_{i_{j+1},\dots,i_k}\circ_j(g,g,\dots,g,f_{i_{j+1}},\dots,f_{i_k})}
\end{equation*}
by at most $2\eta$.
\end{lemma}

\begin{proof}
Note first that
\begin{eqnarray*}
\E_{i_1,\dots,i_k}\sp{f_{i_1},\circ_1(f_{i_2},\dots,f_{i_k})}
&=&\E_{i_1,\dots,i_k}\sp{f_{i_1}-g,\circ_1(f_{i_2},\dots,f_{i_k})}
+\E_{i_2,\dots,i_k}\sp{g,\circ_1(f_{i_2},\dots,f_{i_k})}\\
&=&\E_{i_2,\dots,i_k}\sp{f-g,\circ_1(f_{i_2},\dots,f_{i_k})}
+\E_{i_2,\dots,i_k}\sp{g,\circ_1(f_{i_2},\dots,f_{i_k})}.
\end{eqnarray*}
Since $0\leq *_1(f_{i_2},\dots,f_{i_k})\leq
*_1(\mu_{i_2},\dots,\mu_{i_k})$, our assumption implies that, whenever $i_2,\dots,i_k$ are distinct, $\|*_1(f_{i_2},\dots,f_{i_k})-\circ_1(f_{i_2},\dots,f_{i_k})\|_1\leq\eta$. In this case, therefore,
\begin{equation*}
0\leq\sp{g,*_1(f_{i_2},\dots,f_{i_k})}
-\sp{g,\circ_1(f_{i_2},\dots,f_{i_k})}\leq\eta.
\end{equation*}
We also know that $\sp{g,*_1(f_{i_2},\dots,f_{i_k})}
=\sp{f_{i_2},*_2(g,f_{i_3},\dots,f_{i_k})}$
and that if $i_3,\dots,i_k$ are distinct then 
$*_2(g,f_{i_3},\dots,f_{i_k})=\circ_2(g,f_{i_3},\dots,f_{i_k})$.
Therefore,
\begin{equation*}
0\leq\sp{f_{i_2},\circ_2(g,f_{i_3},\dots,f_{i_k})}
-\sp{g,\circ_1(f_{i_2},\dots,f_{i_k})}\leq\eta.
\end{equation*}

Now the assumption that
$*_j(1,1,\dots,1,\mu_{i_{j+1}},\dots,\mu_{i_k})$ is bounded above by 2
whenever $j\geq 2$ and $i_{j+1},\dots,i_k$ are distinct implies that $\circ_j(g,g,\dots,g,f_{i_{j+1}},\dots,f_{i_k})$ and
$*_j(g,g,\dots,g,f_{i_{j+1}},\dots,f_{i_k})$ are equal under these
circumstances. From this it is a small exercise to show that
\begin{equation*}
\sp{f_{i_2},\circ_2(g,f_{i_3},\dots,f_{i_k})}-\sp{g,\circ_k(g,g,\dots,g)}
=\sum_{j=2}^k\sp{f_{i_j}-g,\circ_j(g,g,\dots,g,f_{i_{j+1}},\dots,f_{i_k})}.
\end{equation*}
Therefore, for $i_2, \dots, i_k$ distinct,
\begin{equation} \label{Eq1}
\sp{g, \circ_1(f_{i_2}, \dots, f_{i_k})} - \sp{g, \circ_k(g, g, \dots, g)}
\end{equation}
differs from 
\begin{equation} \label{Eq2}
\sum_{j = 2}^k \sp{f_{i_j} - g, \circ_j (g, g, \dots, g, f_{i_{j+1}}, \dots, f_{i_k})}
\end{equation}
by at most $\eta$.

The probability that $i_1,\dots,i_k$ are not distinct is at most
$\binom k2 m^{-1}\leq\eta/4k$, and if they are not distinct then the
difference between (\ref{Eq1}) and (\ref{Eq2})
is certainly no more than $4k$ (since all capped convolutions take values in
$[0,2]$ and $\|f_{i_j}\|_1 \leq \|\mu_{i_j}\|_1 \leq 2$). Therefore, taking the expectation over all $(i_1,\dots,i_k)$
(not necessarily distinct) and
noting that $\sp{g,\circ_k(g,g,\dots,g)}=\sp{g,*_1(g,g,\dots,g)}$, 
we find that
\begin{equation*}
\E_{i_1,\dots,i_k}\sp{f_{i_1},\circ_1(f_{i_2},\dots,f_{i_k})}
-\sp{g,*_1(g,g,\dots,g)}
\end{equation*}
differs from
\begin{equation*}
\sum_{j=1}^k
\sp{f-g,\E_{i_{j+1},\dots,i_k}\circ_j(g,g,\dots,g,f_{i_{j+1}},\dots,f_{i_k})}
\end{equation*}
by at most $2\eta$, as claimed.
\end{proof}

To state our counting lemma, we need to define the norm that we shall actually use. 

\begin{definition}
Let $X$ be a finite set and let $S$ be a homogeneous collection of ordered 
subsets of $X$, each of size $k$. Let $\mu = (\mu_1,\dots,\mu_m)$ be a 
sequence of measures on $X$. A $(\mu,1)$-\emph{basic anti-uniform function}
is a function of the form $\circ_j(g, \dots, g, f_{i_{j+1}}, 
\dots, f_{i_k})$, where $1\leq j\leq k$, $i_{j+1},\dots,i_k$ are distinct, 
$0 \leq g \leq 1$ and $0 \leq f_{i_h} \leq \mu_{i_h}$ for 
every $h$ between $j+1$ and $k$. Let $\Phi_{\mu, 1}$ be the set of all 
$(\mu,1)$-basic anti-uniform functions 
and define the norm $\|.\|_{\mu,1}$ by taking $\|h\|_{\mu,1}$ to be 
$\max\{|\sp{h,\phi}|:\phi\in\Phi_{\mu,1}\}$.
\end{definition}
The phrase ``basic anti-uniform function" is borrowed from Green and Tao, 
since our basic anti-uniform functions are closely related to 
functions of the same name that appear in their paper \cite{GT08}. 

Our counting lemma is now as follows. It says that if $\|f-g\|_{\mu, 1}$ is small, then
the ``sparse" expression given by $\E_{i_1,\dots,i_k\in\{1,\dots,m\}}
\sp{f_{i_1},\circ_1(f_{i_2},\dots,f_{i_k})}$ is approximated by the ``dense" expression 
$\sp{g,*_1(g,g,\dots,g)}$. This lemma modifies Lemma \ref{uselesslemma} in
two ways: it splits $f$ up into $m^{-1}(f_1+\dots+f_m)$ and it caps all the convolutions
that appear when one expands out the expression $\sp{f,*_1(f,\dots,f)}$ in 
terms of the $f_i$.

\begin{corollary}\label{counting}
Suppose that the assumptions of Lemma \ref{precounting} hold, and that
$|\sp{f-g,\phi}|\leq\eta/k$ for every basic anti-uniform function $\phi \in \Phi_{\mu,1}$.
Then $\E_xg(x)\geq\E_xf(x)-\eta/k$, and 
\begin{equation*}
\bigl|\E_{i_1,\dots,i_k\in\{1,\dots,m\}}
\sp{f_{i_1},\circ_1(f_{i_2},\dots,f_{i_k})}
-\sp{g,*_1(g,g,\dots,g)}\bigr|\leq 4\eta.
\end{equation*}
\end{corollary}

\begin{proof}
The function $\circ_k(1,1,\dots,1)$ is a basic anti-uniform function,
and it takes the constant value 1. Since $\E_xh(x)=\sp{h,1}$ for any
function $h$, this implies the first assertion. 

Now the probability that $i_1,\dots,i_k$ are distinct is again at 
most $\eta/4k$, and if they are not distinct we at least know that
$|\sp{f-g,\circ_j(g,g,\dots,g,f_{i_{j+1}},\dots,f_{i_k})}|\leq 4$. Therefore,
our hypothesis also implies that
\begin{equation*}
\sum_{j=1}^k
|\sp{f-g,\E_{i_{j+1},\dots,i_k}\circ_j(g,g,\dots,g,f_{i_{j+1}},\dots,f_{i_k})}|
\leq k(\eta/k)+4k(\eta/4k)=2\eta.
\end{equation*}
Combining this with Lemma \ref{precounting}, we obtain the result.
\end{proof}

In order to prove analogues of structural results such as the Erd\H{o}s-Simonovits stability theorem
and the hypergraph removal lemma we shall need to preserve slightly more information
when we replace our sparsely supported function $f$ by a densely supported function 
$g$. For example, to 
prove the stability theorem, we proceed as follows. Given a subgraph $A$ of the random 
graph $G_{n,p}$, we create a weighted subgraph $B$ of $K_n$ that contains
the same number of copies of $H$, up to normalization. However, to make the proof work, we also
need the edge density of $B$ within any large vertex set to correspond
to the edge density of $A$ within that set. Suppose that we have this property as well
and that $A$ is $H$-free. Then $B$ has very few copies of $H$. A robust version of the 
stability theorem then tells us that $B$ may be made $(\chi(H) - 1)$-partite by removing 
a small number of edges (or rather a small weight of weighted edges). Let us look at 
the resulting weighted 
graph $B'$. It consists of $\chi(H) - 1$ vertex sets, all of which have zero weight inside. 
Therefore, in $B$, each of these sets had only a small weight to begin with. Since all
``local densities" of $A$ reflect those of $B$, these vertex sets contain only a very
small proportion of the edges in $A$ as well. Removing these edges makes  
$A$ into a $(\chi(H) - 1)$-partite graph and we are done.

How do we ensure that local densities are preserved? All we have to do is enrich our 
set of basic anti-uniform functions by adding an appropriate set of functions that will 
allow us to transfer local densities from the sparse structure to the dense
one. For example, in the case above we need to know that $A$ and $B$ have roughly
the same inner product (when appropriately weighted) with the characteristic function
of the complete graph on any large set $V$ of vertices. We therefore add these characteristic
functions to our stock of basic anti-uniform functions. For other applications, we need to 
maintain more intricate local density conditions. However, as we shall see, as long as 
the corresponding set of additional functions is sufficiently small, this does not pose a problem.

\section{A conditional proof of the main theorems} \label{CondMain}

In this section, we shall collect together the results of 
Sections \ref{transfers} and \ref{clemma} in order to make clear what
is left to prove. We start with a simple and general lemma about
duality in normed spaces.

\begin{lemma}\label{duality}
Let $\Phi$ be a bounded set of real-valued functions defined on a finite set $X$ such that 
the linear span of $\Phi$ is $\R^X$. Let
a norm on $\R^X$ be defined by $\|f\|=\max\{|\sp{f,\phi}|:\phi\in\Phi\}$.
Let $\|.\|^*$ be the dual norm. Then $\|\psi\|^*\leq 1$ if and
only if $\psi$ belongs to the closed convex hull of $\Phi\cup(-\Phi)$.
\end{lemma}

\begin{proof}
If $\psi=\sum_i\lambda_i\phi_i$ with $\phi_i\in\Phi\cup(-\Phi)$,
$\lambda_i\geq 0$ for each $i$ and $\sum_i\lambda_i=1$, and if
$\|f\|\leq 1$, then $|\sp{f,\psi}|\leq\sum_i\lambda_i|\sp{f,\phi_i}|
\leq 1$. The same is then true if $\psi$ belongs to the closure
of the convex hull of $\Phi\cup(-\Phi)$.

If $\psi$ does not belong to this closed convex hull, then by the 
Hahn-Banach theorem there must be a function $f$ such that 
$|\sp{f,\phi}|\leq 1$ for every $\phi\in\Phi$ and $\sp{f,\psi}>1$.
The first condition tells us that $\|f\|\leq 1$, so the second
implies that $\|\psi\|^*>1$.
\end{proof}

So we already know a great deal about functions $\phi$ with bounded dual norm. Recall, however,
that we must consider positive parts of such functions: we would like to show that 
$\sp{\mu - \nu, \phi_+}$ is small whenever $\|\phi\|^*$ is
of reasonable size. We need the following extra lemma to gain some control over these.

\begin{lemma} \label{products}
Let $\Psi$ be a set of functions that take values in $[-2,2]$ and let
$\e>0$.  Then there exist constants $d$ and $M$, depending on $\e$
only, such that for every function $\psi$ in the convex hull of
$\Psi$, there is a function $\omega$ that belongs to $M$ times the
convex hull of all products $\pm\phi_1\dots\phi_j$ with $j\leq d$ and
$\phi_1,\dots,\phi_j\in\Psi$, such that $\|\psi_+-\omega\|_\infty<\e$.
\end{lemma}

\begin{proof}
We start with the well-known fact that continuous functions on closed
bounded intervals can be uniformly approximated by
polynomials. Therefore, if $K(x)$ is the function defined on $[-2,2]$
that takes the value $0$ if $x\leq 0$ and $x$ if $x\geq 0$, then there
is a polynomial $P$ such that $|P(x)-K(x)|\leq\e$ for every
$x\in[-2,2]$. It follows that if $\psi$ is a function that takes
values in $[-2,2]$, then $\|P(\psi)-\psi_+\|_\infty\leq\e$.

Let us apply this observation in the case where $\psi$ is a convex
combination $\sum_i\lambda_i\phi_i$ of functions $\phi_i\in\Psi$. If
$P(t)=\sum_{j=1}^da_jt^j$, then
\begin{equation*}
P(\psi)=\sum_{j=1}^da_j\sum_{i_1,\dots,i_j}
\lambda_{i_1}\dots\lambda_{i_j}\phi_{i_1}\dots\phi_{i_j}.
\end{equation*}
But $\sum_{i_1,\dots,i_j}\lambda_{i_1}\dots\lambda_{i_j}=1$ for every
$j$, so this proves that we can take $M$ to be
$\sum_{j=1}^d|a_j|$. This bound and the degree $d$ depend on $\e$
only, as claimed.
\end{proof}

Similarly, for colouring problems, where we need to deal with the function
$(\max_{1\leq i\leq r}\phi_i)_+$, we have the following lemma. The proof is very similar to 
that of Lemma \ref{products}, though we must replace the function $K(x)$ that 
has to be approximated with the function $K(x_1, \dots, x_r) = \max\{0,x_1,\dots,x_r\}$ 
and apply a multivariate version of the uniform approximation theorem inside
the set $[-2,2]^r$ (though the case we actually need follows easily from the
one-dimensional theorem).

\begin{lemma} \label{colourproducts}
Let $\Psi$ be a set of functions that take values in $[-2,2]$ and let
$\e>0$.  Then there exist constants $d$ and $M$, depending on $\e$
only, such that for every set of functions $\psi_1, \dots, \psi_r$ in the 
convex hull of $\Psi$, there is a function $\omega$ that belongs to $M$ times the
convex hull of all products $\pm\phi_1\dots\phi_j$ with $j\leq d$ and
$\phi_1,\dots,\phi_j\in\Psi$, such that $\|(\max_{1\leq i\leq r}\psi_i)_+-\omega\|_\infty<\e$. \hfill $\square$
\end{lemma}

We shall split up the rest of the proof of our main result as follows. First,
we shall state a set of assumptions about the set $S$ of 
ordered subsets of $X$. Then we shall show how the transference results we are 
aiming for follow from these assumptions. Then over the next few sections
we shall show how to prove these assumptions for a large class of sets 
$S$. 

The reason for doing things this way is twofold. First, it splits the proof into
a deterministic part (the part we do now) and a probabilistic part (verifying
the assumptions). Secondly, it splits the proof into a part that is completely
general (again, the part we do now) and a part that depends more on the
specific set $S$. Having said that, when it comes to verifying the assumptions,
we do not do so for \textit{individual} sets $S$. Rather, we identify two broad
classes of set $S$ that between them cover all the problems that have traditionally
interested people. This second shift, from the general to the particular, 
will not be necessary until Section \ref{ProbII}. For now, the argument remains quite
general. 

Suppose now that $\mu_1, \dots, \mu_m$ are measures on a finite set $X$ and $\mu = m^{-1} (\mu_1 + \dots + \mu_m)$. 
In subsequent sections, we will take $\mu_1, \dots, \mu_m$ to be the associated measures of random
sets $U_1, \dots, U_m$, each distributed as $X_p$, but for now we will continue to work deterministically.
We shall be particularly interested in the following four properties that such a sequence of measures may have.

\begin{properties*}
\

\begin{enumerate}
\item[P0.] $\|\mu_i\|_1 = 1 + o(1)$ for each $i$, where $o(1) \rightarrow 0$ as $|X| \rightarrow \infty$.

\item[P1.] $\|*_1(\mu_{i_2},\dots,\mu_{i_k})-\circ_1(\mu_{i_2},\dots,\mu_{i_k})\|_1\leq\eta$
whenever $i_2,\dots,i_k$ are distinct integers between 1 and $m$. 

\item[P2.] $\|*_j(1,1,\dots,1,\mu_{i_{j+1}},\dots,\mu_{i_k})\|_\infty\leq 2$  
whenever $j \geq 2$ and $i_{j+1},\dots,i_k$ are distinct integers between 1 and $m$.

\item[P3.] $|\sp{\mu-1,\xi}| < \lambda$
whenever $\xi$ is a product of at most $d$ basic anti-uniform functions from $\Phi_{\mu,1}$.
\end{enumerate}
\end{properties*}

In the remainder of this section, we will prove that if $\mu_1, \dots, \mu_m$ satisfy these
four properties, then any robust density theorem or colouring theorem also holds relative to the 
measure $\mu$. To prove this for density statements, we first need a simple lemma showing that
any density theorem implies an equivalent functional formulation.  For convenience, we will assume that each set in 
$S$ consists of distinct elements from $X$.

\begin{lemma} \label{FunctToSet}
Let $k$ be an integer and $\r, \b, \e>0$ be real numbers. Let $X$ be a sufficiently large finite set and 
let $S$ be a collection of ordered subsets of $X$, each of size 
$k$ and with no repeated elements. Suppose that for every subset $B$ of
$X$ of size at least $\rho |X|$ there are at least $\b|S|$ elements $(s_1,\dots,s_k)$ of 
$S$ such that $s_i\in B$ for each $i$. Let $g$ be a function on $X$ such that $0 \leq g
\leq 1$ and $\|g\|_1 \geq \rho + \e$. Then
\[\E_{s \in S} g(s_1) \dots g(s_k) \geq \b-\e.\]
\end{lemma}

\begin{proof}
Let us choose a subset $B$ of $X$ randomly by choosing each $x\in X$ with probability $g(x)$, with the choices independent. The expected number of elements of $B$ is $\sum_xg(x)\geq(\rho+\e)|X|$ and therefore, by applying standard large deviation inequalities, one may show that if $|X|$ is sufficiently large the probability that $|B|<\rho|X|$ is at most $\e$. Therefore, with probability at least $1-\e$ there are at least $\b|S|$ elements $s$ of $S$ such that $s_i\in B$ for every $i$. It follows that the expected number of such sequences is at least $\b|S|(1-\e)\geq(\b-\e)|S|$. But each sequence $s$ has probability $g(s_1)\dots g(s_k)$ of belonging to $B$, so the expected number is also $\sum_{s\in S}g(s_1)\dots g(s_k)$, which proves the lemma.
\end{proof}

Note that the converse to the above result is trivial (and does not need an extra $\e$), since if $B$ is a set of density $\rho$, then the characteristic function of $B$ has $L_1$-norm $\rho$.

We remark here that the condition that no sequence in $S$ should have repeated elements is not a serious restriction. For one thing, all it typically does is rule out degenerate cases (such as arithmetic progressions with common difference zero) that do not interest us. Secondly, these degenerate cases tend to be sufficiently infrequent that including them would have only a very small effect on the constants. The reason we do not allow them is that it makes the proof neater.

With Lemma \ref{FunctToSet} in hand, we are now ready to prove that a transference principle holds for density theorems. 

\begin{theorem} \label{main:density}
Let $k$ be a positive integer and let $\r, \b, \e > 0$ be real numbers. Let $X$ be a finite set and let $S$ be a homogeneous collection of ordered subsets of $X$, each of size $k$ and having no repeated elements. Suppose that for every subset $B$ of $X$ of size at least $\rho |X|$ there are at least $\b|S|$ elements $(s_1,\dots,s_k)$ of $S$ such that $s_i\in B$ for each $i$. Then there are positive constants $\eta$ and $\lambda$ and positive integers $d$ and $m$ with the following property. If $\mu_1, \dots, \mu_m$ are such that P0, P1, P2 and P3 hold for the constants $\eta, \lambda$ and $d$, $\mu=m^{-1}(\mu_1+\dots+\mu_m)$, and $|X|$ is sufficiently large, then $\E_{s \in S} f(s_1)\dots f(s_k) \geq \b - \e$ for every function $f$ such that $0\leq f\leq\mu$ and $\E_xf(x)\geq \rho + \e$.
\end{theorem}

\begin{proof}
To begin, we apply Lemma \ref{FunctToSet} with $\frac{\e}{2}$ to conclude that if $|X|$ is sufficiently large and $g$ is any function on $X$ with $0 \leq g
\leq 1$ and $\|g\|_1 \geq \rho + \frac{\e}{2}$, then
\[\E_{s \in S} g(s_1) \dots g(s_k) \geq \b-\frac{\e}{2}.\]
For each function $h$, let $\|h\|$ be defined to be the maximum of $|\sp{h,\phi}|$ over
all basic anti-uniform functions $\phi\in\Phi_{\mu,1}$. Let $\eta = \frac{\e}{10}$. We claim that, given $f$ with $0 \leq f \leq \mu$, 
there exists a $g$ with $0 \leq g \leq 1$ such that $\|(1+\frac{\e}{4})^{-1}f -g\|\leq\eta/k$.
Equivalently, this shows that $|\sp{(1+\frac{\e}{4})^{-1}f-g,\phi}|\leq\eta/k$ 
for every $\phi\in\Phi_{\mu,1}$. We will prove this claim in a moment. However, let us first note that it is a sufficient
condition to imply that
\[\E_{s \in S} f(s_1)\dots f(s_k) \geq \b - \e\]
whenever $0\leq f\leq\mu$ and $\E_xf(x)\geq \rho + \e$. Let $m = 2k^3/\eta$ and write $(1+\frac{\e}{4})^{-1}f$ as $m^{-1}(f_1+\dots+f_m)$ with $0\leq f_i\leq\mu_i$. Corollary \ref{counting}, together with P1 and P2, then implies that
$\E_xg(x)\geq (1+\frac{\e}{4})^{-1}\E_xf(x)-\eta/k$ and that
\begin{equation*}
\bigl|\E_{i_1,\dots,i_k\in\{1,\dots,m\}} \sp{f_{i_1},\circ_1(f_{i_2},\dots,f_{i_k})} -\sp{g,*_1(g,g,\dots,g)}\bigr|\leq 4\eta.
\end{equation*}
Since $\eta/k<\e/8$, $(1+\frac{\e}{4})^{-1} \geq 1 - \frac{\e}{4}$ and $1 + o(1) \geq \E_x f(x) \geq \rho + \e$,
\[\E_x g(x) \geq \left(1+\frac{\e}{4}\right)^{-1}\E_xf(x)-\eta/k \geq \rho + \e - \frac{\e}{4} - \frac{\e}{8} - o(1) \geq \rho + \frac{\e}{2},\] 
for $|X|$ sufficiently large, so our assumption about $g$ implies that 
$\sp{g,*_1(g,g,\dots,g)}\geq \b - \frac{\e}{2}$. Since in addition $8\eta < \e$, we can deduce the inequality
$\E_{i_1,\dots,i_k}\sp{f_{i_1},\circ_1(f_{i_2},\dots,f_{i_k})}\geq
\b - \e$, which, since the capped convolution is smaller than the standard convolution, implies that
\[\E_{s \in S} f(s_1)\dots f(s_k) = \sp{f,*_1(f,f,\dots,f)} \geq \E_{i_1,\dots,i_k}\sp{f_{i_1},\circ_1(f_{i_2},\dots,f_{i_k})}\geq \b - \e.\]
It remains to prove that for any $f$ with $0 \leq f \leq \mu$, there exists a $g$ with $0 \leq g \leq 1$ such that $\|(1+\frac{\e}{4})^{-1}f -g\|\leq\eta/k$. An application of Lemma 
\ref{transfer:density} tells us that if $\sp{\mu-1,\psi_+}<\frac{\e}{4}$ for every
function $\psi$ with $\|\psi\|^*\leq k\eta^{-1}$, then this will indeed be the case. 
Now let us try to find a sufficient condition for
this. First, if $\|\psi\|^*\leq k\eta^{-1}$, then Lemma \ref{duality}
implies that $\psi$ is contained in $k\eta^{-1}$ times the convex hull
of $\Phi \cup \{-\Phi\}$, where $\Phi$ is the set of all basic anti-uniform functions. Since
functions in $\Phi \cup \{-\Phi\}$ take values in $[-2,2]$, we can apply Lemma
\ref{products} to find constants $d$ and $M$ and a function $\omega$ that can be written as
$M$ times a convex combination of products of at most $d$ functions
from $\Phi \cup \{-\Phi\}$ such that $\|\psi_+-\omega\|_\infty \leq \e/20$. Hence, for such an $\omega$,
\[\sp{\mu - 1, \psi_+ - \omega} \leq \|\mu - 1\|_1 \|\psi_+ - \omega\|_{\infty} \leq (2 + o(1)) \frac{\e}{20} < \frac{\e}{8},\]
for $|X|$ sufficiently large. From this it follows that if $|\sp{\mu-1,\xi}|<\e/8M$ whenever $\xi$ is a product
of at most $d$ functions from $\Phi \cup \{-\Phi\}$, then
\begin{equation*}
\sp{\mu - 1, \psi_+}=\sp{\mu-1,\omega}+\sp{\mu-1,\psi_+-\omega}<\e/8+\e/8=\e/4.
\end{equation*} 
Therefore, applying P3 with $d$ and $\lambda = \e/8 M$ completes the proof.
\end{proof}

To prove a corresponding theorem for colouring problems, we will again need a lemma saying that colouring theorems always have a functional reformulation. 

\begin{lemma} \label{RamseyFtoS}
Let $k, r$ be positive integers and let $\b >0$ be a real number. Let $X$ be a finite set and 
let $S$ be a collection of ordered subsets of $X$, each of size 
$k$ and having no repeated elements. Suppose that for every 
$r$-colouring of $X$ there are at least $\b|S|$ elements $(s_1,\dots,s_k)$ of 
$S$ such that each $s_i$ has the same colour. Let $g_1,\dots,g_r$ be functions from $X$ to $[0,1]$ such that $g_1+\dots+g_r=1$. Then
\[\E_{s \in S}\sum_{i=1}^rg_i(s_1) \dots g_i(s_k) \geq \b.\]
\end{lemma}

\begin{proof}
Define a random $r$-colouring of $X$ as follows. For each $x\in X$, let $x$ have colour $i$ with probability $g_i(x)$, and let the colours be chosen independently. By hypothesis, the number of monochromatic sequences is at least $\b|S|$, regardless of what the colouring is. But the expected number of monochromatic sequences is $\sum_{s \in S}\sum_{i=1}^rg_i(s_1) \dots g_i(s_k)$, so the lemma is proved.
\end{proof}

We actually need a slightly stronger conclusion than the one we have just obtained. However, if $S$ is homogeneous then it is an easy matter to strengthen the above result to what we need.

\begin{lemma} \label{RamseyFtoS2}
Let $k, r$ be positive integers and let $\b>0$ be a real number. Let $X$ be a finite set and 
let $S$ be a homogeneous collection of ordered subsets of $X$, 
each of size $k$ and having no repeated elements. Suppose that for every $r$-colouring of $X$ there are at least $\b|S|$ elements $(s_1,\dots,s_k)$ of 
$S$ such that each $s_i$ has the same colour. Then there exists $\d>0$ with the following
property. If $g_1,\dots,g_r$ are any $r$ functions from $X$ to $[0,1]$ such that $g_1(x)+\dots+g_r(x)\geq1/2$ for at least $(1-\d)|X|$ values of $x$, then
\[\E_{s \in S}\sum_{i=1}^rg_i(s_1) \dots g_i(s_k) \geq 2^{-(k+1)}\b.\]
\end{lemma}

\begin{proof}
Let $Y$ be the set of $x$ such that $g_1(x)+\dots+g_r(x)<1/2$. Then we can find functions $h_1,\dots,h_r$ from $X$ to $[0,1]$ such that $h_1+\dots+h_r=1$ and $h_i(x)\leq 2g_i(x)$ 
for every $x\in X\setminus Y$. By the previous lemma, we know that
\[\E_{s \in S}\sum_{i=1}^rh_i(s_1) \dots h_i(s_k) \geq \b.\]
Let $T$ be the set of sequences $s\in S$ such that $s_i\in Y$ for at least one $i$. Since $S$
is homogeneous, for each $i$ the set of $s$ such that $s_i\in Y$ has size $|S||Y|/|X|\leq\d|S|$.
Therefore, $|T|\leq k\d|S|$. 
It follows that 
\begin{eqnarray*}
\sum_{s \in S}\sum_{i=1}^rg_i(s_1) \dots g_i(s_k) &\geq& 
\sum_{s\in S\setminus T}\sum_{i=1}^rg_i(s_1) \dots g_i(s_k)\\
&\geq& 2^{-k}\sum_{s\in S}\sum_{i=1}^rh_i(s_1) \dots h_i(s_k)-|T|\\
&\geq& (2^{-k}\b-k\d)|S|.
\end{eqnarray*}
Thus, the lemma is proved if we take $\d=2^{-(k+1)}\b/k$.
\end{proof}

We now prove our main transference principle for colouring theorems. The proof is similar to
that of Theorem \ref{main:density} and reduces to the same conditions, but we include the proof
for completeness.

\begin{theorem} \label{main:colouring}
Let $k, r$ be positive integers and $\b > 0$ be a real number. Let $X$ be a finite set and 
let $S$ be a homogeneous collection of ordered subsets of $X$, each of size 
$k$ and with no repeated elements. Suppose that for every 
$r$-colouring of $X$ there are at least $\b|S|$ elements $(s_1,\dots,s_k)$ of 
$S$ such that each $s_i$ has the same colour. Then there are positive constants $\eta$ and 
$\lambda$ and positive integers $d$ and $m$ with the following property.
If $\mu_1, \dots, \mu_m$ are such that P0, P1, P2 and P3 hold for the constants $\eta, \lambda$ and $d$, $\mu = m^{-1}(\mu_1+\dots+\mu_m)$, and $|X|$ is sufficiently large, then $\E_{s \in S} \sum_{i=1}^r f_i(s_1) \dots f_i(s_k) \geq 2^{-(k+2)} \b$ for every sequence of functions $f_1,\dots,f_r$ such that $0\leq f_i\leq\mu$ for each $i$ and $\sum_{i=1}^r f_i = \mu$.
\end{theorem}

\begin{proof}
An application of Lemmas \ref{RamseyFtoS} and \ref{RamseyFtoS2} tells us that there exists $\d>0$ with the following
property. If $g_1,\dots,g_r$ are any $r$ functions from $X$ to $[0,1]$ such that $g_1(x)+\dots+g_r(x)\geq1/2$ for at least $(1-\d)|X|$ values of $x$, then
\[\E_{s \in S}\sum_{i=1}^rg_i(s_1) \dots g_i(s_k) \geq 2^{-(k+1)}\b.\]
Again we define the norm $\|.\|$ by taking $\|h\|$ to be the maximum of $|\sp{h,\phi}|$
over all basic anti-uniform functions $\phi\in\Phi_{\mu,1}$. Let $\eta$ be such that $8 \eta r < \min(\d, 2^{-(k+1)} \b)$. We claim that, given functions $f_1, \dots, f_r$
with $0\leq f_i\leq\mu$ and $\sum_{i=1}^r f_i = \mu$,
there are functions $g_i$ such that $0\leq g_i\leq 1$, $g_1 + \dots + g_r \leq 1$ and 
$\|(1+\frac{\d}{4})^{-1}f_i -g_i\|\leq\eta/k$. Equivalently, this means that 
$|\sp{(1+\frac{\d}{4})^{-1}f_i-g_i,\phi}|\leq\eta/k$ for every $i$ and every $\phi\in\Phi_{\mu,1}$.
We will return to the proof of this statement. For now, let us show that it implies
\[\E_{s \in S} \sum_{i=1}^r f_i(s_1) \dots f_i(s_k) \geq 2^{-(k+2)} \b.\]
Let $m = 2k^3/\eta$ and write $(1+\frac{\d}{4})^{-1} f_i$ as $m^{-1} (f_{i,1} + \cdots + f_{i,m})$ with $0 
\leq f_{i,j} \leq \mu_j$. Corollary \ref{counting}, together with P1 and P2, then 
implies that $\E_x g_i(x)\geq (1+\frac{\d}{4})^{-1}\E_x f_i(x)-\eta/k$ and that
\begin{equation*}
\bigl|\E_{j_1,\dots,j_k\in\{1,\dots,m\}}
\sp{f_{i, j_1},\circ_1(f_{i, j_2},\dots,f_{i, j_k})}
-\sp{g_i,*_1(g_i,g_i,\dots,g_i)}\bigr|\leq 4\eta.
\end{equation*}
Suppose that there were at least $\d |X|$ values of $x$ for which $\sum_{i=1}^r g_i(x) < \frac{1}{2}$.
Then this would imply that 
\[\E_{x \in X} \sum_{i=1}^r g_i(x) < \frac{1}{2} \d + (1 - \d) \leq 1 - \frac{\d}{2}.\]
But $\E_x g_i(x) \geq (1+\frac{\d}{4})^{-1} \E_x f_i(x) - \eta/k$. Therefore,  
adding over all $i$, we have, since $\eta \leq \d/8r$ and $(1+\frac{\d}{4})^{-1} \geq 1 - \frac{\d}{4}$, that
\[\sum_{i=1}^r \E_{x \in X} g_i(x) \geq \left(1+\frac{\d}{4}\right)^{-1}(1+o(1)) - \frac{\eta r}{k} \geq 1 - \frac{\d}{2},\]
for $|X|$ sufficiently large, a contradiction. Our assumption about the $g_i$ therefore implies the inequality 
$\sum_{i=1}^r \sp{g_i,*_1(g_i,g_i,\dots,g_i)}\geq 2^{-(k+1)} \b$. Since $8 r \eta<2^{-(k+1)} \b$, we can deduce the inequality
\[\sum_{i=1}^r \E_{j_1,\dots,j_k}\sp{f_{i, j_1},\circ_1(f_{i, j_2},\dots,f_{i, j_k})}\geq 2^{-(k+2)} \b,\] 
which, since the capped convolution is smaller than the standard convolution, implies that
\[\E_{s \in S} \sum_{i=1}^r f_i(s_1) \dots f_i(s_k) = \sum_{i=1}^r \sp{f_i,*_1(f_i,f_i,\dots,f_i)} \geq \sum_{i=1}^r \E_{j_1,\dots,j_k}\sp{f_{i, j_1},\circ_1(f_{i, j_2},\dots,f_{i, j_k})}\geq 2^{-(k+2)} \b.\]
As in Theorem \ref{main:density}, we have proved our result conditional upon an assumption, this time that for any functions $f_1, \dots, f_r$ with $0\leq f_i\leq\mu$ and $\sum_{i=1}^r f_i = \mu$,
there are functions $g_i$ such that $0\leq g_i\leq 1$, $g_1 + \dots + g_r \leq 1$ and 
$\|(1+\frac{\d}{4})^{-1}f_i -g_i\|\leq\eta/k$.
An application of Lemma \ref{transfer:colouring} tells us that if
$\sp{\mu-1,(\max_{1 \leq i \leq r} \psi_i)_+}<\d/4$ for every collection of functions $\psi_i$ 
with $\|\psi_i\|^*\leq k\eta^{-1}$, then this will indeed be the case. By Lemma \ref{duality}, each $\psi_i$ is contained in $k\eta^{-1}$ times the convex hull of $\Phi \cup \{-\Phi\}$, where $\Phi$ is the set of all basic anti-uniform functions. Since
functions in $\Phi \cup \{-\Phi\}$ take values in $[-2,2]$, we can apply Lemma
\ref{colourproducts} to find constants $d$ and $M$ and a function $\omega$ that can be written as
$M$ times a convex combination of products of at most $d$ functions
from $\Phi \cup \{-\Phi\}$, such that $\|(\max_{1\leq i \leq r}\psi_i)_+-\omega\|_\infty\leq \d/20$.
From this it follows that if $|X|$ is sufficiently large and $|\sp{\mu-1,\xi}|<\d/8M$ whenever $\xi$ is a product
of at most $d$ functions from $\Phi \cup \{-\Phi\}$, then $\sp{\mu - 1, 
(\max_{1 \leq i \leq r} \phi_i)_+} < \d/4$. Therefore, applying P3 with $d$ and $\lambda = \d/8M$ proves the theorem.
\end{proof}

Finally, we would like to talk a little about structure theorems. To motivate the result that we are about to state, let us begin by giving a very brief sketch of how to prove a sparse version of the triangle removal lemma. (For a precise statement, see Conjecture \ref{RelativeRemoval} in the introduction, and the discussion preceding it.)

The dense version of the lemma states that if a dense graph has almost no triangles, then it is possible to remove a small number of edges in order to make it triangle free. To prove this, one first applies Szemer\'edi's regularity lemma \cite{Sz78} to the graph, and then removes all edges from pairs that are sparse or irregular. Because sparse pairs contain few edges, and very few pairs are irregular, not many edges are removed. If a triangle is left in the resulting graph, then each edge of the triangle belongs to a dense regular pair, and then a simple lemma can be used to show that there must be many triangles in the graph. Since we are assuming that there are very few triangles in the graph, this is a contradiction.

The sparse version of the lemma states that essentially the same result holds in a sparse random graph, given natural interpretations of phrases such as ``almost no triangles". If a random graph with $n$ vertices has edge probability $p$, then the expected number of (labelled) triangles is approximately $p^3n^3$, and the expected number of (labelled) edges is $p n^2$. Therefore, the obvious statement to try to prove, given a random graph $G_0$ with edge probability $p$, is this: for every $\delta>0$ there exists $\epsilon>0$ such that if $G$ is any subgraph of $G_0$ that contains at most $\epsilon p^3n^3$ triangles, then it is possible to remove at most $\delta p n^2$ edges from $G$ and end up with no triangles.

How might one prove such a statement? The obvious idea is to use the transference methods explained earlier to find a $[0,1]$-valued function $g$ defined on pairs of vertices (which we can think of as a weighted graph) that has similar triangle-containing behaviour to $G$. For the sake of discussion, let us suppose that $g$ is in fact the characteristic function of a graph and let us call that graph $\Gamma$ (later, in Corollary~\ref{unconditionalsetstructural}, we will show that such a reduction is always possible). 

If $\Gamma$ has similar behaviour to $G$, then $\Gamma$ contains very few triangles, which is promising. So we apply the dense triangle removal lemma in order to get rid of all triangles. But what does that tell us about $G$? The edges we removed from $\Gamma$ did not belong to $G$. And in any case, how do we use an \textit{approximate} statement (that $G$ and $\Gamma$ have similar triangle-containing behaviour) to obtain an \textit{exact} conclusion (that $G$ with a few edges removed has \textit{no} triangles at all)?

The answer is that we removed edges from $\Gamma$ in ``clumps". That is, we took pairs $(U,V)$ of vertex sets (given by cells of the Szemer\'edi partition) and removed all edges linking $U$ to $V$. So the natural way of removing edges from $G$ is to remove the same clumps that we removed from $\Gamma$. After that, the idea is that if $G$ contains a triangle then it belongs to clumps that were not removed, which means that $\Gamma$ must contain a triple of dense regular clumps, and therefore many triangles, which implies that $G$ must also contain many triangles, a contradiction.

For this to work, it is vital that if a clump contains a very small proportion of the edges of $\Gamma$, then it should also contain a very small proportion of the edges of $G$. More generally, the density of $G$ in a set of the form $U\times V$ should be about $p$ times the density of $\Gamma$ in the same set. Thus, we need a result that allows us to approximate a function by one with a similar triangle count, but we also need the new function to have similar densities inside every set of the form $U\times V$ when $U$ and $V$ are reasonably large. 

In the case of hypergraphs, we need a similar but more complicated statement. The precise nature of the complexity is, rather surprisingly, not too important: the main point is that we shall need to approximate a function dominated by a sparse random measure by a bounded function that has a similar simplex count and similar densities inside all the sets from some set system that is not too large.

In order to state the result precisely, we make the following definition. 

\begin{definition}
Suppose that we have a finite set $X$ and suppose that $\Phi_{\mu,1}$ is a collection of basic anti-uniform functions derived from a collection $S$ of ordered subsets of $X$ and a sequence of measures $\mu = (\mu_1, \dots, \mu_m)$. Then, given a collection of subsets $\mathcal{V}$ of $X$, we define the set of basic anti-uniform functions $\Phi_{\mu,1} (\mathcal{V})$ to be $\Phi_{\mu,1} \cup \{\chi_V : V \in \mathcal{V}\}$, where $\chi_V$ is the characteristic function of the set $V$.
\end{definition}

We also need to modify the third of the key properties, so as to take account of the set system~$\mathcal{V}$.\\

\textit{P3\,$'$. $|\sp{\mu-1,\xi}| < \lambda$ whenever $\xi$ is a product of at most $d$ basic anti-uniform functions from~$\Phi_{\mu,1}(\mathcal{V})$.}\\

Our main abstract result regarding the transfer of structural theorems is the following. It says that not only do the functions $f$ and $g$ reflect one another in the sense that they have similar subset counts, but they may be chosen to have similar densities inside all the sets $V$ from a collection $\mathcal{V}$. The proof, which we omit, is essentially the same as that of Theorem \ref{main:density}: the only difference is that the norm is now defined in terms of $\Phi_{\mu,1}(\mathcal{V})$, which gives us the extra information that $|\sp{f,\chi_V}-\sp{g,\chi_V}|\leq\|f-g\|$ for every $V\in\mathcal{V}$ and hence the extra conclusion at the end.

\begin{theorem} \label{main:structure}
Let $k$ be a positive integer and $\e > 0$ a constant. Let $X$ be a finite set, $S$ a homogeneous collection of ordered subsets of $X$, each of size $k$, and $\mathcal{V}$ a collection of subsets of $X$. Then there are positive constants $\eta$ and $\lambda$ and positive integers $d$ and $m$ with the following
property. If $\mu_1, \dots, \mu_m$ are such that P0, P1, P2 and P$\mathit{3'}$ hold for the constants $\eta, \lambda$ and $d$, then, for $|X|$ sufficiently large, the following holds for $\mu = m^{-1}(\mu_1+\dots+\mu_m)$: whenever $0\leq f\leq\mu$, there exists $g$ with $0 \leq g \leq 1$ such that
\[\E_{s \in S} f(s_1)\dots f(s_k) \geq \E_{s \in S} g(s_1) \dots g(s_k) - \e\]
and, for all $V \in \mathcal{V}$,
\[|\E_{x \in V} f(x) - \E_{x \in V} g(x)| \leq \e \frac{|X|}{|V|}.\]
\end{theorem}

We remark that the second part of the conclusion can be rewritten as
\[|\E_x\chi_V(x)f(x) - \E_x\chi_V(x)g(x)|\leq\e,\]
which is precisely the statement that $|\sp{f,\chi_V}-\sp{g,\chi_V}|\leq\e$. 

Note that for P$\mathrm{3'}$ to have a chance of holding, we cannot have too many sets in the collection $\mathcal{V}$. However, we can easily have enough sets for our purposes. For instance, the collection of pairs of vertex sets in a graph with $n$ vertices has size $4^n$; this is far smaller than the number of graphs on $n$ vertices, which is exponential in $n^2$. More generally, an important role in the hypergraph regularity lemma is played by $k$-uniform hypergraphs $H$ formed as follows: take a $(k-1)$-uniform hypergraph $K$ and for each set $E$ of size $k$, put $E$ into $H$ if and only if all its subsets of size $k-1$ belong to $K$. Since there are far fewer $(k-1)$-uniform hypergraphs than there are $k$-uniform hypergraphs, we have no trouble applying our result.

Since our ultimate goal is to prove a probabilistic theorem, the task that remains to us is to prove that certain random sets satisfy P0, P1, P2 and P3 (or P$\mathrm{3'}$) with high probability. That this is so for P0 follows easily from Chernoff's inequality. It remains to consider P1, P2 and P3. 

\section{Small correlation with a fixed function} \label{ProbI}

One of our main aims in this paper is to show that, with high probability, $|\sp{\mu - 1, \xi}| < \l$ for every product $\xi$ of at most $d$ basic anti-uniform functions, when $\mu$ is chosen randomly with suitable density. This is a somewhat complicated statement, since the set of basic anti-uniform functions depends on our random variable $\mu$. In this section we prove a much easier result, which will nevertheless be useful to us later on: we shall show that, for any \emph{fixed} bounded function $\xi$, $|\sp{\mu - 1, \xi}| < \l$ with high probability. 

To prove this, we shall need a standard probabilistic estimate, Bernstein's inequality, which allows one to bound the sum of independent and not necessarily identically distributed random variables.

\begin{lemma} \label{Bernstein}
Let $Y_1, Y_2, \dots, Y_n$ be independent random variables. Suppose that each $Y_i$ lies
in the interval $[0, M]$. Let $S = Y_1 + Y_2 + \cdots + Y_n$. Then
\[\P(|S - \E(S)| \geq t) \leq 2 \exp\left\{\frac{-t^2}{2\left(\sum \V (Y_j) + \frac{Mt}{3}\right)}\right\}.\]
\end{lemma}

We are now ready to prove that $\sp{\mu - 1, \xi}$ is bounded with high probability for any fixed bounded function $\xi$. 

\begin{lemma} \label{correlation}
Let $X$ be a finite set and let $U=X_p$. Let $\mu$ be the associated measure of $U$. Then, for any constants $C$ and $\l$ with $C\geq \l$ and any positive
function $\xi$ with $\|\xi\|_{\infty} \leq C$,
\[\P(|\sp{\mu - 1, \xi}| \geq \l) \leq 2 e^{-\l^2 p |X|/3C^2}.\]
\end{lemma}

\begin{proof}
For each $x\in X$, $\mu(x)\xi(x)$ is a random variable that takes values in the interval $[0,p^{-1}C]$. The expectation of $\mu(x)\xi(x)$ is $\xi(x)$, so the expectation of $\sp{\mu-1,\xi}$ is 0. Also, the variance of $\mu(x)\xi(x)$ is at most $\E(\mu(x)^2\xi(x)^2)$, which is $\xi(x)^2p^{-1}$, which is at most $C^2p^{-1}$. 

Let $S=\sum_x\mu(x)\xi(x)$. Then the probability that $|\sp{\mu-1,\xi}|\geq\lambda$ equals the probability that $|S-\E S|\geq\lambda|X|$. Therefore, by Bernstein's inequality, 
\begin{eqnarray*}
\P(|\sp{\mu-1,\xi}|\geq\lambda) & \leq &
2 \exp\left\{\frac{-(\lambda |X|)^2}{2\left(C^2 p^{-1} |X| + C\lambda p^{-1}|X|/3\right)}\right\}\\
& = & 2 \exp\left\{\frac{-\lambda^2 p|X|}{2\left(C^2 + C\lambda/3\right)}\right\}\\
& \leq & 2 \exp\left\{ -\lambda^2 p|X|/3C^2\right\},
\end{eqnarray*}
where to prove the second inequality we used the assumption that $C\geq\lambda$.
\end{proof}

Before we move on to the next section, it will be helpful to state Chernoff's inequality, the standard estimate for the tails of the binomial distribution. As we have already noted, P0 is a straightforward consequence of this lemma. 

\begin{lemma} \label{Chernoff}
Let $0 < p \leq 1$ and $0 < \d \leq \frac{1}{2}$ be real numbers and $X$ a finite set. Then
\[\P(||X_p| - p|X|| \geq \d p|X|) \leq 2 e^{-\d^2 p |X|/4}.\]
\end{lemma}

\section{The set of basic anti-uniform functions has few extreme points} \label{BAUrestriction}

A slightly inaccurate description of what we are going to do next is that we shall show that if P1 and P2 hold with high probability, then so does P3. In order to understand how and why what we shall actually do differs from this, it is important to understand the difficulty that we now have to overcome. The result of the previous section tells us that for a random measure $\mu$ and any given function $\xi$, $|\sp{\mu - 1, \xi}|$ is bounded with high probability. We now need to show that this is the case for \textit{all} functions $\xi$ that are products of at most $d$ basic anti-uniform functions. As we have already commented, this is tricky, because which functions are basic anti-uniform functions depends on $\mu$. 


To get a clearer notion of the problem, let us look at a subcase of the general fact that we are trying to prove, by thinking how we might try to show that, with high probability, $\sp{\mu_1-1,\circ_1(f_2,\dots,f_k)}$ is small whenever $0\leq f_i\leq\mu_i$ for $i=2,3,\dots,k$. That is, for the time being we shall concentrate on basic anti-uniform functions themselves rather than on products of such functions.

A question that will obviously be important to us is the following: for how many choices of functions $f_2, \dots, f_k$ do we need to establish that $|\sp{\mu_1-1, \circ_1(f_2,\dots,f_k)}|$ is small? At first glance, the answer might seem to be infinitely many, but one quickly realizes that a small uniform perturbation to the functions $f_2,\dots,f_k$ does not make much difference to $\sp{\mu_1-1,\circ_1(f_2,\dots,f_k)}$. So it will be enough to look at some kind of net of the functions. 

However, even this observation is not good enough, since the number of functions in a net will definitely be at least exponentially large in $p |X|$. Although the probability we calculated in the previous section is exponentially small in $p|X|$, the constant is small, whereas the constant involved in the size of a net will not be small. So it looks as though there are too many events to consider.

It is clear that the only way round this problem is to prove that the set of basic anti-uniform functions $\circ_1(f_2,\dots,f_k)$ is somehow smaller than expected. And once one thinks about this for a bit, one realizes that this may well be the case. So far, we have noted that $\circ_1(f_2,\dots,f_k)$ is not much affected by small uniform perturbations to the functions $f_i$. However, an important theme in additive combinatorics is that convolutions tend to be robust under a much larger class of perturbations: roughly speaking, a ``quasirandom'' perturbation of one of the $f_i$ is likely to have little effect on $\circ_1(f_2,\dots,f_k)$.

It is not immediately obvious how to turn this vague idea into a
precise one, so for a moment let us think more abstractly.  We have a
class $\Gamma$ of functions, and a function $\nu$, and we would like
to prove that $\sp{\nu,\phi}$ is small for every $\phi\in\Gamma$. To
do this, we would like to identify a much smaller class of functions
$\Delta$ such that if $\sp{\nu,\psi}$ is small for every
$\psi\in\Delta$ then $\sp{\nu,\phi}$ is small for every
$\phi\in\Gamma$. The following very simple lemma tells us a 
sufficient (and also in fact necessary) condition on $\Delta$ for us to be able to make 
this deduction.

\begin{lemma} \label{convexhull}
Let $\Gamma$ and $\Delta$ be two closed sets of functions from $X$ to
$\R$ and suppose that both are centrally symmetric. Then the following 
two statements are equivalent:

(i) For every function $\nu$, 
$\max\{|\sp{\nu,\phi}|:\phi\in\Gamma\}\leq\max\{|\sp{\nu,\psi}|:\psi\in\Delta\}$.

(ii) $\Gamma$ is contained in the convex hull of $\Delta$.
\end{lemma}

\begin{proof}
The statement we shall use is just the easy direction of this equivalence,
which is that (ii) implies (i). To see this, let $\phi\in\Gamma$. Then we can write
$\phi$ as a convex combination $\sum_i\lambda_i\psi_i$ of elements of $\Delta$,
and that implies that $|\sp{\nu,\phi}|\leq\sum_i\lambda_i|\sp{\nu,\psi_i}|$. If 
$|\sp{\nu,\psi}|\leq t$ for every $\psi\in\Delta$, then this is at most $t$, which proves
(i), since $\nu$ and $\phi$ were arbitrary.

Now let us suppose that $\Gamma$ is not contained in the convex hull of $\Delta$,
and let $\phi$ be an element of $\Gamma$ that does not belong to this convex hull. Then
the Hahn-Banach theorem and the fact that $\Delta$ is closed and centrally symmetric
guarantee the existence of a function $\nu$ such that $\sp{\nu,\phi}>1$, but 
$|\sp{\nu,\psi}|\leq 1$ for every $\psi\in\Delta$, which contradicts (i).
\end{proof}

The reason Lemma \ref{convexhull} is useful is that it gives us a strategy for proving 
that $|\sp{\mu-1,\xi}|$ is small for all products of at most $d$ basic anti-uniform functions:
try to show that these functions belong to the convex 
hull of a much smaller set. In fact, this is not quite what we shall do. Rather,
we shall show that every $\xi$ can be \textit{approximated} by an element of the 
convex hull of a much smaller set. To prepare for the more elaborate statement
we shall use, we need another easy lemma.

The statement of the lemma is not quite what one might expect. The reason for
this is that the simplest notion of approximation, namely uniform approximation,
is too much for us to hope to attain. Instead, we go for a kind of weighted
uniform approximation, where we allow the functions to differ quite a lot, but
only in a few specified places.

\begin{lemma} \label{L1error}
Let $H$ be a non-negative function defined on $X$ such that 
$\|H\|_1\leq\e$ and $\|H\|_\infty\leq R$. Let $U=X_p$ and let $\mu$ be the associated measure of $U$. 
Then, with probability at least $1 - 2 \exp(-\e^2 p |X|/3 R^2)$, we have the
estimate $|\sp{\mu-1,\phi}-\sp{\mu-1,\psi}|\leq 3\e$
for every pair of functions $\phi$ and $\psi$ such that $|\phi-\psi|\leq H$. 
\end{lemma}

\begin{proof}
The fact that $\|H\|_1\leq\e$ implies that $|\sp{1,\phi}-\sp{1,\psi}|\leq\e$ as well.
Also, $|\sp{\mu,\phi-\psi}|\leq\sp{\mu,H}$. Therefore, it remains to estimate the
probability that $\sp{\mu,H}>2\e$. Lemma \ref{correlation} with $\l = \e$ 
and $C = R$ implies that the probability that $\sp{\mu - 1, H} > \e$ is at most 
$2 \exp(- \e^2 p |X|/3R^2)$. Therefore, with probability at least $1 - 
2\exp(-\e^2 p |X|/3 R^2)$, $\sp{\mu, H} \leq 2 \e$. The result follows.
\end{proof}

If we use Lemma \ref{convexhull} and Lemma \ref{L1error} in combination, then
we can show the following result.

\begin{corollary} \label{approxworks}
Let $H$ be a non-negative function defined on $X$ such that 
$\|H\|_1\leq\e$ and $\|H\|_\infty\leq R$. Let $U=X_p$ and let $\mu$ be the associated 
measure of $U$. 
Let $\Gamma$ and $\Delta$ be two sets of functions and suppose that for
every $\phi \in \Gamma$ there exists $\psi$ in the convex hull of $\Delta$ such
that $|\phi-\psi|\leq H$. Then 
\begin{equation*}
\max\{|\sp{\mu - 1,\phi}| : \phi\in\Gamma\}\leq\max\{|\sp{\mu-1,\psi'}|:\psi'\in\Delta\}+3\e
\end{equation*}
with probability at least $1 - 2\exp(-\e^2 p |X|/3 R^2)$.
\end{corollary}

\begin{proof}
By Lemma \ref{L1error}, the probability is at least $1- 2\exp(-\e^2 p |X|/3R^2)$ 
that $|\sp{\mu-1,\phi}-\sp{\mu-1,\psi}|\leq 3\e$ whenever $|\phi-\psi|\leq H$.
By the easy direction of Lemma \ref{convexhull}, 
$|\sp{\mu-1,\psi}|\leq\max\{|\sp{\mu-1,\psi'}|:\psi'\in\Delta\}$ for every $\psi$ in
the convex hull of $\Delta$. This proves the result.
\end{proof}

How do we define an appropriate set of functions $\D$? A simple observation gets us close to the set we need, but we shall need to make a definition before we can explain it. 

\begin{definition}
Let $0<q\leq p\leq 1$, $U=X_p$ and let $V=U_{q/p}$. Let $\mu$ be the associated measure of $U$ and let $\nu$ be the associated measure of $V$ considered as a set distributed as $X_q$. Let $f$ be a function with $0\leq f\leq\mu$. Then the \emph{normalized restriction} $f_\nu$ of $f$ to $V$
is the function defined by taking $f_\nu(x)=(p/q)f(x)$ if $x\in V$ and 0 otherwise.
\end{definition} 

The normalization is chosen to give us the following easy lemma. Note that the expectation below is a ``probabilistic" expectation rather than a mere average over a finite set.

\begin{lemma} \label{averagerestrict}
Let $U=X_p$ be a set with associated measure $\mu$ and let $V=U_{q/p}$ be a random subset of $U$ with associated measure $\nu$. Then, for any function $0\leq f\leq\mu$, $f=\E_Vf_\nu$.
\end{lemma}

\begin{proof}
For each $x\in U$ we have 
\begin{equation*}
\E_Vf_\nu(x)=(p/q)f(x)\P[x\in V]=f(x),
\end{equation*}
and for each $x\notin U$ we have $f(x)=\E_Vf_\nu(x)=0$.
\end{proof}

This lemma expresses $f$ as a convex combination of normalized restrictions,
which is potentially useful to us, since if $q/p$ is a small constant, then
a typical contribution to this convex combination comes from a restriction to 
a set that is quite a lot smaller than $U$. That will allow us to find a small net
for the set of all possible restrictions, whose convex hull can then be used to approximate
the set of all possible functions $f$ with $0\leq f\leq\mu$. 

Furthermore, we can use Lemma \ref{averagerestrict} to write convolutions and products of
convolutions as convex combinations as well, as the next lemma easily implies. The lemma
itself is very easy, so we omit the proof.

\begin{lemma} \label{easyconvex}
For $1\leq i\leq m$, let $f_i$ be a fixed function and let $g_i$ be a random function
such that $f_i=\E g_i$. Let $\kappa(f_1,\dots,f_m)$ be a
multilinear form in the functions $f_1,\dots,f_m$. Then 
\begin{equation*}
\kappa(f_1,\dots,f_m)=\E \kappa(g_1,\dots,g_m).
\end{equation*}
\end{lemma}

The rough idea, and the third main idea of the paper, is to rewrite every product of convolutions as an average of products of basic anti-uniform functions built out of normalized restrictions. Since there are ``fewer'' of these, we will have fewer events that need to hold. We must, however, be careful when we apply this idea. For a start, we cannot afford to let $q$ become too small. If $q$ is too small, then, given associated measures $\nu_2, \dots, \nu_k$ of sets distributed as $X_q$, we can no longer guarantee that the convolutions $*_1(\nu_2, \cdots, \nu_k)$ are sufficiently well-behaved for our purposes. Even when we choose $q$ to be large enough, there will still be certain rogue choices of sets. However, for $q$ sufficiently large, we can take care of these rogue sets by averaging and showing that they make a small contribution. Thus, there is a delicate balance involved: $q$ has to be small enough to give rise to a small class of functions, but large enough for these functions to have the properties required of them.

Another problem arises from the form of the basic anti-uniform functions. Recall that our starting point is a collection of (randomly chosen) sets $U_1,\dots,U_m$ with associated measures $\mu_1,\dots,\mu_m$. The basic anti-uniform functions we want to approximate are of the form $\circ_j(g,\dots,g,f_{j+1},\dots,f_k)$, where $0\leq g\leq 1$ and $0\leq f_h\leq\mu_{i_h}$ for some sequence $(i_{j+1},\dots,i_k)$ of integers between 1 and $m$. Therefore, we must approximate capped convolutions of functions some of which are bounded above by 1 and some of which are bounded above by associated measures of sparse random sets. This creates a difficulty for us. It is still true that if $V=X_q$ is a random set with associated measure $\nu$, then $g = \E_V g_\nu$, but if we exploit that fact directly, then the number of sets $V$ that we have to consider is on the order of $\binom{|X|}{q|X|}$. Since we will take $q/p$ to be a constant, this is much larger than $\exp(c p |X|)$ for any constant $c$ and therefore too large to use in a probabilistic estimate given by a simple union bound. To get round this problem we shall find a much smaller set $\mathcal{V}$ such that $g=\E_{V \in \mathcal{V}}g_\nu$. 

We shall need the following piece of notation to do this. Suppose that the elements of the set $X$ are ordered in some arbitrary way as $x_1, \cdots, x_n$, say. Then, given a subset $V=\{x_{j_1}, \dots, x_{j_l}\}$ of $X$ and an integer $a$ between $0$ and $n-1$, we define the set $V + a$ to be the set formed by translating the indices by $a$. That is, $V + a = \{x_{j_1 + a}, \dots, x_{j_l + a}\}$ where the sums are taken modulo $n$.  (This ``translation" operation has no mathematical significance: it just turns out to be a convenient way of defining a small collection of sets.)

Let us write $\nu+a$ for the characteristic measure of $V+a$. Write $g_{\nu+ a}$ for the function given by $\frac{|X|}{|V|} g(x)$ if $x \in V + a$ and $0$ otherwise. A proof almost identical to that of Lemma \ref{averagerestrict} implies that $g=\E_a g_{\nu+a}$ for any function $g$, and in particular for any function $g$ such that $0 \leq g \leq 1$. From this and Lemma \ref{averagerestrict} itself it follows that if $(W_1,\dots,W_{j-1})$ are any subsets of $X$, and $\omega_i$ is the characteristic measure of $W_i$, then 
\begin{equation*}
*_j(g,\dots,g,f_{j+1},\dots,f_{k})=\E_{a_1,\dots,a_{j-1}}\E_{V_{j+1},\dots,V_k}*_j(g_{\omega_1+a_1},\dots,g_{\omega_{j-1}+a_{j-1}},(f_{j+1})_{\nu_{j+1}},\dots,(f_{k})_{\nu_k}),
\end{equation*}
where for $h>j$ the set $V_h$ is distributed as $(U_{i_h})_{q/p}$. There is one practical caveat, in that this identity holds when $\omega_1, \dots, \omega_{j-1}$ are characteristic measures, but it is more natural for us to deal with associated measures. However, if we assume that $W_1, \dots, W_{j-1}$ were chosen with probability $q$ and $|W_i| = (1+o(1)) q |X|$, then the distinction vanishes and the identity above holds (up to a $o(1)$ term) with associated measures rather than characteristic ones.

This observation is encouraging, because it represents the convolution $*_j(g,\dots,g,f_{j+1},\dots,f_{k})$ as an average of convolutions from a small class of functions. However, it certainly does not solve our problems completely, since we need a statement about \textit{capped} convolutions. Of course, it would be very surprising if it did solve our problems, since so far we have not said anything about the size of $q$ and the sets $W_1,\dots,W_{j-1}$. In order to transfer the trivial observation above from convolutions to capped convolutions, we shall need $q$ to be sufficiently large and the $W_i$ to be ``sufficiently random".

\subsection{Sufficient randomness}

First, let us describe two properties that are closely related to the properties P1 and P2 defined earlier, and discuss how they are related. The properties will apply to a sequence of measures $\nu_1,\dots,\nu_{j-1},\nu_{j+1},\dots,\nu_k$ and parameters $\eta>0$ and $j\leq k$. 

\begin{enumerate}
\item[Q1.]  $\|*_j(\nu_1,\dots,\nu_{j-1},\nu_{j+1},\dots,\nu_k)
-\circ_j(\nu_1,\dots,\nu_{j-1},\nu_{j+1},\dots,\nu_k)\|_1\leq \eta.$
\item[Q2.]  $\|*_j(1,\dots,1,\nu_{j+1},\dots,\nu_k)\|_\infty\leq 2.$
\end{enumerate}

The main difference between these new properties and the properties P1 and P2 is that we are not quantifying over a whole set of sequences. For example, P1 is the property that Q1 holds with $j=1$ for every sequence $(\mu_{i_2},\dots,\mu_{i_k})$ taken from a sequence $(\mu_1,\dots,\mu_m)$. 

A less obvious difference is that, while we are ultimately interested in obtaining properties of the measures $\mu_1,\dots,\mu_m$, we shall deduce these from probabilistic statements about typical sequences of measures $\nu_1,\dots,\nu_k$ chosen binomially with a smaller probability. This will be illustrated by the main result of this section.  

Let us define what we mean by ``sufficiently random" and then show that what we need can be obtained if Q1 holds with sufficiently high probability for suitable $\eta$. The next definition highlights the property that we want to get out of the sufficient randomness: that capped convolutions should be pretty similar to actual convolutions.

The randomness property we need of our sets $W_i$ is roughly speaking that almost all sequences of sets that appear in the averages we consider satisfy Q1 for some small $\eta$. That will allow us to prove a statement about capped convolutions, because almost all the convolutions that appear in the average in the observation above can then be approximated by their capped counterparts. Here is the formal definition. 

\begin{definition}
Let $\eta>0$ be a real number, let $0<q\leq p\leq 1$ and let $W_1, \dots, W_{j-1}$ and $Z_{j+1}, \dots, Z_k$ be subsets of $X$. We say that $W_1, \dots, W_{j-1}, Z_{j+1},\dots,Z_k$ are \emph{sufficiently random} if $|W_h| = (1+o(1))q |X|$ for every $h < j$, $|Z_h| = (1 +o(1)) p|X|$ for every $h > j$ and, if $\omega_1, \dots, \omega_{j-1}$ are the associated measures of $W_1, \dots, W_{j-1}$ defined with weight $q^{-1}$, then the following statement holds. 
\begin{itemize}
\item Let a sequence $(a_1, \dots, a_{j-1}, V_{j+1}, \dots, V_{k})$ be chosen randomly and independently such that $a_h \in \Z_n$ for every $h < j$, $V_{h}=(Z_h)_{q/p}$ for every $h>j$ and, for each $h>j$, let $\nu_{h}$ be the associated measure of $V_{h}$ defined with weight $q^{-1}$. Then the probability that the $(k-1)$-tuple $(\omega_1+a_1,\dots,\omega_{j-1}+a_{j-1},\nu_{j+1},\dots,\nu_{k})$ satisfies Q1 is at least $1 - o(|X|^{-k})$.
\end{itemize}
\end{definition}

Strictly speaking, the definition of sufficient randomness depends on the parameters $\eta$, $p$ and $q$, but these will be clear from the context.

Our next main lemma says that if Q1 holds with sufficiently high probability, then the probability that $W_1, \dots, W_{j-1}, Z_{j+1},\dots,Z_k$ are sufficiently random, if we choose them independently at random in a suitable way, is also close to $1$. 

\begin{lemma} \label{Ass1'}
Suppose that if $\nu_1, \dots, \nu_k$ are the associated measures of sets $V_1, \dots, V_k$, each chosen binomially with probability $q \geq p_0$, then property Q1 holds with probability $1 - o(|X|^{-k})$. For $1\geq p\geq q\geq p_0$, let $W_1,\dots,W_{j-1}$ be independent random subsets of $X$ with each $W_h=X_q$, and let $Z_{j+1},\dots,Z_k$ be independent random sets with each $Z_h=X_p$. Then the probability that $W_1, \dots, W_{j-1}, Z_{j+1},\dots,Z_k$ are sufficiently random is $1-o(1)$.
\end{lemma}

\begin{proof}
Consider the following way of choosing $k-1$ random sets. First we choose $W_1,\dots,W_{j-1}$ and $Z_{j+1},\dots,Z_k$ as in the statement of the lemma. Chernoff's inequality easily implies that with probability $1 - o(|X|^{-k})$ we have $|W_h| = (1+o(1))q |X|$ for all $h < j$ and $|Z_h| = (1 +o(1)) p|X|$ for all $h > j$. Next, we choose $a_1,\dots,a_{j-1},V_{j+1},\dots,V_{k}$ randomly and independently such that $a_h\in\Z_n$ for every $h<j$ and $V_{h}=(Z_h)_{q/p}$ for every $h>j$. Then the sets $W_1+a_1,\dots,W_{j-1}+a_{j-1},V_{j+1},\dots,V_{k}$ are independent random sets, each distributed as $X_q$.  

Let their associated measures be $\omega_1+a_1,\dots,\omega_{j-1}+a_{j-1},\nu_{j+1},\dots,\nu_{k}$. Then, since $q \geq p_0$, our assumption tells us that this sequence of measures satisfies Q1 with probability $1-o(|X|^{-k})$. Therefore, with probability $1-o(1)$, when we choose the $W_i$ and the $Z_i$, the probability that the sequence $\omega_1+a_1,\dots,\omega_{j-1}+a_{j-1},\nu_{j+1},\dots,\nu_{k}$ satisfies Q1, conditional on that choice of the $W_i$ and $Z_i$, is at least $1-o(|X|^{-k})$. Hence, $W_1, \dots, W_{j-1}, Z_{j+1},\dots,Z_k$ are sufficiently random with probability $1-o(1)$, as claimed.
\end{proof}

In particular, if $Z_{j+1},\dots,Z_k$ are binomial random subsets of $X$, each chosen with probability $p$, then, with high probability, there is a choice of sets $W_1, \dots, W_{j-1}$ such that $W_1, \dots, W_{j-1}, Z_{j+1},\dots,Z_k$ are sufficiently random.

\subsection{The proof for basic anti-uniform functions}

Suppose that $Lq = p \geq q \geq p_0$ for some large constant $L$. It will be by choosing this constant $L$ to be large enough that we will make our trick of using normalized restrictions work. Throughout this section, we will assume that $\eta>0$ is some parameter yet to be specified, and $Z_{j+1},\dots,Z_k$ is a sequence of sets, with associated measures $\zeta_{j+1},\dots,\zeta_k$ defined with weight $p^{-1}$, such that
\begin{enumerate} 

\item[(i)]
If $j = 1$, then Q1 holds for the sequence of measures $(\zeta_2,\dots,\zeta_k)$.

\item[(ii)]
If $j > 1$, then Q2 holds for the sequence of measures $(\zeta_{j+1},\dots,\zeta_k)$.

\item[(iii)]
There exist sets $W_1, \dots, W_{j-1}$ such that $W_1, \dots, W_{j-1}, Z_{j+1},\dots,Z_k$ are sufficiently random (with parameters $j$ and $\eta$).

\end{enumerate}


In the remainder of the section, we will show that if conditions (i), (ii) and (iii) hold (with parameters $j$ and $\eta$ for suitable $\eta$), then the set of basic anti-uniform functions defined using $\zeta_{j+1},\dots,\zeta_k$ has a small net. To be more precise, we need some definitions. In what follows, we will write $\omega_{h,a}$ for the associated measure of $W_h+a$ (or, more accurately, the translate by $a$ of the associated measure $\omega_h$ of $W_h$), where again these associated measures are defined with weight $q^{-1}$. We will also write $\omega'_{h,a}$ for the characteristic measure of $W_h + a$.

\begin{definition}
Let $\Phi(\zeta_{j+1},\dots,\zeta_k)$ be the set of functions 
$\circ_j(g,\dots,g,f_{j+1},\dots,f_k)$, where $0\leq g\leq 1$ and $0\leq f_h\leq\zeta_h$
for each $h > j$. Let $\Psi(\zeta_{j+1},\dots,\zeta_k)$ be the set of functions 
$\circ_j(f_1,\dots,f_{j-1},f_{j+1},\dots,f_k)$ such that the constituent functions $f_h$ have 
the following properties. If $h < j$, then $0\leq f_h\leq \omega'_{h,a}$ for some $a$, and if $h > j$ then $0\leq f_h\leq\nu_h$, where $\nu_h$ is the associated measure, defined with weight $q^{-1}$, of some set
$V_h\in (Z_h)_{q/p}$ such that $|V_h|\leq 2q|X|$.
\end{definition}

We shall now show that every function in $\Phi(\zeta_{j+1},\dots,\zeta_k)$
can be approximated by a convex combination of functions in $\Psi(\zeta_{j+1},\dots,\zeta_k)$.
This will be very useful to us, because $\Psi(\zeta_{j+1},\dots,\zeta_k)$ is a much ``smaller"
set than $\Phi(\zeta_{j+1},\dots,\zeta_k)$. However, we need to be rather careful about precisely 
what we mean by ``can be approximated by''.

\begin{lemma} \label{restrictions}
Let $L$ and $\a < 1$ be positive constants. If $\eta$ is sufficiently small (depending on $\a$) and $|X|$ is sufficiently large (depending on $L$ and $\a$), then there is a non-negative function $H$ such that $\|H\|_1\leq\a$, $\|H\|_\infty\leq 2$,
and, for every function $\circ_j(g,\dots,g,f_{j+1},\dots,f_k)\in\Phi(\zeta_{j+1},\dots,\zeta_k)$, there exists a function $\sigma$ in the convex hull of $\Psi(\zeta_{j+1},\dots,\zeta_k)$ such that
\begin{equation*}
0\leq\circ_j(g,\dots,g,f_{j+1},\dots,f_k)-\sigma\leq H.
\end{equation*}
\end{lemma}

\begin{proof}
Let us choose a random function $\psi\in\Psi(\zeta_{j+1},\dots,\zeta_k)$ as follows. Suppose that $W_1, \dots, W_{j-1}$ are the sets given by condition (iii) and that for each $h$ the associated measure of $W_h$ is $\omega_h$ (as in the definition of sufficient randomness). We start by
choosing a random sequence $(\omega'_1,\dots,\omega'_{j-1},\nu_{j+1},\dots,\nu_k)$. Here,
each $\omega'_h$ is chosen uniformly at random from the $|X|$ measures $\omega'_{h,a}$ and
each $\nu_h$ is the associated measure of a set $V_h$, where the sets $V_h$ are independent 
and distributed as $(Z_h)_{q/p}$. We then let $\psi$ be the function
\begin{equation*}
\circ_j(g_{\omega'_1},\dots,g_{\omega'_{j-1}},(f_{j+1})_{\nu_{j+1}},\dots,(f_k)_{\nu_k})
\end{equation*}
if every $V_h$ has size at most $2q|X|$, and the zero function otherwise. Finally, we take
$\sigma$ to be the expectation of $\psi$, which is certainly a convex combination of 
functions in $\Psi(\zeta_{j+1},\dots,\zeta_k)$.

Let us begin with the first inequality. Here we shall prove the slightly stronger result that
the inequality holds even if we take 
$\psi=\circ_j(g_{\omega'_1},\dots,g_{\omega'_{j-1}},(f_{j+1})_{\nu_{j+1}},\dots,(f_k)_{\nu_k})$
for all choices of $V_h$ (rather than setting it to be zero when one of the $V_h$ is too large).

Let $T$ be the function from $\R$ to $\R$ defined by $T(y)=\min\{y,2\}$ and let
$S(y)=y-T(y)=\max\{y-2,0\}$. Then
\begin{eqnarray*}
\circ_j(g,\dots,g,f_{j+1},\dots,f_k)&=&T(*_j(g,\dots,g,f_{j+1},\dots,f_k))\\
&=&T(\E (*_j(g_{\omega'_1},\dots,g_{\omega'_{j-1}},(f_{j+1})_{\nu_{j+1}},\dots,(f_k)_{\nu_k}))),
\end{eqnarray*}
by Lemmas \ref{averagerestrict} and \ref{easyconvex} and the fact that 
$g=\E_a g_{\omega'_{h,a}}$ (the reason we use the characteristic measures $\omega'_{h,a}$ rather than the associated measures $\omega_{h,a}$ here is so that this identity is exactly true rather than merely approximately true with high probability). On the other hand,
\begin{equation*}
\E(\circ_j(g_{\omega'_1},\dots,g_{\omega'_{j-1}},(f_{j+1})_{\nu_{j+1}},\dots,(f_k)_{\nu_k}))
=\E T(*_j(g_{\omega'_1},\dots,g_{\omega'_{j-1}},(f_{j+1})_{\nu_{j+1}},\dots,(f_k)_{\nu_k})).
\end{equation*}
Since $T$ is a concave function, the result follows from Jensen's inequality.

Now let us define $H$ and prove the second inequality. Since capped convolutions are smaller than convolutions and, as above, $*_j(g,\dots,g,f_{j+1},\dots,f_k))
=\E (*_j(g_{\omega'_1},\dots,g_{\omega'_{j-1}},(f_{j+1})_{\nu_{j+1}},\dots,(f_k)_{\nu_k}))$, 
the left-hand side of this inequality is at most
$\E (*_j(g_{\omega'_1},\dots,g_{\omega'_{j-1}},(f_{j+1})_{\nu_{j+1}},\dots,(f_k)_{\nu_k})
-\psi)$, which is $\E\tau$, where
\begin{equation*}
\tau=S(*_j(g_{\omega'_1},\dots,g_{\omega'_{j-1}},(f_{j+1})_{\nu_{j+1}},\dots,(f_k)_{\nu_k}))
\end{equation*}
if every $V_h$ has size at most $2q|X|$, and $*_j(g_{\omega'_1},\dots,g_{\omega'_{j-1}},(f_{j+1})_{\nu_{j+1}},\dots,(f_k)_{\nu_k})$ otherwise.

If every $V_h$ has size at most $2q|X|$, then $\tau\leq S(*_j(\omega'_1,\dots,\omega'_{j-1},
\nu_{j+1},\dots,\nu_k))$, since $S$ is an increasing function. If there is some set $V_h$ which
is too large, then we use the bound 
\[*_j(g_{\omega'_1},\dots,g_{\omega'_{j-1}},(f_{j+1})_{\nu_{j+1}},\dots,(f_k)_{\nu_k}) \leq
|X|^k *_j(1,\dots,1) = |X|^k,\]
which follows since $g_{\omega'_h}\leq\omega'_h \leq \frac{|X|}{|W_h|} \leq |X|$ and $(f_h)_{\nu_h}\leq\nu_h\leq(p/q)\zeta_{h} \leq q^{-1} \leq |X|$ for every $h$. Since, by Chernoff's inequality, the probability that some one of the $V_h$ has size larger than $2q|X|$ is exponentially small in $q |X|$, the contribution of these bad terms is $o(1)$ everywhere.

We also have the trivial bound
\begin{equation*}
\circ_j(g,\dots,g,f_{j+1},\dots,f_k)-\sigma\leq 2.
\end{equation*}
Accordingly, if we set $H=\min\{\eta+\E S(*_j(\omega'_1,\dots,\omega'_{j-1},\nu_{j+1},\dots,\nu_k)),2\}$, 
then, provided $|X|$ is sufficiently large, we have a function that satisfies the second inequality and trivially satisfies the inequality $\|H\|_\infty\leq 2$.

It remains to bound $\|H\|_1$. Let $\eta = \a/4$. The sufficient randomness assumption tells us that when we choose our random sequence $(\omega_1,\dots,\omega_{j-1},\nu_{j+1},\dots,\nu_k)$, the probability that it satisfies Q1 is $1-o(|X|^{-k})$. Since there are at most $|X|^k$ ways of choosing $\omega_1,\dots,\omega_{j-1}$, it follows that with probability $1-o(1)$, every single such choice results in a sequence that satisfies Q1. That is, we have the inequality
\begin{equation*}
\|*_j(\omega_{1,a_1},\dots,\omega_{j-1,a_{j-1}},\nu_{j+1},\dots,\nu_k)-
\circ_j(\omega_{1,a_1},\dots,\omega_{j-1,a_{j-1}},\nu_{j+1},\dots,\nu_k)\|_1
\leq \eta.
\end{equation*}
The condition that $|W_i| = (1+o(1))q |X|$ implies that, for every $1 \leq i \leq k$ and every $a \in |X|$, $\omega_{i,a} = (1+o(1)) \omega'_{i,a}$. Therefore, for $|X|$ sufficiently large,
\begin{equation*}
\|*_j(\omega'_{1,a_1},\dots,\omega'_{j-1,a_{j-1}},\nu_{j+1},\dots,\nu_k)-
\circ_j(\omega'_{1,a_1},\dots,\omega'_{j-1,a_{j-1}},\nu_{j+1},\dots,\nu_k)\|_1
\leq 2\eta,
\end{equation*}
for any $\nu_{j+1},\dots,\nu_k$ such that every choice of $\omega_1,\dots,\omega_{j-1}$ yields a sequence that satisfies Q1.

If $\nu_{j+1},\dots,\nu_k$ are such that there exists $(a_1,\dots,a_{j-1})$ for which
$(\omega_{1,a_1},\dots,\omega_{j-1,a_{j-1}},\nu_{j+1},\dots,\nu_k)$ does not satisfy Q1, then
we use a ``trivial" bound instead. For each fixed choice of $(V_{j+1},\dots,V_k)$
we have
\begin{eqnarray*}
\E\|*_j(\omega'_{1,a_1},\dots,\omega'_{j-1,a_{j-1}},\nu_{j+1},\dots,\nu_k)\|_1
&=&\|*_j(1,\dots,1,\nu_{j+1},\dots,\nu_k)\|_1\\
&\leq&(p/q)^{k-j}\|*_j(1,\dots,1,\zeta_{j+1},\dots,\zeta_k)\|_1,
\end{eqnarray*}
where the expectation here is taken over all sequences $(a_1, \dots, a_{j-1})$.
The inequality follows from the fact that $0\leq\nu_h\leq (p/q)\zeta_h$ for each $i$. The constant $\eta$ is at most 1, so applying assumption (i) if $j = 1$, we find that 
\[\|*_1(\zeta_2,\dots,\zeta_k)\|_1 \leq \|\circ_1(\zeta_2,\dots,\zeta_k)\|_1 + \|*_1(\zeta_2,\dots,\zeta_k) - \circ_1(\zeta_2,\dots,\zeta_k)\|_1 \leq 2 + \eta \leq 3.\]
Similarly, applying assumption (ii) if $j > 1$, we have $\|*_j(1,\dots,1,\zeta_{j+1},\dots,\zeta_k)\|_1 \leq 2$. In either case,
\begin{equation*}
\E\|*_j(\omega'_{1,a_1},\dots,\omega'_{j-1,a_{j-1}},\nu_{j+1},\dots,\nu_k)\|_1\leq 3(p/q)^{k-1}\leq 3 L^k.
\end{equation*}
As $|X|$ tends to infinity, the probability that the first bound does not hold for every
$(a_1,\dots,a_{j-1})$ tends to zero, and the second bound always holds. Therefore, 
if $|X|$ is sufficiently large, it follows that
\begin{eqnarray*}
\|H\|_1&\leq& 2\eta + \E\|*_j(\omega'_{1,a_1},\dots,\omega'_{j-1,a_{j-1}},\nu_{j+1},\dots,\nu_k)-
\circ_j(\omega'_{1,a_1},\dots,\omega'_{j-1,a_{j-1}},\nu_{j+1},\dots,\nu_k)\|_1\\
&\leq& 4 \eta = \a,
\end{eqnarray*}
where the expectation is taken over all sequences containing those $\nu_{j+1}, \dots, \nu_k$ such that, for all choices of $a_1, \dots, a_{j-1}$, $(\omega_{1,a_1},\dots,\omega_{j-1,a_{j-1}},\nu_{j+1},\dots,\nu_k)$ satisfies Q1. The result follows.
\end{proof}

What we have shown is not just that every element of $\Phi(\zeta_{j+1},\dots,\zeta_k)$
can be approximated well in $L_1$ and reasonably well in $L_\infty$ by a
convex combination of elements of $\Psi(\zeta_{j+1},\dots,\zeta_k)$, but rather the stronger
statement that the difference is bounded above by a \textit{fixed} bounded function 
with small $L_1$-norm. This will be important to us later.

The title of this section was ``The set of basic anti-uniform functions has few
extreme points.'' That is an oversimplification: the next result is what we actually 
mean. 

\begin{lemma} \label{presmallnet}
Let $0<\a \leq 1/2k$ and $L \geq 2$ be a positive integer with $p = Lq$. Then, for $|X|$ sufficiently large depending on $L$ and $\a$, the following holds. Let $Z_{j+1},\dots,Z_k$ be subsets of $X$ with associated measures $\zeta_{j+1},\dots,\zeta_k$ defined with weight $p^{-1}$, and suppose that assumptions (i), (ii) and (iii) are satisfied. Then there is a collection $\Psi'=\Psi'(\zeta_{j+1},\dots,\zeta_k)$ of at most $|X|^{k-1}\binom{2 p |X|}{2q|X|}^{k-j}(2/\a)^{(k-1)2q|X|}$ functions that take values in $[0,2]$, and non-negative functions $H$ and $H'$ with $\|H\|_1\leq\a$, $\|H\|_\infty\leq 2$, $\|H'\|_1\leq 3\a(k-1)$ and $\|H'\|_\infty\leq 2$,
such that for every function $\phi=\circ_j(g,\dots,g,f_{j+1},\dots,f_k)$ 
in $\Phi(\zeta_{j+1},\dots,\zeta_k)$ there is a function $\psi$ in the convex hull
of $\Psi'$ with $\psi\leq\phi\leq\psi+H+H'$.
\end{lemma}

\begin{proof} 
Choose a positive integer $t$ such that $\a/2\leq t^{-1}\leq\a$. Then there exists a
non-negative function $g'$ with $0 \leq g' \leq 1$ such that every value taken by $g'$ is a multiple of $t^{-1}$,
and such that $0\leq g'\leq g\leq g'+\a$. Also, for every $h$ and every function $f_h$
such that $0\leq f_h\leq\zeta_h$ there exists a function $f_h'$ with $0 \leq f_h' \leq \zeta_h$ taking values that are multiples of
$t^{-1} \zeta_h$ such that $0\leq f_h'\leq f_h\leq f_h'+\a\zeta_h$. 

We would now like to show, for any such choice of $g$ and $f_{j+1},\dots,f_k$, that
the functions $\phi=\circ_j(g,\dots,g,f_{j+1},\dots,f_k)$ and 
$\phi'=\circ_j(g',\dots,g',f_{j+1}',\dots,f_k')$ are reasonably close. 
We shall consider the two cases $j=1$ and $j>1$ separately. 

If $j=1$ then 
\begin{eqnarray*}
\phi-\phi'
&=&\sum_{h=2}^k(\circ_1(f_2',\dots,f_{h-1}',f_h,\dots,f_k)-\circ_1(f_2',\dots,f_h',f_{h+1},\dots,f_k))\\
&\leq&\sum_{h=2}^k(*_1(f_2',\dots,f_{h-1}',f_h,\dots,f_k)-*_1(f_2',\dots,f_h',f_{h+1},\dots,f_k))\\
&=&\sum_{h=2}^k *_1(f_2',\dots,f_{h-1}',f_h-f_h',f_{h+1},\dots,f_k)\\
&\leq&\sum_{h=2}^k *_1(\zeta_2,\dots,\zeta_{h-1},\a\zeta_h,\zeta_{h+1},\dots,\zeta_k)\\
&=&\a(k-1) (*_1(\zeta_2,\dots,\zeta_k)).\\
\end{eqnarray*}
Since $\eta\leq 1$, assumption (i) implies that 
\begin{equation*}
\|*_1(\zeta_2,\dots,\zeta_k)\|_1\leq 3,
\end{equation*}
so we find that $\circ_1(f_2,\dots,f_k)-\circ_1(f_2',\dots,f_k')$ is bounded above by 
a function $H'$ with $L_1$-norm at most $3\a(k-1)$. It is clearly also bounded
above by 2.

Lemma \ref{restrictions} gives us $H$ with $\|H\|_1\leq\a$,
$\|H\|_\infty\leq 2$, and also $\psi\in\Psi(\zeta_2,\dots,\zeta_k)$ such that
$0\leq\circ_1(f_2',\dots,f_k')-\psi\leq H$. Putting these two facts together implies
the required bounds on $H$ and $H'$ for the case $j=1$.

If $j>1$ then a very similar argument shows that
\begin{eqnarray*}
\phi-\phi'&\leq&\a(k-1)(*_j(1,\dots,1,\zeta_{j+1},\dots,\zeta_k)).
\end{eqnarray*}
By assumption (ii), 
$\|*_j(1,\dots,1,\zeta_{j+1},\dots,\zeta_k)\|_\infty\leq 2$, so in this case we have a function
$H'$ with $L_\infty$-norm at most $2\a(k-1)\leq 2$ and therefore with 
$L_1$-norm at most $2\a(k-1)$.

All that remains is to count the number of functions in $\Psi(\zeta_{j+1},\dots,\zeta_k)$ that are normalized restrictions
of functions of the form $\circ_j(g',\dots,g',f_{j+1}',\dots,f_k')$. It is here that we shall use the
assumption that the sets $Z_h$ each have
cardinality $(1 +o(1)) p |X| \leq 2p|X|$. There are at most $|X|^{j-1}$
choices for the set $(a_1, \dots, a_{j-1})$, and for each $j+1 \leq i \leq k$, because of the 
upper bound on the sizes of the $Z_i$  and $V_i$, there are at most 
$|X| \binom{2 p |X|}{2q|X|}$ choices for the set $V_i$. (Note that since $p = L q \geq 2q$, the largest binomial
coefficient is indeed this one.) Finally, each valuation of each function has at 
most $t \leq 2/\a$ possible results and each of the $k-1$ functions has a domain of size at most $2q|X|$. 
Therefore, the number of normalized restrictions is at most
\[|X|^{k-1} \binom{2 p |X|}{2q|X|}^{k-j}(2/\a)^{(k-1)2q|X|},\]
as required. 
\end{proof}

\subsection{The proof for products of basic anti-uniform functions}

To connect the results of the previous subsection with basic anti-uniform functions, take a sequence $U_1,\dots,U_m$ of subsets of $X$ with associated measures $\mu_1,\dots,\mu_m$. Then, for each $j$ and each sequence $(i_{j+1},\dots,i_k)$ of distinct indices between 1 and $m$, we shall apply the results with $Z_h=U_{i_h}$ and $\zeta_h=\mu_{i_h}$. Then the functions in the set $\Phi(\zeta_{j+1},\dots,\zeta_k)$ are basic anti-uniform functions.

In this section, it will be clear that we are talking about measures $\mu_1,\dots,\mu_m$, and therefore it will be convenient to write $\Phi(i_{j+1},\dots,i_k)$ and $\Psi(i_{j+1},\dots,i_k)$ instead of $\Phi(\mu_{i_{j+1}},\dots,\mu_{i_k})$ and $\Psi(\mu_{i_{j+1}},\dots,\mu_{i_k})$.

Our next task is to generalize Lemma \ref{presmallnet} to a result that applies not just
to basic anti-uniform functions but also to products of at most $d$ such functions. This
is a formal consequence of Lemma \ref{presmallnet}. The exact nature of the bounds 
we obtain for $\|J\|_1$ and $\|J\|_\infty$ is unimportant: what matters is that the first
can be made arbitrarily small and the second is bounded. We need a definition. 

\begin{definition}
If $\phi\in\Phi(i_{j+1},\dots,i_k)$, then define the \emph{profile} of $\phi$ to be the ordered set $(i_{j+1},\dots,i_k)$, and if $\xi$ is a product of $d$ basic anti-uniform functions $\phi_h$, then define the profile of $\xi$ to be the set of all $d$ profiles of the $\phi_h$. We will refer to $d$ as the \emph{size} of the profile.
\end{definition}

\begin{corollary} \label{smallnet}
Let $0<\a \leq 1/2k$ and $L \geq 2$ be a positive integer with $p = Lq$. Then, for $|X|$ sufficiently large depending on $L$ and $\alpha$, the following holds. Suppose that $A$ is a profile of size $d$ and, for every $(i_{j+1}, \dots, i_k)$ in $A$, the sets $U_{i_{j+1}}, \dots, U_{i_k}$ satisfy assumptions (i), (ii) and (iii). Then there is a collection $\Delta=\Delta(A)$ of at most $|X|^{kd}\binom{2 p |X|}{2 q |X|}^{kd}(2/\a)^{2kdq|X|}$
functions that take values in $[0,2^d]$ and a non-negative function 
$J=J(A)$ with $\|J\|_1\leq d\a k 6^{d}$ and $\|J\|_\infty\leq d 6^{d}$,
such that for every function $\xi$ that is a product of basic anti-uniform functions
with profile $A$, there is a function $\psi$ in the convex hull
of $\Delta$ with $\psi\leq\xi\leq\psi+J$.
\end{corollary}

\begin{proof}
Every function $\xi$ with profile $A$ is a product $\phi_1\dots\phi_d$, 
where each $\phi_i$ is a basic anti-uniform function 
with some fixed profile. That is, each $\phi_i$ belongs to a fixed set
of the form $\Phi(i_{j+1},\dots,i_k)$. By Lemma \ref{presmallnet} we 
can find $\psi_i$ such that $\psi_i\leq\phi_i\leq\psi_i+J_i$, where 
$\psi_i$ belongs to the convex hull of a set $\Psi'_i$ of size at
most $|X|^{k}\binom{2 p |X|}{2 q |X|}^{k}(2/\a)^{2kq|X|}$, and $J_i$ is a fixed function 
such that $\|J_i\|_\infty\leq 4$ and $\|J_i\|_1\leq 4 \a k$.

It follows that $\prod_i\psi_i\leq\xi\leq\prod_i(\psi_i+J_i)$. But
\begin{equation*}
\prod_{i=1}^d(\psi_i+J_i)-\prod_{i=1}^d\psi_i\leq\sum_{h=1}^dJ_h\prod_{i\ne h}(\psi_i+J_i).
\end{equation*}
Since each $\psi_i$ has $L_\infty$-norm at most $2$, the latter function has $L_1$-norm 
at most $d\a k 6^{d}$ and $L_\infty$-norm at most $d 6^d$, as claimed.
\end{proof}

We are now ready for the main result of this section. It will be convenient once again to give names to certain assumptions.

\begin{enumerate}
\item[R1($r,j$).]  If $Z_1,\dots,Z_k$ are chosen independently from $X_r$ and their associated measures are $\zeta_1,\dots,\zeta_k$, then $\|*_j(\zeta_1,\dots,\zeta_{j-1},\zeta_{j+1},\dots,\zeta_k) -\circ_j(\zeta_1,\dots,\zeta_{j-1},\zeta_{j+1},\dots,\zeta_k)\|_1\leq \eta$ with probability $1-o(|X|^{-k})$.

\item[R2($r$).]  With the notation as in R1, the probability that $\|*_j(1,\dots,1,\zeta_{j+1},\dots,\zeta_k)\|_\infty\leq 2$ for every $j\geq 2$ is $1-o(1)$.
\end{enumerate}

Note that R1($r,j$) is saying that Q1 holds with high probability, and R2($r$) is saying that Q2 holds with high probability for every $j$ (when the $\nu_i$ are the associated measures of random sets from $X_r$). 

\begin{lemma} \label{allprofile}
For any positive constant $\lambda$ and positive integer $d$, there exist $\eta, m$ and $L$ such that the following holds. Let $0<p_0\leq 1/L$ and suppose that assumptions R1($r,j$)  and R2($r$) hold for every $j$ and for every $r\geq p_0$. Let $p\geq Lp_0$, let $U_1,\dots,U_m$ be chosen independently from $X_p$, and let $\mu_1,\dots,\mu_m$ be their associated measures. Then, with probability $1 - o(1)$, they satisfy property P3. That is, setting $\mu = m^{-1} (\mu_1 + \dots + \mu_m)$, $|\sp{\mu-1,\xi}| < \lambda$ whenever $\xi$ is a product of at most $d$ basic anti-uniform functions from $\Phi_{\mu,1}$.
\end{lemma}

\begin{proof}
Let $A$ be a profile and suppose that $i$ is not involved in $A$.
Let $\Gamma=\Gamma(A)$ be the set of all products of at most $d$ basic anti-uniform functions with profile $A$. Let $q=p/L$ for a constant $L$ yet to be determined. By assumption R1($q,j$), Lemma~\ref{Ass1'} implies that for every sequence  $(i_{j+1}, \dots, i_k)$, the probability that assumption (iii) holds with parameters $p$ and $q$ for the sets $U_{i_{j+1}},\dots,U_{i_k}$ is $1-o(1)$. Therefore, the probability that it holds for all $j$ and all sequences $(i_{j+1}, \dots, i_k)$ in the profile $A$ is also $1-o(1)$. By assumptions R1($p,1$) and R2($p$), we also know that assumptions (i) and (ii) hold with probability $1-o(1)$ for any given $(i_{j+1}, \dots, i_k)$ and, therefore, for all $(i_{j+1}, \dots, i_k)$ in the profile $A$. 

We may therefore apply Corollary \ref{smallnet} to conclude that there exists a set $\Delta = \Delta(A)$ of at most $|X|^{kd}\binom{2 p |X|}{2 q |X|}^{kd}(2/\a)^{2kdq|X|}$ functions such that, for every function $\xi \in \Gamma$, there exists $\psi$ in $\Delta$ with $|\xi - \psi| \leq H$, where $\|H\|_1 \leq d \a k 6^d$ and $\|H\|_{\infty} \leq d 6^d$. If we let $\a = \l/12 k d 6^d$, Corollary \ref{approxworks} implies that, with probability $1 - o(1)$,
\begin{equation*}
\max\{|\sp{\mu_i-1,\phi}|:\phi\in\Gamma\}\leq\max\{|\sp{\mu_i-1,\psi'}|:\psi'\in\Delta\}+\frac{\l}{4}.
\end{equation*}
Note that this step depends critically on the fact that $\mu_i$ is entirely independent of the
set $\Delta(A)$. It was for this purpose that we chose $m$ random sets $U_1,\dots,U_m$ 
rather than one single random set $U$. This observation is also important in the next step, which is to prove
that $\max\{|\sp{\mu_i-1,\psi'}|:\psi'\in\Delta\} \leq \l/4$ with probability $1-o(1)$. 

By Lemma \ref{correlation}, since $\|\psi'\|_{\infty} \leq 2^d$ for all $\psi' \in \Delta$, the probability that $|\sp{\mu_i - 1, \psi'}| > \l/4$ is at most $2 \exp(-\l^2 p |X|/2^{2d + 10})$ for any given $\psi'$. Since $p = Lq$, we may estimate the number of elements in $\D(A)$ as follows.
\begin{eqnarray*}
|X|^{kd}\binom{2 p |X|}{2 q |X|}^{kd}(2/\a)^{2kdq|X|} & \leq & |X|^{kd} \left(\frac{6 p|X|}{2 q|X|}\right)^{2kdq|X|} \left(\frac{24 k d 6^d}{\lambda}\right)^{2kdq|X|}\\
& \leq & |X|^{kd} (3 L)^{2 k d p |X|/L} \left(\frac{24 k d 6^d}{\lambda}\right)^{2kdp|X|/L}\\
& = & |X|^{kd} \left(\left(\frac{72 L k d 6^d}{\lambda}\right)^{2kd}\right)^{p|X|/L}.
\end{eqnarray*}
If we choose $L$ sufficiently large (depending on $k, d$ and $\l$), then we can arrange for the sum of the probabilities, which is at most $2\exp(-\l^2 p |X|/2^{2d + 10}) |\D(A)|$, to be $o(1)$. 

We are almost done. We now wish to prove a result about $\mu=m^{-1}(\mu_1+\dots+\mu_m)$. Applying our result so far to all profiles simultaneously, we find that with probability $1-o(1)$, $|\sp{\mu_i - 1, \xi}| \leq \l/2$ for every $\mu_i$ and $\xi$ such that $i$ is not involved in the profile of $\xi$. Fix a particular $\xi_0$. If we choose $i$ at random, the probability that it is involved in the profile of $\xi_0$ is at most $(k-1)d/m$. Furthermore, for any $i$, we have the trivial bound $|\sp{\mu_i - 1, \xi_0}| \leq 2^{d+2}$, since $\|\xi_0\|_{\infty} \leq 2^{d}$ and, for $|X|$ sufficiently large, $\|\mu_i - 1\|_1 \leq 3$.  Therefore,
\[|\sp{\mu-1, \xi_0}| \leq \E_i|\sp{\mu_i - 1, \xi_0}| \leq \frac{(k-1)d}{m} 2^{d+2} + \frac{\l}{2} \leq \l,\]
provided $m \geq k d 2^{d+3}/\l$. The result follows. 
\end{proof}

\subsection{Obtaining P3$'$ as well}

It is possible to add a fixed set of bounded functions $\mathcal{F}$ to the collection of basic anti-uniform functions, provided only that this set has size smaller than $2^{p |X|/L_0}$, where $L_0$ is again some constant depending only on $k$, $\lambda$ and $d$, and the above proof continues to work. Indeed, adding such a collection can increase the size of the set of products of basic anti-uniform functions by a factor of at most $2^{d p|X|/L_0}$. Therefore, when we come to the final line of the penultimate paragraph of the proof of the previous lemma, provided $L_0$ and $L$ have been chosen small enough, the probability that the random measure $\mu_i$ correlates with any given function is still small enough to guarantee that with high probability $\max\{|\sp{\mu_i-1,\psi'}|:\psi'\in\Gamma'\} \leq \l/4$, where $\Gamma'$ is the set of functions formed from products of at most $d$ characteristic functions from $\mathcal{F}$ and basic anti-uniform functions whose profile does not involve $\mu_i$. The remainder of the proof is the same, in that we add over all profiles and rule out the set of small exceptions where the set $U_i$ is involved in the profile of $\xi$.

Later, when we come to apply this observation, $\mathcal{F}$ will be a collection of characteristic functions. For example, to prove a stability version of Tur\'an's theorem, the set $\mathcal{F}$ will be the collection of characteristic measures of vertex subsets of $\{1, \dots, n\}$. This has size $2^n$. Therefore, provided $p \geq C n^{-1}$, for $C$ sufficiently large, we will have control over local densities.  

\section{Probabilistic estimates I: tail estimates} \label{ProbII}

In this section, we shall focus on showing that property P2 holds with high probability. That is, we shall show that under suitable conditions, with high probability $\|*_j(1,1,\dots,1,\mu_{i_{j+1}},\dots,\mu_{i_k})\|_{\infty} \leq 2$ for every $j\geq 2$ and every sequence  $i_{j+1},\dots,i_k$ of distinct integers between 1 and $m$. It will be helpful for the next section if we actually prove the following very slightly more general statement. For every $1 \leq j \leq k$, every collection of measures $\nu_1, \dots, \nu_k$ such that at least one of the measures other than $\nu_j$ is the constant measure 1 and the rest are distinct measures of the form $\mu_{i_j}$ has the property that $\|*_j(\nu_1,\dots,\nu_k)\|_{\infty} \leq \frac{3}{2}$. 

Up to now, our argument has been general. Unfortunately, we must now be more specific about the kind of sets that we are dealing with. We shall split into two cases. First, we shall look at systems $S$ with the following property.

\begin{definition} A system $S$ of ordered sequences of length $k$ in a set $X$ \emph{has two degrees of freedom} if, whenever $s$ and $t$ are two elements of $S$ and there exist $i\ne j$ such that $s_i=t_i$ and $s_j=t_j$, we have $s=t$. 
\end{definition}

\noindent This includes the case when $S$ is the set of arithmetic progressions in $\mathbb{Z}_n$, and higher-dimensional generalizations concerning homothetic copies of a single set. 

After that, we will look at graphs and hypergraphs. In this case, the required estimates are much more difficult. Thankfully, most of the hard work has already been done for us by Janson, Ruci\'nski and, in one paper, Oleszkiewicz \cite{J90, JOR04, JR02, JR04, JR09} (see also the paper of Vu, \cite{Vu01}). We shall return to these estimates later.

\subsection{The proof for systems with two degrees of freedom}

Let $U_1,\dots,U_m$ be independent random sets chosen binomially and let their associated measures be $\mu_1,\dots,\mu_m$. We are interested in quantities of the form $*_j(\nu_1,\dots,\nu_k)(x)$, where each $\nu_i$ (with $i\ne j$) is equal to either the constant function 1 or to one of the measures $\mu_r$. We also insist that no two of the $\nu_i$ are equal to the same $\mu_r$ and that at least one of the $\nu_i$ is the constant function. 

Suppose that the set of $i$ such that $\nu_i$ is one of the $\mu_r$ is $\{a_1,\dots,a_l\}$ and that $\nu_{a_h}=\mu_{b_h}$ for $h=1,2,\dots,l$. Then we can interpret $*_j(\nu_1,\dots,\nu_k)(x)$ as follows. Recall that $S_j(x)$ is the set of all $s=(s_1,\dots,s_k)\in S$ such that $s_j=x$. Then $*_j(\nu_1,\dots,\nu_k)(x)$ is equal to $p^{-l}$ times the proportion of $s\in S_j(x)$ such that $s_{a_h}\in U_{b_h}$ for every $h=1,\dots,l$. This is because $\nu_{a_h}(s_{a_h})=p^{-1}$ if $s_{a_h}\in U_{b_h}$ and 0 otherwise.

Now let us regard sequences $s\in S$ as fixed and $U_1,\dots,U_m$ as random variables. For each $s$, let $E(s)$ be the event that $s_{a_h}\in U_{b_h}$ for every $h=1,\dots,l$ (so $E(s)$ is an event that depends on $U_1,\dots,U_m$). We claim that if $s$ and $t$ are distinct sequences in $S_j(x)$, then $E(s)$ and $E(t)$ are independent. The reason for this is that we know that $s_j=t_j$, and our assumption that $S$ has two degrees of freedom therefore implies that there is no other $i$ such that $s_i=t_i$. It follows that the events $s_{a_h}\in U_{b_h}$ and $t_{a_h}\in U_{b_h}$ are independent (since the sets $U_i$ are chosen binomially) and hence that $E(s)$ and $E(t)$ are independent (since the sets $U_{b_1},\dots,U_{b_l}$ are independent).

\begin{lemma} \label{LInfInd}
Let $X$ be a finite set, let $S$ be a homogeneous collection of ordered subsets of $X$, each of size $k$, and suppose that $S$ has two degrees of freedom. Let $U_1,\dots,U_k$ be random subsets of $X$ with associated measures $\mu_1,\dots,\mu_k$, each chosen binomially with probability $p$. Let $1\leq j\leq k$ and let $L$ be a subset of $\{1,2,\dots,k\}\setminus\{j\}$ of cardinality $l < k - 1$. For each $i\leq k$, let $\nu_i=\mu_i$ if $i\in L$ and 1 otherwise. Let $x\in\Z_n$. Then the probability that $*_j(\nu_1,\dots,\nu_{j-1},\nu_{j+1},\dots,\nu_k)(x) \leq \frac{3}{2}$
is at least $1 - 2\exp(-p^l |S_j(x)|/16)$.
\end{lemma}

\begin{proof}
Let $\chi_i$ be the characteristic function of $U_i$. Suppose that $L = \{a_1, \dots, a_l\}$. Then 
\[*_j(\nu_1,\dots,\nu_{j-1},\nu_{j+1},\dots,\nu_k)(x) = p^{-l} \E_{s \in S_j(x)} \prod_{i \in L} \chi_i (s_i).\]
Now $\prod_{i \in L} \chi_i (s_i)$ is the characteristic function of the event $E(s)$ mentioned just before the statement of this lemma, in the case when $b_h=a_h$ for every $h$. As we have discussed, these events are independent. Moreover, they each have probability $p^l$. Therefore, $\E_{s \in S_j(x)} \prod_{i \in L} \chi_i (s_i)$ is an average of $|S_j(x)|$ independent Bernoulli random variables of probability $p^l$.

By Chernoff's inequality, Lemma \ref{Chernoff}, $\sum_{s \in S_j(x)} \prod_{i \in L} \chi_i (s_i) \leq \frac{3}{2} p^l |S_j(x)|$ with probability at least $1 - 2 \exp(-p^l |S_j(x)|/16)$. Therefore, $\E_{s \in S_j(x)} \prod_{i \in L} \chi_i (s_i) \leq \frac{3}{2} p^l$ with the same probability. This proves the result.
\end{proof}

It is perhaps not immediately obvious how the bound for the probability in the last lemma relates to sharp values for $p$ in applications. To get a feel for this, consider the case when $S$ is the set of $k$-term arithmetic progressions in $\Z_n$. Then $|S_j(x)|=n$ for every $x$ and $j$ as $|X|=n$. We want to be able to take $p$ to be around $n^{-1/(k-1)}$. With this value, $\exp(-p^l|S_j(x)|/16)$, takes the form $\exp(-cn^{1-l/(k-1)})$. In the worst case, when $l=k-2$, this works out to be $\exp(-cn^{1/(k-1)})$, which drops off faster than any power of $n$. If we took $l=k-1$ then we would no longer have an interesting statement: that is why convolutions where every $\nu_i$ is equal to some $\mu_j$ must be treated in a different way.

\subsection{The proof for strictly balanced graphs and hypergraphs}

We now turn to the more difficult case of finding copies of a fixed balanced graph or hypergraph. Again, we are trying to show that $*_j(\nu_1,\dots,\nu_k)(x)$ is reasonably close to 1 with very high probability, but now this quantity is a normalized count of certain graphs or hypergraphs. Normally when one has a large deviation inequality, one expects the probability of large deviations to be exponentially small in the expectation. In the graph case a theorem of roughly this variety may be proved for the lower tail by using Janson's inequality \cite{J90}, but the behaviour of the upper tail is much more complex. The best that can be achieved is a fixed power of the expectation. The result that we shall use in this case is due to Janson and Ruci\'nski \cite{JR04}. Before we state it, we need some preliminary discussion.

To begin with, let us be precise about what we are taking as $X$ and what we are taking as $S$. We are counting copies of a fixed labelled $r$-uniform hypergraph $H$. Let $H$ have vertex set $V$ of size $m$ and (labelled) edge set $(e_1,\dots,e_r)$. (That is, each $e_i$ is a subset of $V$ of size $r$ and we choose some arbitrary ordering.) Let $W$ be a set of size $n$ (which we think of as large) and let $X=W^{(r)}$, the set of all subsets of $W$ of order $r$.

Given any injection $\phi:V\to W$ we can form a sequence $(s_1,\dots,s_k)$ of subsets of $W$ by setting $s_i=\phi(e_i)$. We let $S$ be the set of all sequences that arise in this way. The elements of $S$ are copies of $H$ with correspondingly labelled edges. 

If we fix an edge $e\in X$ and an index $j$, then $S_j(e)$ is the set of all sequences $(s_1,\dots,s_k)$ in $S$ such that $s_j=e$. To obtain such a sequence, one must take a bijection from $e_j$ (which is a subset of $V$ of order $r$) to $e$ (which is a subset of $W$ of order $r$) and extend it to an injection $\phi$ from $V$ to $W$. One then sets $s_i=\phi(e_i)$ for each $i$. 

Now let $U_1,\dots,U_m$ be independent random subsets of $X$, chosen binomially with probability $p$, and let their associated measures be $\mu_1,\dots,\mu_m$. Suppose once again that $\nu_1,\dots,\nu_k$ are measures, some of which are constant and some of which are equal to distinct $\mu_i$. Suppose that the non-trivial measures, not including $\nu_j$ if it is non-trivial, are $\nu_{a_1}, \dots, \nu_{a_l}$, and suppose that $\nu_{a_i}=\mu_{b_i}$ for $i=1,2,\dots,l$. Then the value $*_j(\nu_1,\dots,\nu_k)(e)$ of the $j$th convolution at $e$ is equal to 
\[\E_{s \in S_j(e)} \prod_{1 \leq i \leq l} \mu_{b_i} (s_{a_i}).\]
This is $p^{-l} |S_j(e)|^{-1}$ times the number of sequences $(s_1,\dots,s_k)\in S$ such that $s_j=e$ and $s_{a_i}\in U_{b_i}$ for every $1\leq i\leq l$. If we define $H'$ to be the subhypergraph of $H$ that consists of the edges $e_{a_1},\dots,e_{a_l},$ then each such sequence is a so-called $e_j$-\textit{rooted} copy of $H'$ in $(e,X)$. That is, it is a copy of $H'$ where we insist that the vertices in $e_j$ map bijectively to the vertices in $e$. We are interested in the number of rooted copies such that the edges fall into certain sparse random sets. This is not an easy calculation, but it has been done for us by Janson and Ruci\'nski. In order to state the result we shall need, let us define formally the random variable that we wish not to deviate much from its mean.

\begin{notation*}
Let $K$ be a labelled $r$-uniform hypergraph and $f$ an edge in $K$. Let $l$ be the number of edges in $K\char92\{f\}$ and let $U_1, \dots, U_l$ be random binomial subhypergraphs of the complete $r$-uniform hypergraph $K_n^{(r)}$ on $n$ vertices, each edge being chosen with probability $p$, with characteristic functions $\chi_1, \dots, \chi_l$. Let $S_{f}$ be the set consisting of all labelled ordered copies of $K\char92\{f\}$ in $K_n^{(r)}$ that are $f$-rooted at a given edge $e$. Then the random variable $Y_K^{f}$ is given by
\[\sum_{s \in S_{f}} \prod_{1 \leq i \leq l} \chi_i (s_i).\]
\end{notation*}

Strictly speaking $Y_K^{f}$ depends on $e$ as well, but we omit this from the notation because it makes no difference to the probabilities which edge $e$ we choose. (So we could, for example, state the result for $e=\{1,2,\dots,r\}$ and deduce it for all other $e$.)

The number of injections $\phi$ that extend a bijection from $f$ to $e$ is $r!(n-r)(n-r-1)\dots(n-v_K+1)$, and for each one the probability that $s_i\in U_i$ for every $i$ is $p^l=p^{e_K-1}$, so the expectation $\E Y_K^{f}$ is 
\[p^{e_K-1}r!(n-r)(n-r-1)\dots(n-v_K+1).\]
The precise details will not matter to us much, but note that the order of magnitude is $p^{e_K-1}n^{v_K-r}$.

We are now ready to state the result of Janson and Ruci\'nski. It is actually a very special case of a much more general result (Corollary 4.1 from \cite{JR04}). To explain the general statement would lead us too far astray so we restrict ourselves to stating the required corollary.

\begin{lemma} \label{JRHyper}
Let $K$ be a labelled $r$-uniform hypergraph and $f$ a fixed edge. Then there exists a constant $c$ such that the random variable
$Y_K^{f}$ satisfies
\[\P\left(Y_K^{f} \geq \frac{3}{2} \E Y_K^{f}\right) \leq 2 n^{v_K} \exp\left(-c
\min_{L \subseteq K} (\E Y_L^{f})^{1/v_L}\right).\]
\end{lemma}

A better, indeed almost sharp, result has recently been proved by Janson and Ruci\'nski \cite{JR09}. Unfortunately, though the result almost certainly extends to hypergraphs, it is stated by these authors only for graphs. However, the previous result is more than sufficient for our current purposes. 

We are now ready to show that if $X = K_n^{(r)}$, $S$ is the collection of labelled copies of a strictly balanced hypergraph $H$ in $X$ and $p \geq n^{-1/m_r(H)}$, then P2 holds with high probability. The proof is essentially the same as it was for systems with two degrees of freedom, except that we have to use the results of Janson and Ruci\'nski instead of Chernoff's inequality. Recall that an $r$-uniform hypergraph $H$ is strictly $r$-balanced if $\frac{e_H - 1}{v_H - r} > \frac{e_K - 1}{v_K - r}$ for every proper subhypergraph $K$ of $H$.

\begin{lemma} \label{LInfHyper}
Let $H$ be a strictly $r$-balanced $r$-uniform hypergraph with $k$ edges. Let $X = K_n^{(r)}$ and let $S$ be the collection of 
labelled ordered copies of $H$ in $X$. Let $U_1,\dots,U_k$ be random subsets of $X$, each chosen binomially with probability $p$, and let their characteristic measures be $\mu_1,\dots,\mu_k$. Let $1\leq j\leq k$ and let $L$ be a subset of $\{1,2,\dots,k\}\setminus\{j\}$ of cardinality $l < k - 1$. For each $i\leq k$, let $\nu_i=\mu_i$ if $i\in L$ and 1 otherwise. Let $e\in X$. Then for $p \geq n^{-1/m_r(H)}$ there exist positive constants $a$ and $A$ such that the probability that $*_j(\nu_1,\dots,\nu_{j-1},\nu_{j+1},\dots,\nu_k)(e) \leq \frac{3}{2}$ is at least $1 - 2 n^{v_H} e^{-A n^a}$.
\end{lemma}

\begin{proof}
Let $\chi_i$ be the characteristic function of $U_i$. Then 
\[*_j(\nu_1,\dots,\nu_{j-1},\nu_{j+1},\dots,\nu_k)(e) = p^{-l} \E_{s \in S_j(e)} \prod_{i \in L} \chi_i (s_i).\]
The sum $\sum_{s \in S_j(e)} \prod_{i \in L} \chi_i (s_i)$ counts the number of rooted copies of some proper subhypergraph $K$ of $H$. By Lemma \ref{JRHyper}, the probability that $\sum_{s \in S_j(e)} \prod_{i \in L} \chi_i (s_i) \geq \frac{3}{2} p^l |S_j(e)|$ is at most 
\[2 n^{v_K} \exp\left(-c \min_{J \subseteq K} (\E Y_J^{e_j})^{1/v_J}\right) = 2 n^{v_K} \exp\left(-c'
\min_{J \subseteq K} (n^{v_J - r} p^{e_J - 1})^{1/v_J}\right).\]
Since $H$ is strictly $r$-balanced, we know that $\frac{e_H - 1}{v_H - r} > \frac{e_J - 1}{v_J - r}$ for every $J\subseteq K$. Therefore, there is a positive constant $a'$ such that if $p \geq n^{-1/m_r(H)}$, then for each $J\subseteq K$ we have the inequality
\[n^{v_J-r} p^{e_J - 1} \geq n^{v_J-r}n^{-\left(\frac{v_H - r}{e_H - 1}\right) (e_J - 1)}\geq
\left(n^{1 -\left(\frac{v_H - r}{e_H - 1}\right) \left(\frac{e_J - 1}{v_J - r}\right)}\right)^{v_J - r} \geq n^{a'}.\]
Therefore, 
\[\min_{J \subseteq K} (n^{v_J - r} p^{e_J - 1})^{1/v_J} \geq n^{a},\]
for some $a$, and hence the probability that $\sum_{s \in S_j(e)} \prod_{i \in L} \chi_i (s_i) \geq \frac{3}{2} p^l |S_j(e)|$ is at most $2 n^{v_H} e^{-A n^{a}}$ for some positive constants $A$ and $a$. The lemma follows.
\end{proof}

\section{Probabilistic estimates II: bounding $L_1$-differences} \label{ProbIII}

Our one remaining task is to show that property P1 holds with sufficiently high probability. In other words, we must show that if $U_1,\dots,U_m$ are subsets of $X$ chosen binomially with suitable probability $p$, and if their associated measures are $\mu_1,\dots,\mu_m$, then with high probability 
\begin{equation*}
\|*_j(\mu_{i_1},\dots,\mu_{i_{j-1}},\mu_{i_{j+1}},\dots,\mu_{i_k})-\circ_j(\mu_{i_1},\dots,\mu_{i_{j-1}},\mu_{i_{j+1}},\dots,\mu_{i_k})\|_1\leq\eta
\end{equation*}
whenever $j$ is an integer between 1 and $k$ and $i_1,\dots,i_{j-1},i_{j+1},\dots,i_k$ are distinct integers between 1 and $m$. Of course, if we can prove this for one choice of $j$ and $i_1,\dots,i_{j-1},i_{j+1},\dots,i_k$ then we have proved it for all, since $m$ and $k$ are bounded. So without loss of generality let us prove it for $j=1$ and for the sequence $(2,\dots,k)$. That is, we shall prove that with high probability
\begin{equation*}
\|*_1(\mu_2,\dots,\mu_k)-\circ_1(\mu_2,\dots,\mu_k)\|_1\leq\eta.
\end{equation*}
Our results will also imply the stronger statement R1($p, 1$), which was required for Lemma~\ref{allprofile}.

The basic approach is to show that with high probability the sets $U_2,\dots,U_{k-1}$ have certain properties that we can exploit, and that if they have those properties then the conditional probability that $\|*_1(\mu_2,\dots,\mu_k)-\circ_1(\mu_2,\dots,\mu_k)\|_1\leq\eta$ is also high. This strategy is almost forced on us: there are some choices of $U_2,\dots,U_{k-1}$ that would be disastrous, and although they are rare we have to take account of their existence.

To get some idea of what the useful properties are, let us suppose that we have chosen $U_2,\dots,U_{k-1}$, let us fix $x\in X$, and let us think about the random variable $*_1(\mu_2,\dots,\mu_k)(x)$ (which, given our choices, depends just on the random set $U_k$). This is, by definition,
\begin{equation*}
\E_{s\in S_1(x)}\mu_2(s_2)\dots\mu_{k-1}(s_{k-1})\mu_k(s_k).
\end{equation*}
At this point we need an extra homogeneity assumption. We would like to split up the above expectation according to the value of $s_k$, but that will lead to problems if different values of $s_k$ are taken different numbers of times. Let us suppose that for each $y$ the number of $s\in S_1(x)$ such that $s_k=y$, which is just the cardinality of the set $S_1(x)\cap S_k(y)$, only ever takes one of two values, one of which is 0. 

In the case of arithmetic progressions of length $k$ in $\Z_p$, with $p$ prime, $S_1(x)\cap S_k(y)$ consists of a unique arithmetic progression (degenerate if $x=y$), the progression with common difference $(k-1)^{-1}(y-x)$ that starts at $x$. In the case of, say, $K_5$s in a complete graph, where $s_1$ and $s_{10}$ represent disjoint edges of $K_5$, $S_1(e)\cap S_{10}(e')$ will be empty if $e$ and $e'$ are edges of $K_n$ that share a vertex, and will have cardinality $n-4$ if they are disjoint. In general, in all natural examples this homogeneity assumption is satisfied. Moreover, the proportion of $y$ for which $S_1(x)\cap S_k(y)=\emptyset$ tends to be $O(1/n)$ and tends to correspond to degenerate cases (when those are not allowed).

With the help of this assumption, we can rewrite the previous expression as follows. Let us write $K(x)$ for the set of $y$ such that $S_1(x)\cap S_k(y)\ne\emptyset$. Then
\begin{eqnarray*}
*_1(\mu_2,\dots,\mu_k)(x)&=&\E_{s\in S_1(x)}\mu_2(s_2)\dots\mu_{k-1}(s_{k-1})\mu_k(s_k)\\
&=&\E_{y\in K(x)}\mu_k(y)\E_{s\in S_1(x)\cap S_k(y)}\mu_2(s_2)\dots\mu_{k-1}(s_{k-1}).
\end{eqnarray*}
Writing $W(x,y)$ for $\E_{s\in S_1(x)\cap S_k(y)}\mu_2(s_2)\dots\mu_{k-1}(s_{k-1})$, we can condense this to $\E_{y\in K(x)}\mu_k(y)W(x,y)$.

Now we are thinking of $\mu_2,\dots,\mu_{k-1}$ as fixed, and of the expressions we write as random variables that depend on the random measure $\mu_k$. Note that the expectation of $*_1(\mu_2,\dots,\mu_k)(x)$ is $*_1(\mu_2,\dots,\mu_{k-1},1)(x)$. By the results of the previous section, we are free to assume that this is at most 3/2 for every $x$. 

Our plan is to prove that the expectation of $*_1(\mu_2,\dots,\mu_k)(x)-\circ_1(\mu_2,\dots,\mu_k)(x)$ is small for each $x$, which will show that the expectation of $\|*_1(\mu_2,\dots,\mu_k)-\circ_1(\mu_2,\dots,\mu_k)\|_1$ is small. Having done that, we shall argue that it is highly concentrated about its expectation.

Now, as we have seen, the random variable $*_1(\mu_2,\dots,\mu_k)(x)$ is equal to $\E_{y\in K(x)}\mu_k(y)W(x,y)$, which is a sum of independent random variables $V_y$, where $V_y=(p|K(x)|)^{-1}W(x,y)$ with probability $p$ and 0 otherwise. The expectation $\E_{y\in K(x)}W(x,y)$ of this sum is $*_1(\mu_2,\dots,\mu_{k-1},1)(x)$, which we are assuming to be at most 3/2. If we also know that each $V_y$ is small, then the chances that this sum is bigger than 2 are very small. From this it is possible to deduce that the expectation of $*_1(\mu_2,\dots,\mu_k)(x)-\circ_1(\mu_2,\dots,\mu_k)(x)$ is small. The following technical lemma makes these arguments precise.

In the statement of the next lemma, we write $\E_{y\in K}$ for the average over $K$, and $\E$ for the probabilistic expectation (over all possible choices of $\mu_k$ with their appropriate probabilities).

\begin{lemma}\label{meaniflarge}
Let $0\leq p\leq 1$ and let $0<\alpha\leq 1$. Let $K$ be a set and for each $y\in K$ let $V_y$ be a random variable that takes the value $C_y > 0$ with probability $p$ and 0 otherwise. Suppose that the $V_y$ are independent and that each $C_y$ is at most $\a$. Let  $S=\sum_{y\in K}V_y$ and suppose that $\E S\leq 3/2$. Let $T=\max\{S-2,0\}$. Then $\E T\leq 14\a e^{-1/14\a}$.
\end{lemma}

\begin{proof}
If we increase the number of random variables or any of the values $C_y$, then the expectation of $T$ increases. Therefore, we are done if we can prove the result in the case where $\E S=3/2$.

We shall use the elementary identity 
\begin{equation*}
\E T =\int_0^\infty\P[T \geq t]dt =\int_0^\infty\P[S \geq 2 + t]dt.
\end{equation*}
Since $\E S=3/2$, if $S\geq 2+t$ it follows that $S-\E S\geq t+1/2$. Let us bound the probability of this event using Bernstein's inequality (Lemma \ref{Bernstein}).

For this we need to bound $\sum_y\V(V_y)$, which is at most $\sum_y\E(V_y^2)$, which is at most $\a\sum_y\E(V_y)$, by our assumption about the upper bound for each $C_y$. But this is $\a\E S=3\a/2$. Therefore,
\begin{equation*}
\P[S\geq 2+t]\leq 2\exp\left\{\frac{-(t+1/2)^2}{2(3\a/2+\a(t+1/2)/3)}\right\}.
\end{equation*}
Writing $s=t+1/2$, this gives us $2 \exp(-s^2/(3\a+2\a s/3))$. When $s\geq 1/2$ (as it is everywhere in the integral we are trying to bound), this is at most $2\exp(-s^2/(6\a s+2\a s/3))\leq 2\exp(-s/7\a)$, so we have an upper bound of
\begin{equation*}
2\int_{1/2}^\infty \exp(-s/7\a)ds=14\a e^{-1/14\a},
\end{equation*}
which proves the lemma.
\end{proof}

\begin{corollary} \label{expectationatx}
Suppose that $\mu_2, \dots, \mu_{k-1}$ are fixed and that $W(x,y)\leq\a p|K(x)|$ for every $x$ and $y$ and $*_1(\mu_2,\dots,\mu_{k-1},1)(x) \leq 3/2$ for every $x$. Then 
\begin{equation*}
\E(*_1(\mu_2,\dots,\mu_k)(x)-\circ_1(\mu_2,\dots,\mu_k)(x))\leq 14\a e^{-1/14\a}
\end{equation*}
for every $x$.
\end{corollary}

\begin{proof}
As noted above, $*_1(\mu_2,\dots,\mu_k)(x)$ is a sum of independent random variables $V_y$ that take the value $(p|K(x)|)^{-1}W(x,y)$ with probability $p$ and 0 otherwise. By our hypothesis about $W(x,y)$, we can take $C_y=\a$ for each $y$ and apply the previous lemma. Then $S=*_1(\mu_2,\dots,\mu_k)(x)$ and $T=*_1(\mu_2,\dots,\mu_k)(x)-\circ_1(\mu_2,\dots,\mu_k)(x)$, so the result follows.
\end{proof}

The next result but one is our main general lemma, after which we shall have to argue separately for different kinds of system. We shall use the following concentration of measure result, which is an easy and standard consequence of Azuma's inequality.

\begin{lemma} \label{azuma}
Let $X^{(t)}$ be the collection of all subsets of size $t$ of a finite set $X$. Let $c,\lambda>0$ and let $F$ be a function defined on $X^{(t)}$ such that $|F(U)-F(V)|\leq c$ whenever $|U\cap V|=t-1$. Then if a random set $U\in X^{(t)}$ is chosen, the probability that $|F(U)-\E F|\geq \lambda$ is at most $2 \exp(-\lambda^2/2c^2t)$.
\end{lemma}

Most of the conditions of the next lemma have been mentioned in the discussion above, but we repeat them for convenience (even though the resulting statement becomes rather long).

\begin{lemma}\label{L1Diff}
Let $X$ be a finite set and let $S$ be a homogeneous collection of ordered subsets of $X$, each of size $k$. Let $\sigma$ be a positive integer and suppose that, for all $x, y \in X$, $|S_1(x) \cap S_k(y)| \in \{0, \sigma\}$. For each $x$, let $K(x)$ be the set of $y$ such that $S_1(x) \cap S_k(y) \ne\emptyset$, and suppose that all the sets $K(x)$ have the same size.

Let $\mu_2, \dots, \mu_{k-1}$ be fixed measures such that $*_1(\mu_2, \dots, \mu_{k-1}, 1)(x)$ and $*_k(1, \mu_2, \dots, \mu_{k-1})(x)$ are at most $3/2$ for every $x\in X$. For each $x, y \in X$, let
\[W(x,y) = \E_{s \in S_1(x) \cap S_k(y)} \mu_2(s_2) \dots \mu_{k-1}(s_{k-1})\]
and suppose that $W(x,y)\leq\a p|K(x)|$ for every $x$ and $y$.

Let $U_k$ be a random set chosen binomially with probability $p$, let $\mu_k$ be its associated measure, and let $\eta = 28\a e^{-1/14\a}$. Then
\begin{equation*}
\P[\|*_1(\mu_2,\dots,\mu_k) -\circ_1(\mu_2,\dots,\mu_k)\|_1 > \eta]\leq 2 |X| e^{-\eta^2 p |X|/144} + 2 e^{-p|X|/4}.
\end{equation*}
\end{lemma}

\begin{proof}
Corollary~\ref{expectationatx} and linearity of expectation imply that 
\begin{equation*}
\E\|*_1(\mu_2,\dots,\mu_k) -\circ_1(\mu_2,\dots,\mu_k)\|_1\leq 14\a e^{-1/14\a}=\eta/2.
\end{equation*}
Let us write $Z$ for the random variable $\|*_1(\mu_2,\dots,\mu_k) -\circ_1(\mu_2,\dots,\mu_k)\|_1$. To complete the proof, we shall show that $Z$ is highly concentrated about its mean.

To do this, we condition on the size of the set $U_k$ and apply Lemma \ref{azuma}. Suppose, then, that $|U_k|=t$. We must work out by how much we can change $Z$ if we remove an element of $U_k$ and add another. 

Since the function $x\mapsto\max\{x-2,0\}$ is 1-Lipschitz, the amount by which we can change $Z$ is at most the amount by which we can change $Y=\|*_1(\mu_2,\dots,\mu_k)\|_1$. But
\begin{eqnarray*}
\|*_1(\mu_2,\dots,\mu_k)\|_1&=&\E_x\E_{s\in S_1(x)}\mu_2(s_2)\dots\mu_{k-1}(s_{k-1})\mu_k(s_k)\\
&=&\E_y\mu_k(y)\E_{s\in S_k(y)}\mu_2(s_2)\dots\mu_{k-1}(s_{k-1})\\
&=&\E_y\mu_k(y)*_k(1,\mu_2,\dots,\mu_{k-1})(y).
\end{eqnarray*}
We are assuming that $*_k(1,\mu_2,\dots,\mu_{k-1})(y)$ is never more than 3/2, and $\mu_k(y)$ is always either $p^{-1}$ or 0, so changing one element of $U_k$ cannot change $Y$ by more than $3(p|X|)^{-1}$. (The division by $|X|$ is because we are taking an average over $y$ rather than a sum over $y$.) 

Lemma \ref{azuma} now tells us that the probability that $Z-\E Z\geq \eta/2$ given that $|U_k|=t$ is at most $2\exp(-\eta^2p^2|X|^2/72t)$. It follows that if $t\leq 2p|X|$ then the probability is at most $2\exp(-\eta^2p|X|/144)$. By Chernoff's inequality, the probability that $t>2p|X|$ is at most $2 \exp(-p|X|/4)$. Putting these two facts together and adding over all possible values of $t$, we obtain the result stated.
\end{proof}

Our aim is to prove that property P1 holds with high probability for a given small constant $\eta>0$. Therefore, it remains to prove that, under suitable conditions on $p$, we have the bound $W(x,y)\leq\a p|K(x)|$ for every $x,y\in X$ such that $S_1(x)\cap S_k(y)$ is non-empty, where $\a$ is also a given small constant. Here, the argument once again depends on the particular form of the set of sequences $S$. 

In the case of sets with two degrees of freedom, this is trivial. Let us suppose that $|K(x)|=t$ for every $x\in X$. By definition, $S_i(x)\cap S_j(y)$ is either empty or a singleton for every $1\leq i<j\leq k$ and every $x,y\in X$. It follows, when $S_1(x)\cap S_k(y)$ is non-empty, that
\begin{eqnarray*}
W(x,y)&=&\E_{s\in S_1(x)\cap S_k(y)}\mu_2(s_2)\dots\mu_{k-1}(s_{k-1})\\
&=&\mu_2(r_2)\dots\mu_{k-1}(r_{k-1})\\
&\leq&p^{-(k-2)},
\end{eqnarray*}
where $r=(x,r_2,\dots,r_{k-1},y)$ is the unique element of $S$ that belongs to $S_1(x)\cap S_k(y)$. This is smaller than $\a pt$ as long as $p\geq (\a t)^{-1/(k-1)}$. Recall that in a typical instance, such as when $S$ is the set of $k$-term arithmetic progressions in $\mathbb{Z}_n$ for some prime $n$, $t$ will be very close to $n$ (or in that case actually equal to $n$), and we do indeed obtain a bound of the form $C n^{-1/(k-1)}$ that is within a constant of best possible.

Thus, we have essentially already finished the proof of a sparse random version of Szemer\'edi's theorem, and of several other similar theorems. We will spell out the details of these applications later in the paper. Now, however, let us turn to the more difficult task of verifying the hypothesis about $W$ in the case of graphs and hypergraphs. 

Let $H$ be a strictly $r$-balanced $r$-uniform hypergraph. Recall that $m_r(H)$ is the ratio $(e_H-1)/(v_H-r)$. The significance of $m_r(H)$ is that if $G_{n,p}^{(r)}$ is a random $r$-uniform hypergraph on $n$ vertices, with each edge chosen with probability $p$, then the expected number of labelled copies of $H$ containing any given edge of $G_{n,p}^{(r)}$ is approximately $p^{e_H-1}n^{v_H-r}$ (the ``approximately" being the result of a few degenerate cases), so we need $p\geq n^{-1/m_r(H)}$ for this expected number to be at least 1, which, at least in the density case, is a trivial necessary condition for our theorems to hold. Our main aim now is to prove that $W(x,y) \leq \alpha p |K(x)|$ holds when $p\geq Cn^{-1/m_r(H)}$, where $C$ is a constant that depends only on $\alpha$ and the hypergraph $H$.

In the next result, we shall take $p$ to equal $Cn^{-1/m_r(H)}$ and prove that the conclusion holds provided $C$ is sufficiently large. However, it turns out that we have to split the result into two cases. In the first case, we also need to assume that $C$ is smaller than $n^c$ for some small positive constant $c$, or else the argument breaks down. However, when $C$ is larger than this (so not actually a constant) we can quote results of Janson and Ruci\'nski to finish off the argument. (Some of our results, in particular colouring theorems, are monotone, in the sense that the result for $p$ implies the result for all $q\geq p$. In such cases we do not need to worry about large $p$.)

\begin{lemma} \label{ExtraCondHyper}
Let $H$ be a strictly $r$-balanced $r$-uniform hypergraph and let $S$ be the collection of labelled ordered copies of $H$ in the complete $r$-uniform hypergraph $K_n^{(r)}$. Then, for any positive constants $\a$ and $A$, there exist constants $c>0$ and $C_0$ such that, if $n$ is sufficiently large, $C_0\leq C\leq n^c$, and $p=C n^{-1/m_r(H)}$, then, with probability at least $1 - n^{-A}$, if $U_2, \dots, U_{e-1}$ are random subgraphs $G_{n,p}^{(r)}$ of $K_n^{(r)}$ with associated measures $\mu_2, \dots, \mu_{e-1}$, 
\[W(x,y) = \E_{s \in S_1(x) \cap S_e(y)} \mu_2(s_2) \dots \mu_{e-1}(s_{e-1}) \leq \a p |K(x)|,\]
for all $x,y \in X$, where we have written $e$ for $e_H$.
\end{lemma}

\begin{proof}
Let $\chi_i$ be the characteristic function of $U_i$ for each $i\leq e_H$. Let $\sigma$ be the size of each non-empty set $S_1(x)\cap S_e(y)$ and suppose $|K(x)| = t$ for each $x$. Then 
\begin{equation*}
W(x,y)=\sigma^{-1}p^{-(e_H-2)}\sum_{s\in S_1(x)\cap S_e(y)}\chi_2(s_2)\dots\chi_{e-1}(s_{e-1}).
\end{equation*}
But $\sum_{s\in S_1(x)\cap S_e(y)}\chi_2(s_2)\dots\chi_{e-1}(s_{e-1})$ is the number of sequences $(s_1,\dots,s_e)\in S$ such that $s_1=x$, $s_e=y$ and $s_i\in U_i$ for $i=2,3,\dots,e-1$. Therefore, our aim is to prove that with high probability this number is at most $\a pt\sigma p^{e_H-2}=\a p^{e_H-1}\sigma t$. Let $h$ be the number of vertices in the union of the first and $e$th edges. Then $\sigma$ is almost exactly $n^{v_H-h}$ and $t$ is almost exactly $n^{h-r}$, so it is enough to prove that with high probability the number of such sequences is at most $(\a/2)p^{e_H-1}n^{v_H-r}=(\a/2)C^{e_H-1}$. To do this, let us estimate from above the probability that there are at least $(v_H \ell)^{v_H}$ such sequences.

It will be convenient to think of each sequence in $S_1(x)\cap S_e(y)$ as an embedding $\phi$ from $H$ to $K_n^{(r)}$ such that, writing $f_1,\dots,f_e$ for the edges of $H$, we have $\phi(f_1)=x$ and $\phi(f_e)=y$. Let us call $\phi$ \textit{good} if in addition $\phi(e_i)\in U_i$ for $i=2,3,\dots,e-1$. Now if there are $(v_H \ell)^{v_H}$ good embeddings, then there must be a sequence $\phi_1,\dots,\phi_\ell$ of good embeddings such that each $\phi_i(H)$ contains at least one vertex that is not contained in any of $\phi_1(H),\dots,\phi_{i-1}(H)$. That is because the number of vertices in the union of the images of the embeddings has to be at least $v_H \ell$, since the number of embeddings into a set of size $u$ is certainly no more than $u^{v_H}$, and because each embedding has $v_H$ vertices.

Let us fix a sequence of embeddings $\phi_1,\dots,\phi_\ell$ such that each one has a vertex in its image that is not in the image of any previous one. Let $v_1,\dots,v_m$ be the sequence of vertices obtained by listing all the vertices of $\phi_1(H)$ in order (taken from an initial fixed order of the vertices of $H$), then all the vertices of $\phi_2(H)$ that have not yet been listed, again in order, and so on. For each $i\leq \ell$, let $V_i$ be the set of vertices in $\phi_i(H)$ but no earlier $\phi_j(H)$. We shall now estimate the probability that every $\phi_i$ is good. If we already know that $\phi_1,\dots,\phi_{i-1}$ are all good, then what we need to know is how many edges belong to $\phi_i(H)$ that do not belong to $\phi_j(H)$ for any $j<i$. Let $w_i=|V_i|$ be the number of vertices that belong to $\phi_i(H)$ and to no earlier $\phi_j(H)$, and let $d_i$ be the number of edges. Then the conditional probability that $\phi_i$ is good is $p^{d_i}$. It follows that the probability that $\phi_1,\dots,\phi_\ell$ are all good is $p^{d_1+\dots+d_\ell}$. The number of possible sequences of embeddings of this type is at most $m^{v_H \ell}n^m$, since there are at most $n^m$ sequences $v_1,\dots,v_m$, and once we have chosen $v_1, \dots,  v_m$ there are certainly no more than $m^{v_H}$ ways of choosing the embedding $\phi_i$ (assuming that its image lies in the set $\{v_1,\dots,v_m\}$). Therefore, the probability that there exists a good sequence of $\ell$ embeddings of this type is at most $m^{v_H \ell}p^{d_1+\dots+d_\ell}n^{w_1+\dots+w_\ell}$.

At this point, we use the hypothesis that $H$ is strictly balanced. Since $w_i \leq v_H - h \leq v_H - (r+1)$,
\[\frac{e_H - 1 - d_i}{v_H - r - w_i} < \frac{e_H - 1}{v_H - r},\]
which implies that $d_i/w_i > m_r(H)$. In fact, since there are only finitely many possibilities for $w_i$ and $d_i$, it tells us that there is a constant $c'>0$ depending on $H$ only such that $d_i\geq m_r(H)(w_i+c')$. Since $p=Cn^{-1/m_r(H)}$, this tells us that $p^{d_i}\leq C^{d_i}n^{-(w_i+c')}$, and hence that
\begin{equation*}
m^{v_H \ell}p^{d_1+\dots+d_\ell}n^{w_1+\dots+w_\ell}\leq m^{v_H \ell}C^{d_1+\dots+d_\ell}n^{-\ell c'}.
\end{equation*}

To complete the proof, let us show how to choose $C$, just to be sure that the dependences are correct. We start by choosing $\ell$ such that $\ell c'\geq 2A$. Bearing in mind that $m\leq v_H \ell$ and that $d_1+\dots+d_\ell \leq e_H \ell$, we place on $C$ the upper bound $C\leq (v_H \ell)^{-v_H \ell}n^{A/e_H \ell}$, which ensures that $m^{v_H \ell}C^{d_1+\dots+d_\ell}n^{-\ell c'}\leq n^{-A}$. Finally, we need $C$ to be large enough for $(\a/2)C^{e_H-1}$ to be greater than $(v_H \ell)^{v_H}$, since then the probability that there are at least $(\a/2)C^{e_H-1}$ sequences is at most $n^{-A}$, which is what we were trying to prove. Thus, we need $C$ to be at least $(2(v_H \ell)^{v_H}/\a)^{1/(e_H - 1)}$.
\end{proof}

To handle the case where $C \geq n^c$, we shall again need to appeal to the work of Janson and Ruci\'nski on upper tail estimates. The particular random variable we will be interested in, which concerns hypergraphs which are rooted on two edges, is defined as follows.

\begin{notation*}
Let $K$ be an $r$-uniform hypergraph and $f_1, f_2$ edges in $K$. Let $l$ be the number of edges in $K\char92\{f_1, f_2\}$ and let $U_1, \dots, U_l$ be random binomial subhypergraphs of the complete $r$-uniform hypergraph $K_n^{(r)}$ on $n$ vertices, each edge being chosen with probability $p$, with characteristic functions $\chi_1, \dots, \chi_l$. Let $S_{f_1,f_2}$ be the set consisting of all labelled ordered copies of $K\char92\{f_1, f_2\}$ in $K_n^{(r)}$ that are rooted at given edges $e_1$ and $e_2$. Then the random variable $Y_K^{f_1, f_2}$ is given by
\[\sum_{s \in S_{f_1, f_2}} \prod_{1 \leq i \leq l} \chi_i (s_i).\]
\end{notation*}

The necessary tail estimate (which is another particular case of Corollary 4.1 in \cite{JR04}) is now the following. Note that $\E Y_K^{f_1, f_2}$ is essentially $p^{e_K - 2} n^{v_K - h}$, where $h$ is the size of $f_1 \cup f_2$. 

\begin{lemma} \label{JRTwoEdge}
Let $K$ be an $r$-uniform hypergraph and $f_1, f_2$ fixed edges. Then there exists a constant $c$ such that the random variable
$Y_K^{f_1,f_2}$ satisfies, for $\gamma \geq 2$,
\[\P\left(Y_K^{f_1,f_2} \geq \gamma \E Y_K^{f_1, f_2}\right) \leq 2 n^{v_K} \exp\left(-c
\min_{L \subseteq K} \left(\gamma \E Y_L^{f_1, f_2}\right)^{1/v_L}\right).\]
\end{lemma}

The required estimate for $p \geq n^{-1/m_k(H) + c}$ is now an easy consequence of this lemma. 

\begin{lemma} \label{ExtraCondHyper2}
Let $H$ be a strictly $r$-balanced $r$-uniform hypergraph and let $S$ be the collection of labelled ordered copies of $H$ in the complete $r$-uniform hypergraph $K_n^{(r)}$. Then, for any positive constants $\a$ and $c$, there exist constants $b$ and $B$ such that, if $n$ is sufficiently large, $C \geq n^c$, and $p=C n^{-1/m_r(H)}$, then, with probability at least $1 - 2 n^{v_H} e^{-Bn^b}$, if $U_2, \dots, U_{e-1}$ are random subgraphs $G_{n,p}^{(r)}$ of $K_n^{(r)}$ with associated measures $\mu_2, \dots, \mu_{e-1}$, 
\[W(x,y) = \E_{s \in S_1(x) \cap S_e(y)} \mu_2(s_2) \dots \mu_{e-1}(s_{e-1}) \leq \a p |K(x)|,\]
for all $x,y \in X$, where we have written $e$ for $e_H$.
\end{lemma}

\begin{proof}
Let $\chi_i$ be the characteristic function of $U_i$ for each $i\leq e_H$. Let $\sigma$ be the size of each non-empty set $S_1(x)\cap S_e(y)$ and suppose $|K(x)| = t$ for each $x$. Note that $\E Y_H^{e_1, e_H} = p^{e_H-2} \sigma$. We may apply Lemma \ref{JRTwoEdge} with $\gamma = \a p t$ to tell us that $\sum_{s \in S_1(x) \cap S_e(y)} \chi_2(s_2) \dots \chi_{e-1}(s_{e-1}) \geq \gamma p^{e_H-2} \sigma$ with probability at most 
\[2 n^{v_H} \exp\left(-c \min_{L \subseteq H} \left(\gamma \E Y_L^{e_1, e_H}\right)^{1/v_L}\right) = 2 n^{v_H} \exp\left(-c
\min_{L \subseteq H} \left(\gamma n^{v_L - h} p^{e_L - 2}\right)^{1/v_L}\right),\]
where $h$ is the size of $e_1 \cup e_H$. Note that, as $t$ is almost exactly $n^{h-r}$, $\gamma n^{v_L - h} p^{e_L - 2} \geq (\a/2) n^{v_L - r} p^{e_L - 1}$. Since $H$ is strictly $r$-balanced, for any proper subgraph $L$ of $H$,
\[n^{v_L - r} p^{e_L - 1} \geq \left(n^{1 -\left(\frac{v_H - r}{e_H - 1}\right) \left(\frac{e_L - 1}{v_L - r}\right)}\right)^{v_L - r} \geq n^{b'}.\]
Since also $n^{v_H - r} p^{e_H - 1} \geq n^{\e}$, the required bound holds with probability at least $1 - 2 n^{v_H} e^{-B n^b}$ for some constants $B$ and $b$. Since 
\begin{equation*}
W(x,y)=\sigma^{-1}p^{-(e_H-2)}\sum_{s\in S_1(x)\cap S_e(y)}\chi_2(s_2)\dots\chi_{e-1}(s_{e-1}),
\end{equation*}
the result now follows for $n$ sufficiently large.
\end{proof}

\section{Summary of our results so far} \label{summary}

We are about to discuss several applications of our main results. In this brief section, we prepare for these applications by stating the abstract results that follow from the work we have done so far. Since not every problem one might wish to solve will give rise to a system of sequences $S$ that either has two degrees of freedom or concerns copies of a strictly balanced graph or hypergraph, we begin by stating sufficient conditions on $S$ for theorems of the kind we are interested in to hold. We have of course already done this, but since some of our earlier conditions implied other ones, there is scope for stating the abstract results more concisely. That way, any further applications of our methods will be reduced to establishing two easily stated probabilistic estimates, and showing that suitable robust versions of the desired results hold in the dense case.

Having done that, we remark that we have proved that the estimates hold when $S$ has two degrees of freedom or results from copies of a strictly balanced graph or hypergraph. So in these two cases, if the robust results hold in the dense case, then we can carry them over unconditionally to the sparse random case.

The proofs in this section require little more than the putting together of results from earlier in the paper.

\subsection{Conditional results}

Recall that a system $S$ of sequences $s=(s_1,\dots,s_k)$ with values in a finite set $X$ is \textit{homogeneous} if for every $j\leq k$ and every $x\in X$ the set $S_j(x)=\{s\in S:s_j=x\}$ has the same size. Let $S$ be a homogeneous system of sequences with elements in a finite set $X$, and let us assume that no sequence in $S$ has repeated elements. We shall also assume that all non-empty sets of the form $S_1(x)\cap S_k(y)$ have the same size. Coupled with our first homogeneity assumption, this implies that for each $x$ the number of $y$ such that $S_1(x)\cap S_k(y)$ is non-empty is the same.

We are about to state and prove a theorem that is similar to Theorem \ref{main:density}, but with conditions that are easier to check and a conclusion that is more directly what we want to prove. The first condition is what we proved in Lemmas \ref{LInfInd} and \ref{LInfHyper}.  We suppose that $X$ is a given finite set, $S$ is a given homogeneous system of sequences with terms in $X$, and $p_0$ is a given probability. 

\begin{condition1*} 
Let $U_1,\dots,U_k$ be independent random subsets of $X$, each chosen binomially with probability $p \geq p_0$, and let $\mu_1,\dots,\mu_k$ be their associated measures. Let $1\leq j\leq k$ and for each $i\ne j$ let $\nu_i$ equal either $\mu_i$ or the constant measure 1 on $X$, with at least one $\nu_i$ equal to the constant measure. Then with probability at least $1-o(|X|^{-k})$,
\begin{equation*}
*_j(\nu_1,\dots,\nu_{j-1},\nu_{j+1},\dots,\nu_k)(x)\leq 3/2
\end{equation*}
for every $x\in X$. 
\end{condition1*}

Recall that if $L$ is the set of $i$ such that $\nu_i=\mu_i$, then $*_j(\nu_1,\dots,\nu_{j-1},\nu_{j+1},\dots,\nu_k)(x)$ is $p^{-|L|}$ times the number of $s\in S_j(x)$ such that $s_i\in U_i$ for every $i\in L$. Since the expected number of such sequences is $p^{|L|}|S_j(x)|$, Condition $1$ is saying that their number is not too much larger than its mean. (One would usually expect a concentration result that said that their number is, with high probability, close to its mean.)

The second condition tells us that the hypotheses of Lemma \ref{L1Diff} hold. Again we shall take $X$, $S$ and $p$ as given.

\begin{condition2*}
Let $U_2,\dots,U_{k-1}$ be independent random subsets of $X$, each chosen binomially with probability $p \geq p_0$, and let $\mu_2,\dots,\mu_{k-1}$ be their associated measures. Let $\a>0$ be an arbitrary positive constant. For each $x$, let $t$ be the number of $y$ such that $S_1(x)\cap S_k(y)$ is non-empty. Then with probability at least $1-o(|X|^{-k})$,
\begin{equation*}
W(x,y)=\E_{s\in S_1(x)\cap S_k(y)}\mu_2(s_2)\dots\mu_{k-1}(s_{k-1})\leq \a pt
\end{equation*}
for every $x,y$ such that $S_1(x)\cap S_k(y)$ is non-empty.
\end{condition2*}

This is not a concentration assumption. For instance, in the case of systems with two degrees of freedom, it follows trivially from the fact that $|S_1(x)\cap S_k(y)|\leq 1$ and each $\mu_i(s_i)$ is at most $p^{-1}$. In more complicated cases, we end up wishing to prove that a certain integer-valued random variable with mean $n^{-c}$ has a probability $n^{-A}$ of exceeding a large constant $C$.

We are now ready to state our main conditional results. Note that in all of these it is necessary to assume that the probability $q$ with which we choose our random set $U$ is smaller than some positive constant $\d$. For colouring theorems this is not a problem, because these properties are always monotone. It is therefore enough to know that the property holds almost surely for a particular probability $q$ to know that it holds almost surely for all probabilities larger than $q$. 

For density theorems, we can also overcome this difficulty by partitioning any random set with large probability into a number of smaller random sets each chosen with probability less than $\d$. With high probability, each of these smaller random sets will satisfy the required density theorem. If we take a subset of the original set above a certain density, then this subset must have comparable density within at least one of the sets of the partition. Applying the required density theorem within this set, we can find the required substructure, be it a $k$-term arithmetic progression or a complete graph of order $t$. 

Alternatively, if we know a (robust) sparse density theorem for a small value of $p$, we can deduce it for a larger value $q$ as follows. We can pick a random set $V=X_p$ by first choosing $U=X_q$ and then choosing $V=U_{p/q}.$ Since the result is true for almost every $V=X_p$, it will be the case that for almost every $U=X_q$, almost every $V=U_{p/q}$ will satisfy the result. It follows by a simple averaging argument that for almost every $U=X_q$ the robust version of the density theorem holds again.

Unfortunately, for structural results, no simple argument of this variety seems to work and we will have to deal with each case as it comes.

\begin{theorem} \label{conditionaldensity}
Suppose that $S$, $X$ and $p_0$ satisfy Conditions 1 and 2. Suppose also that there exist positive constants $\rho$ and $\b$ such that for every subset $B\subset X$ of density at least $\rho$ there are at least $\b|S|$ sequences $s=(s_1,\dots,s_k)\in S$ such that $s_i\in B$ for every $i$. Then, for any $\e > 0$, there exist positive constants $C$ and $\d$ with the following property. Let $U$ be a random subset of $X$, with elements chosen independently with probability $C p_0 \leq q \leq \d$. Then, with probability $1 - o(1)$, every subset $A$ of $U$ of density at least $\rho+\e$ contains at least $(\b-\e)p^k|S|$ sequences such that $s_i\in A$ for every $i$.
\end{theorem}

\begin{proof}
Basically the result is true because Theorem \ref{main:density} proves the conclusion conditional on the four key properties set out before the statement of Theorem \ref{main:density}, and our probabilistic arguments in the last few sections show that these properties follow from Conditions 1 and 2. Indeed, we would already be done if it were not for one small extra detail: we need to deal with the fact that Theorem \ref{main:density} has a conclusion that concerns $m$ random sets $U_1,\dots,U_m$, whereas we want a conclusion that concerns a single random set $U$. 

Let $\eta$, $\lambda$, $d$ and $m$ be as required by Theorem \ref{main:density}. Condition 1 implies that property P2 holds with probability $1-o(|X|^{-k})$. Lemma \ref{L1Diff} tells us that property P1 holds with probability $1-o(|X|^{-k})$ provided that Conditions $1$ and $2$ hold, for some $\a$ that depends on $\eta$. Property P0 plainly holds with high probability. Finally, Lemma \ref{allprofile} tells us that if properties P0, P1 and P2 hold with probability $1-o(|X|^{-k})$, then property P3 holds with probability $1 - o(1)$. Thus, with probability $1-o(1)$, we have all four properties.

It follows from Theorem \ref{main:density} that if $U_1,\dots,U_m$ are independent random subsets of $X$, each chosen binomially with probability $p \geq p_0$, and $\mu_1,\dots,\mu_m$ are their associated measures, then, with probability $1-o(1)$, $\E_{s\in S}h(s_1)\dots h(s_k)\geq \b-\e$ whenever $0\leq h\leq m^{-1}(\mu_1+\dots+\mu_m)$ and $\|h\|_1\geq\rho+3\e/4$.

Let $U=U_1\cup\dots\cup U_m$. Then $U$ is a random set with each element chosen independently with probability $q=1-(1-p)^m  \geq mp(1-\e/8)$, provided $\d$ (and hence $p$) is sufficiently small. Let $\mu$ be the associated measure of $U$, let $0\leq f\leq\mu$ and suppose that $\|f\|_1\geq\rho+\e$. Then replacing $f$ by the smaller function $h=\min\{f,m^{-1}(\mu_1+\dots+\mu_m)\}$ we have $\|h\|_1\geq\rho+3\e/4$, which implies that $\E_{s\in S}h(s_1)\dots h(s_k)\geq \b-\e$, which in turn implies that $\E_{s\in S}f(s_1)\dots f(s_k)\geq \b-\e$.
\end{proof}

Conditions 1 and 2 also imply an abstract colouring result and an abstract structural result in a very similar way.

\begin{theorem} \label{conditionalcolouring}
Suppose that $S$, $X$ and $p_0$ satisfy Conditions 1 and 2. Suppose also that $r$ is a positive integer and $\b$ a positive constant such that for every colouring of $X$ with $r$ colours there are at least $\b|S|$ sequences $s=(s_1,\dots,s_k)\in S$ such that each $s_i$ has the same colour.  Then there exist positive constants $C$ and $\d$ with the following property. Let $U$ be a random subset of $X$, with elements chosen independently with probability $C p_0 \leq q \leq \d$. Then, with probability $1-o(1)$, every colouring of $U$ with $r$ colours contains at least $2^{-(k+2)}\b p^k|S|$ sequences $s=(s_1,\dots,s_k)\in S$ such that each $s_i$ has the same colour and each $s_i$ is an element of $U$.
\end{theorem}

The only further ingredient needed to prove this theorem is Theorem \ref{main:colouring}. Other than this, the proof is much the same as that of Theorem \ref{conditionaldensity}.

\begin{theorem} \label{conditionalstructural}
Suppose that $S$, $X$ and $p_0$ satisfy Conditions 1 and 2 and let $\mathcal{V}$ be a collection of $2^{o(p|X|)}$ subsets of $X$. Then, for any $\e > 0$, there exist positive constants $C$ and $\d$ with the following property. Let $U$ be a random subset of $X$, with elements chosen independently with probability $C p_0 \leq q \leq \d$, and associated measure $\mu$. Then, with probability $1-o(1)$, for every function $f:X\to\mathbb{R}$ with $0\leq f\leq\mu$  there exists a function $g:X\to\R$ with $0\leq g\leq 1$ such that 
\begin{equation*}
\E_{s\in S}f(s_1)\dots f(s_k) \geq \E_{s\in S}g(s_1)\dots g(s_k) - \e
\end{equation*}
and, for all $V\in\mathcal{V}$,
\begin{equation*}
|\sum_{x\in V}f(x)-\sum_{x\in V}g(x)|\leq\e|X|.
\end{equation*}
\end{theorem}

The main extra point to note here is that Conditions $1$ and $2$ imply not just property P3 but also property P3$'$. This allows us to apply Theorem \ref{main:structure}.

\subsection{The critical exponent}

The aim of this paper has been to prove results that are, in terms of $p$, best possible to within a constant. A preliminary task is to work out the probability below which we cannot hope to prove a result. (For density problems, it is easy to prove that below this probability the result is not even true. For natural colouring problems, it usually seems to be the case that the result is not true, but the known proofs are far from trivial.) To within a constant, the probability in question is the probability $p$ such that the following holds: for each $j\leq k$ and each $x\in X$ the expected number of elements $s\in S$ such that $s_j=x$ (that is, such that $s\in S_j(x)$), and $s_i$ belongs to $X_p$ for each $i\ne j$ is equal to 1. 

In concrete situations, $X$ will be one of a family of sets of increasing size, and $S$ will be one of a corresponding family of sets of sequences. Then it is usually the case that the probability $p$ calculated above is within a constant of $|X|^{-\a}$ for some rational number $\a$ that does not depend on which member of the family one is talking about. In this situation, we shall call $\a$ the \textit{critical exponent} for the family of problems. Our results will then be valid for all $p$ that exceed $C|X|^{-\a}$ for some constant $C$. We shall denote the critical exponent by $\a_S$, even though strictly speaking it depends not on an individual $S$ but on the entire family of sets of sequences.

To give an example, if $S$ consists of all non-degenerate edge-labelled copies of $K_4$ in $K_n$, then the expected number of copies with a particular edge in a particular place, given that that edge belongs to $U$, is $2(n-2)(n-3)p^5$ (since each $S_j(e)$ has size $2(n-2)(n-3)$ and there are five edges that must be chosen). Setting that equal to 1 tells us that $p$ is within a constant of $n^{-2/5}$, so the critical exponent is $2/5$. (This is a special case of the formula $1/m_k(K)=(v_K-k)/(e_K-1)$.)

This calculation is exactly what we do in general: if each element of $S$ is a sequence of length $k$ and we are given that $x\in X_p$, then the expected number of elements of $S_j(x)$ that have all their terms in $X_p$ is $p^{k-1}|S_j(x)|$. This equals 1 when $p=|S_j(x)|^{-1/(k-1)}$. If $|S_j(x)|=C|X|^\theta$ for some $\theta$ that is independent of the size of the problem, then the critical exponent is therefore $\theta/(k-1)$. 

If we can prove a robust density theorem for $S$, and can show that Conditions 1 and 2 hold when $p_0=C|X|^{-\a_S}$ for some constant $C$, then we have proved a result that is best possible to within a constant. For colouring theorems, we cannot be quite so sure that the result is best possible, but in almost all examples where the 0-statement has been proved, it does indeed give a bound of the form $c|X|^{-\a_S}$. 

\subsection{Unconditional results}

In this section we concentrate on the two kinds of sequence system for which we have proved that Conditions 1 and 2 hold when $p_0=C|X|^{-\a_S}$.

Recall that $S$ \textit{has two degrees of freedom} if $S_i(x)\cap S_j(y)$ is either empty or a singleton whenever $i\ne j$. A good example of such a system is the set of $k$-term arithmetic progressions in $\Z_p$ for some prime $p$. We say that $S$ \textit{is a set of copies of a hypergraph} if $K$ is a $k$-uniform hypergraph with edges $a_1,\dots,a_e$, $X$ is the complete $k$-uniform hypergraph $K_n^{(k)}$, and $S$ is the set of all sequences of the form $(\phi(a_1),\dots,\phi(a_e))$, where $\phi$ is an injection from the vertices of $K$ to $\{1,2,\dots,n\}$. 

As above, we assume that $S$ has the additional homogeneity property that $S_1(x)\cap S_k(y)$ always has the same size when it is non-empty. (In the hypergraph case, $e$ plays the role of $k$ and $k$ has a different meaning: thus, the property in that case is that $S_1(x)\cap S_e(y)$ always has the same size when it is non-empty.) And in the hypergraph case, we make the further assumption that the hypergraph $K$ is \textit{strictly balanced}, which means that for every proper subhypergraph $J\subset K$ we have the inequality $\frac{e_J - 1}{v_J - k} < \frac{e_K-1}{v_K - k}$. When this happens, we write $m_k(K)$ as shorthand for $\frac{e_K-1}{v_K-k}$. 

Given a system $S$ with two degrees of freedom, let $t$ be the size of each $S_j(x)$, and suppose that $t=|X|^\g$. Then the critical exponent of $S$ is $\g/(k-1)$. (Note that $|X|^{-\a_S}=t^{-1/(k-1)}$.) When $S$ is a set of copies of a strictly balanced hypergraph $K$, the critical exponent is $1/m_k(K)$. It is straightforward to show that sparse density results cannot hold for random subsets of $X$ chosen with probability $c|X|^{-\a_S}$ if $c$ is a sufficiently small positive constant. Broadly speaking, we shall show that they \textit{do} hold for random subsets chosen with probability $C|X|^{-\a_S}$ when $C$ is a sufficiently \textit{large} positive constant.

Let us call a system $S$ \textit{good} if the above properties hold. That is, roughly speaking, a good system is a system with certain homogeneity properties that either has two degrees of freedom or comes from copies of a graph or hypergraph. We shall also assume that $|X|$ is sufficiently large. When we say ``there exists a constant $C$," this should be understood to depend only on $k$ in the case of systems of two degrees of freedom, and only on $K$ in the case of copies of a strictly balanced hypergraph, together with parameters such as density or the number of colours in a colouring that have been previously mentioned in the statement.

\begin{theorem}\label{unconditionaldensity}
Let $X$ be a finite set and let $S$ be a good system of ordered subsets of $X$. Suppose that there exist positive constants $\rho$ and $\b$ such that for every subset $B\subset X$ of density at least $\rho$ there are at least $\b|S|$ sequences $s=(s_1,\dots,s_k)\in S$ such that $s_i\in B$ for every $i$. Then, for any $\e > 0$, there exist positive constants $C$ and $\d$ with the following property. Let $U$ be a random subset of $X$, with elements chosen independently with probability $C |X|^{-\a_S} \leq p \leq \d$. Then, with probability $1 - o(1)$, every subset $A$ of $U$ of order at least $(\rho+\e)|U|$ contains at least $(\b-\e)p^k|S|$ sequences such that $s_i\in A$ for every $i$.
\end{theorem}

\begin{proof}
By Theorem \ref{conditionaldensity}, all we have to do is check Conditions 1 and 2. Condition 1 is given to us by Lemma \ref{LInfInd} when $S$ has two degrees of freedom, and by Lemma \ref{LInfHyper} when $S$ is a system of copies of a graph or hypergraph, even when $C=1$. (In the case where $S$ has two degrees of freedom, see the remarks following Lemma \ref{LInfInd} for an explanation of why the result implies Condition 1 when $p=|X|^{-\a_S}$.) 

When $S$ has two degrees of freedom, Condition 2 holds as long as $p^{-(k-2)}\leq\a pt$, as we have already remarked. This tells us that $p$ needs to be at least $(\a t)^{-1/(k-1)}$. In this case, $t=|S_j(x)|$ for each $x$ and $j$, so $(\a t)^{-1/(k-1)}$ is within a constant of $|X|^{-\a_s}$, as required. When $S$ comes from copies of a strictly balanced graph or hypergraph, Lemmas \ref{ExtraCondHyper} and \ref{ExtraCondHyper2} give us Condition 2, again with $p=C|X|^{-\a_S}$. 
\end{proof}

Exactly the same proof (except that we use Theorem \ref{conditionalcolouring} instead of Theorem \ref{conditionaldensity}) gives us the following general sparse colouring theorem.

\begin{theorem} \label{unconditionalcolouring}
Let $X$ be a finite set and let $S$ be a good system of ordered subsets of $X$. Suppose that $r$ is a positive integer and that $\b$ is a positive constant such that for every colouring of $X$ with $r$ colours there are at least $\b|S|$ sequences $s=(s_1,\dots,s_k)\in S$ such that each $s_i$ has the same colour. Then there exist positive constants $C$ and $\d$ with the following property. Let $U$ be a random subset of $X$, with elements chosen independently with probability $C |X|^{-\a_S} \leq p \leq \d$. Then, with probability $1-o(1)$, every colouring of $U$ with $r$ colours contains at least $2^{-(k+2)}\b p^k|S|$ sequences $s=(s_1,\dots,s_k)\in S$ such that each $s_i$ has the same colour and each $s_i$ is an element of $U$.
\end{theorem}

Finally, we have the following general sparse structural theorem.

\begin{theorem} \label{unconditionalstructural}
Let $X$ be a finite set and let $S$ be a good system of ordered subsets of $X$. Then, for any $\e > 0$, there exist positive constants $C$ and $\d$ with the following property. Let $U$ be a random subset of $X$, with elements chosen independently with probability $C |X|^{-\a_S} \leq p \leq \d$, let $\mu$ be the associated measure of $U$ and let $\mathcal{V}$ be a collection of $2^{o(p|X|)}$ subsets of $X$. Then, with probability $1-o(1)$, for every function $f$ with $0\leq f\leq\mu$ there exists a function $g$ with $0\leq g\leq 1$ such that
\begin{equation*}
\E_{s\in S}f(s_1)\dots f(s_k) \geq \E_{s\in S}g(s_1)\dots g(s_k) - \e
\end{equation*}
and, for all $V\in\mathcal{V}$,
\begin{equation*}
|\sum_{x\in V}f(x)-\sum_{x\in V}g(x)|\leq\e|X|.
\end{equation*}
\end{theorem}

In applications, we often want $g$ to take values in $\{0,1\}$ rather than $[0,1]$. This can be achieved by a simple and standard modification of the above result.

\begin{corollary}\label{unconditionalsetstructural}
Let $X$ be a finite set and let $S$ be a good system of ordered subsets of $X$. Then, for any $\e > 0$, there exist positive constants $C$ and $\d$ with the following property. Let $U$ be a random subset of $X$, with elements chosen independently with probability $C |X|^{-\a_S} \leq p \leq \d$, let $\mu$ be the associated measure of $U$ and let $\mathcal{V}$ be a collection of $2^{o(p|X|)}$ subsets of $X$. Then, with probability $1-o(1)$, for every function $f$ with $0\leq f\leq\mu$ there exists a function $h$ taking values in $\{0,1\}$ such that
\begin{equation*}
\E_{s\in S}f(s_1)\dots f(s_k) \geq \E_{s\in S}h(s_1)\dots h(s_k) - \e
\end{equation*}
and, for all $V\in\mathcal{V}$,
\begin{equation*}
|\sum_{x\in V}f(x)-\sum_{x\in V}h(x)|\leq\e|X|.
\end{equation*}
\end{corollary}

\begin{proof}
The basic idea of the argument is to choose a function $g$ that satisfies the conclusion of Theorem \ref{unconditionalstructural} with $\e$ replaced by $\e/2$, and to let $h(x)=1$ with probability $g(x)$ and 0 with probability $1-g(x)$, all choices being made independently. Then concentration of measure tells us that with high probability the estimates are not affected very much.

Note first that the expectation of $\E_{s\in S}h(s_1)\dots h(s_k)$ is $\E_{s\in S}g(s_1)\dots g(s_k).$ By how much can changing the value of $h(x)$ change the value of $\E_{s\in S}h(s_1)\dots h(s_k)$? Well, if $x$ is one of $s_1,\dots,s_k$ then $h(s_1)\dots h(s_k)$ can change by at most 1 and otherwise it does not change. The probability that $x$ is one of $s_1,\dots,s_k$ is $k/|X|$, by the homogeneity of $S$ (which tells us that each $s_j$ is uniformly distributed). By Azuma's inequality, it follows that the probability that $|\E_{s\in S}g(s_1)\dots g(s_k)-\E_{s\in S}h(s_1)\dots h(s_k)|\geq\e/2$ is at most $2\exp(-\e^2|X|/8k^2)$. This gives us the first conclusion with very high probability.

The second is obtained in a similar way. For each $V\in\mathcal{V}$ the probability that $|\sum_{x\in V}h(x)-\sum_{x\in V}g(x)|\geq\e|X|/2$ is, by Azuma's inequality, at most $2\exp(-\e^2|X|/8)$. Since there are $2^{o(p|X|)}$ sets in $\mathcal{V}$, a union bound gives the second conclusion with very high probability as well. 
\end{proof}

\section{Applications} \label{Applications}

\subsection{Density results}

As a first example, let us prove Theorem \ref{ApproxSzem}, the sparse analogue of Szemer\'edi's theorem. We shall consider Szemer\'edi's theorem as a statement about arithmetic progressions mod $p$ in dense subsets of $\Z_p$ for a prime $p$. We do this because the set of $k$-term arithmetic progressions in $\Z_p$ is a homogeneous system with two degrees of freedom. However, once we have the result for this version of Szemer\'edi's theorem, we can easily deduce it for the more conventional version concerning a sparse random subset of $[n]$. We simply choose a prime $p$ between $2n$ and $4n$, pick a sparse random subset $U$ of $\Z_p$, and then apply the result to  subsets of $U$ that happen to be subsets of $\{1,2,\dots,n\}$, since arithmetic progressions in these subsets will not wrap around. Similar arguments allow us to replace $[n]$ by $\Z_p$ for our later applications, so we mention once and for all now that for each application it is easy to deduce from the result we state a result for sparse subsets of intervals (or grids in the multidimensional case).

Since we wish to use the letter $p$ to denote a probability, we shall now let $n$ be a large prime.

By Theorem 9.1, all we have to do is check the robust version of Szemer\'edi's theorem, which can be proved by a simple averaging argument, originally observed by Varnavides \cite{V59} (who stated it for $3$-term progressions).

\begin{theorem} \label{Varnavides}
Let $k$ be an integer and $\d > 0$ a real number. Then there
exists an integer $n_0$ and $c > 0$ such that if $n$ is a prime greater than or equal to $n_0$ and
$B$ is a subset of $\Z_n$ with $|B| \geq \d n$ then $B$
contains at least $c n^2$ $k$-term arithmetic progressions.
\end{theorem}

\begin{proof}
Let $m$ be such that every subset of $\{1,2,\dots,m\}$ of density $\d/2$ contains a $k$-term arithmetic progression. Now let $B$ be a subset of $\Z_n$ of density $\d$. For each $a$ and $d$ with $d\ne 0$, let $P_{a,d}$ be the mod-$n$ arithmetic progression $\{a,a+d,\dots, a+(m-1)d\}$. If we choose $P_{a,d}$ at random, then the expected density of $B$ inside $P_{a,d}$ is $\d$, so with probability at least $\d/2$ it is at least $\d/2$. It follows with probability at least $\d/2$ that $P_{a,d}$ contains an arithmetic progression that is contained in $B$. Since $P_{a,d}$ contains at most $m (m-1)$ $k$-term arithmetic progressions, it follows that with probability at least $\d/2$ at least $1/m(m-1)$ of the $k$-term progressions inside $P_{a,d}$ are contained in $B$. But every $k$-term arithmetic progression is contained in the same number of progressions $P_{a,d}$. Therefore, the number of progressions in $B$ is at least $\frac{\delta}{2} \frac{n(n-1)}{m(m-1)} \geq \frac{\delta}{2} \frac{n^2}{m^2}$. 
\end{proof}

Very similar averaging arguments are used to prove the other robust density results we shall need in this subsection, so we shall be sketchy about the proofs and sometimes omit them altogether.

The next result is the sparse version of Szemer\'edi's theorem. Recall that we write $X_p$ for a random subset of $X$ where each element is chosen independently with probability $p$, and we say that a set $I$ is $(\d, k)$-\textit{Szemer\'edi} if every subset of $I$ with cardinality at least $\d|I|$ contains a $k$-term arithmetic progression. 

\begin{theorem} \label{RelativeSzemAgain}
Given $\delta > 0$ and a natural number $k \geq 3$, there exists
a constant $C$ such that if $p \geq C n^{-1/(k-1)}$, then the probability
that $(\Z_n)_p$ is $(\d, k)$-Szemer\'edi is $1-o(1)$.
\end{theorem}

\begin{proof}
In the case where $p$ is not too large, this follows immediately from Theorems \ref{unconditionaldensity} and \ref{Varnavides}. The result for larger probabilities can be deduced by using the argument given before Theorem \ref{conditionaldensity}. Alternatively, note that a subset of relative density $\d$ within a subset of $[n]_p$ has density $\d p$ in $[n]$. So if $p$ is larger than a fixed constant $\lambda$ (as it will be in the case not already covered by Theorem \ref{unconditionaldensity}), we can just apply Szemer\'edi's theorem itself. 
\end{proof}

A simple corollary of Theorem \ref{RelativeSzemAgain} is a sparse analogue of van der Waerden's theorem \cite{VdW27} on arithmetic progressions in $r$-colourings of $[n]$. Note that this theorem was proved much earlier by R\"odl and Ruci\'nski \cite{RR95} and is known to be tight. 

Let us now prove sparse versions of two generalizations of Szemer\'edi's theorem. The first generalization is the multidimensional Szemer\'edi theorem, due to Furstenberg and Katznelson \cite{FK78}. We shall state it in its robust form, which is in fact the statement that Furstenberg and Katznelson directly prove. (It also follows from the non-robust version by means of an averaging argument.)

\begin{theorem} \label{MultiSzem}
Let $r$ be a positive integer and $\d > 0$ a real
number. If $P \subset \Z^r$ is a fixed set, then there is a
positive integer $n_0$ and a constant $c > 0$ such that, for $n \geq n_0$, every
subset $B$ of the grid $[n]^r$ with $|B| \geq \d n^r$ contains $c
n^{r+1}$ subsets of the form $a + dP$, where $a \in [n]^r$ and $d$ is a positive integer.
\end{theorem}

Just as with Szemer\'edi's theorem, this statement is equivalent to the same statement for subsets of $\Z_n^r$. So let $P$ be a subset of $\Z^r$ and let $(x_1,\dots,x_k)$ be an ordering of the elements of $P$. Let $S$ be the set of sequences of the form $s_{a,d}=(a+x_1d,\dots,a+x_kd)$ with $d\ne 0$. Then $S$ is homogeneous and has two degrees of freedom. Moreover, if $n$ is large enough, then no two elements of $s_{a,d}$ are the same. From the conclusion of Theorem \ref{MultiSzem} it follows that there are at least $c|S|$ sequences in $S$ with all their terms in $B$. We have therefore checked all the conditions for Theorem \ref{unconditionaldensity}, so we have the following sparse version of the multidimensional Szemer\'edi theorem. (As before, the result for larger $p$ follows easily from the result for smaller $p$.) We define a subset $I$ of $\Z_n^r$ to be $(\d, P)$-\textit{Szemer\'edi} if every subset of $I$ with cardinality at least $\d|I|$ contains a homothetic copy $a+dP$ of $P$. 

\begin{theorem} \label{ApproxMultiSzem}
Given integers $r$ and $k$, a real number $\delta > 0$ and a
subset $P \subset \Z^r$ of order $k$, there exists a constant $C$
such that if $p \geq C n^{-1/(k-1)}$, then the probability that $(\Z_n^r)_p$ is
$(\d, P)$-Szemer\'edi is $1-o(1)$.
\end{theorem}

The second generalization of Szemer\'edi's theorem we wish to look at is the polynomial Szemer\'edi theorem of Bergelson and Leibman \cite{BL96}. Their result is the following.

\begin{theorem} \label{PolySzem}
Let $\d > 0$ be a real number, let $k$ be a positive integer and let $P_1,\dots,P_k$ be polynomials with integer coefficients that vanish at zero. Then there exists an integer $n_0$ such that if $n \geq n_0$ and $B$ is a subset of $[n]$ with $|B| \geq \d n$ then $B$ has a subset of the form $\{a,a+P_1(d),\dots,a+P_k(d)\}$.
\end{theorem}

We will focus on the specific case where the polynomials are $x^r, 2x^r,\dots,(k-1)x^r$ (so $k$ has been replaced by $k-1$). In this case, the theorem tells us that we can find a $k$-term arithmetic progression with common difference that is a perfect $r$th power. We restrict to this case, because it is much easier to state and prove an appropriate robust version for this case than it is for the general case. 

Note that if $a, a + d^r, \cdots, a + (k-1) d^r \in [n]$, then $d \leq (n/k)^{1/r}$. This observation and another easy averaging argument enable us to replace Theorem \ref{PolySzem} by the following equivalent robust statement about subsets of $\Z_n$ (see, for example, \cite{HL08}).

\begin{corollary} \label{ZnPolySzem}
Let $k, r$ be integers and $\d > 0$ a real number. Then there exists an integer $n_0$ and a constant $c > 0$ such that if $n \geq n_0$ and $B$ is a subset of $\Z_n$ with $|B| \geq \d n$ then $B$ contains at least $c n^{1 + 1/r}$ pairs $(a,d)$ such that $a, a + d^r, \cdots, a + (k-1) d^r \in B$ and 
$d \leq (n/k)^{1/r}$.
\end{corollary}

Let us say that a subset $I$ of $\Z_n$ is $(\d, k, r)$-\textit{Szemer\'edi} if every subset of $I$ with cardinality at least $\d |I|$ contains a $k$-term progression of the form $a, a + d^r, \dots, a + (k-1) d^r$ with $d \leq (n/k)^{1/r}$.

\begin{theorem} \label{ApproxPolySzem}
Let $k,r$ be integers and $\d>0$ a real number. Then there exists a constant $C$
such that if $p \geq C n^{-1/(k-1)r}$, then the probability that $(\Z_n)_p$ is $(\d,k,r)$-Szemer\'edi
is $1-o(1)$. 
\end{theorem}

\begin{proof}
Let $X = \Z_n$ and $S$ be the collection of progressions of the form $a, a + d^r, \dots, a + (k-1)d^r$ with $d \leq (n/k)^{1/r}$. Because of this restriction on $d$, $S$ has two degrees of freedom. It is also obviously homogeneous. The size of each $S_j(x)$ is $n^{1/r}$ to within a constant, so the critical exponent is $\g/(k-1)$ with $\g=1/r$. Therefore, provided $p$ is at most some constant $\lambda$, the result follows from Theorem~\ref{unconditionaldensity} and Corollary~\ref{ZnPolySzem}. For $p$ larger than $\lambda$, the result follows from the polynomial Szemer\'edi theorem itself.
\end{proof}

Note that the particular case of this theorem when $k=2$ was already proved by Nguyen \cite{N09}. To see that this result is sharp, note that the number of $k$-term progressions with $r$th power difference in the random set is roughly $p^k n^{1+1/r}$. This is smaller than the number of vertices $p n$ when $p = n^{-1/(k-1)r}$. 
\bigskip

We will now move on to proving sparse versions of Tur\'an's theorem for strictly $k$-balanced $k$-uniform hypergraphs. As we mentioned in the introduction, some of the dense results are not known, but this does not matter to us, since our aim is simply to show that whatever results can be proved in the dense case carry over to the sparse random case when the probability exceeds the critical probability.

For a $k$-uniform hypergraph $K$, let ex$(n,K)$ denote the largest number of edges a subgraph of $K_n^{(k)}$ can have without containing a copy of $K$. As usual, we need a robust result that says
that once a graph has more edges than the extremal number for $K$, by a constant proportion of the total number of edges in $K_n^{(k)}$, then it must contain many copies of $K$. The earliest version of such
a supersaturation result was proved by Erd\H{o}s and Simonovits \cite{ES83}. The proof is another easy averaging argument along the lines of the proof of Theorem \ref{Varnavides}. 

\begin{theorem} \label{ErdosSim}
Let $K$ be a $k$-uniform hypergraph. Then, for any
$\e > 0$, there exists $\d > 0$ such that if $L$ is a
$k$-uniform hypergraph on $n$ vertices and
\[e(L) > \mathrm{ex}(n,K) + \e n^k,\]
then $L$ contains at least $\d n^{v_K}$ copies of $K$.
\end{theorem}

Let $\pi_k(K)$ be the limit as $n$ tends to infinity of ex$(n,K)/\binom{n}{k}$. We will say that a $k$-uniform hypergraph $H$ is $(K,
\e)$-Tur\'an if any subset of the edges of $H$ of size
\[(\pi_k(K) + \e) e(H)\]
contains a copy of $K$. Recall that $G_{n,p}^{(k)}$ is a random $k$-uniform hypergraph on $n$ vertices,
where each edge is chosen with probability $p$, and when $K$ is strictly $k$-balanced $m_k(K)=(e_K-1)/(v_K-k)$. 

\begin{theorem} \label{RelativeHyperTuran}
For every $\epsilon > 0$ and every strictly $k$-balanced $k$-uniform hypergraph $K$, there
exists a constant $C$ such that if $p\geq C n^{-1/m_k(K)}$, then the probability that
$G_{n,p}^{(k)}$ is $(K,\e)$-Tur\'an is $1-o(1)$.
\end{theorem}

\begin{proof}
For $p$ smaller than a fixed constant $\lambda$, the result follows immediately from Theorems \ref{unconditionaldensity} and \ref{ErdosSim}. For $p \geq \lambda$, we may apply the argument discussed before Theorem \ref{conditionaldensity}. That is, we may partition $G_{n,p}^{(k)}$ into a small set of random graphs, each of which has density less than $\lambda$ and each of which is $(K,\e)$-Tur\'an. If now we have a subgraph of $G_{n,p}^{(k)}$ of density at least $\frac{\mathrm{ex}(n, K)}{\binom{n}{k}} + \e$, then this subgraph must have at least this density in one of the graphs from the partition. Applying the fact that this subgraph is $(K,\e)$-Tur\'an implies the result.
\end{proof}

In particular, this implies Theorem \ref{ApproxTuran}, which is the particular case of this theorem where $K$ is a strictly balanced graph. Then $\mathrm{ex}(n,K)$ is known to be $\left(1 - \frac{1}{\chi(K) - 1} + o(1)\right) \binom{n}{2}$, where $\chi(K)$ is the chromatic number of $K$.

\subsection{Colouring results}

We shall now move on to colouring results that do not follow from their corresponding density versions. Let us begin with Ramsey's theorem. As ever, the main thing we need to check is that a suitable robust version of the theorem holds. And indeed it does: it is a very simple consequence of Ramsey's theorem that was noted by Erd\H{o}s \cite{E62}. 

\begin{theorem} \label{GraphMult}
Let $H$ be a hypergraph and let $r$ be a positive integer.
Then there exists an integer $n_0$ and a constant $c > 0$ such
that, if $n \geq n_0$, any colouring of the edges of $K_n^{(k)}$ with
$r$ colours is guaranteed to contain $c n^{v_H}$ monochromatic copies
of $H$.
\end{theorem}

Once again the proof is the obvious averaging argument: choose $m$ such that if the edges of $K_m^{(k)}$ are coloured with $r$ colours, there must be a monochromatic copy of $H$, and then a double count shows that for every $r$-colouring of the edges of $K_n^{(k)}$ there are at least $\binom n{v_H}/\binom m{v_H}$ monochromatic copies of~$H$. 

Recall that, given a $k$-uniform hypergraph $K$ and a natural number $r$, a hypergraph 
is $(K,r)$-\textit{Ramsey} if every $r$-colouring of its edges contains a monochromatic copy of $K$. We are now ready to prove Theorem \ref{ApproxRamsey}, which for convenience we restate here.

\begin{theorem}
Given a natural number $r$ and a strictly $k$-balanced $k$-uniform hypergraph $K$, there
exists a positive constant $C$ such that if $p\geq Cn^{-1/m_k(K)}$, then the probability that 
$G_{n,p}^{(k)}$ is $(K,r)$-Ramsey is $1-o(1)$.
\end{theorem}

\begin{proof}
For a sufficiently large constant $C$, the result for $p=Cn^{-1/m_k(K)}$ follows from Theorems \ref{unconditionalcolouring} and \ref{GraphMult}. For $q > p$, the result follows from the monotonicity of the Ramsey property. To see this, choose a random hypergraph $G_{n,q}^{(k)}$ and then choose a subhypergraph by randomly selecting each edge of $G_{n,q}^{(k)}$ with probability $p/q$. The resulting hypergraph is distributed as $G_{n,p}^{(k)}$, so with probability $1-o(1)$ it is $(K,r)$-Ramsey. But then any $r$-colouring of $G_{n,q}^{(k)}$ will yield an $r$-colouring of this $G_{n,p}^{(k)}$, which always contains a monochromatic copy of $K$.
\end{proof}

With only slightly more effort we can obtain a robust conclusion. Theorem \ref{unconditionalcolouring} tells us that with high probability the number of monochromatic copies of $K$ in any $r$-colouring of $G_{n,p}^{(k)}$ is $cp^{e_K}n^{v_K}$ for some constant $c>0$, and then an averaging argument implies that with high probability the number of monochromatic copies in an $r$-colouring of $G_{n,q}^{(k)}$ is $cq^{e_K}n^{{v_K}}$.

We shall now take a look at Schur's theorem \cite{S16}, which states that if the set $\{1, \dots, n\}$ is $r$-coloured, then there exist monochromatic subsets of the form $\{x, y, x+y\}$. As with our results concerning Szemer\'edi's theorem, it is more convenient to work in $\Z_n$. To see that this implies the equivalent theorem in $[n]$, let $[n]_p$ be a random subset of $[n]$ made from  the union of two smaller random subsets, each chosen with probability $q$ such that $p = 2q - q^2$. Call these sets $U_1$ and $U_2$. Then the subset of $\Z_{2n}$ formed by placing the set $U_1$ in the position $\{1, \dots, n\}$ and the set $-U_2$ in the set $\{-1, \dots, -n\}$ (the overlap $n=-n$ is irrelevant to the argument, since it is unlikely to be in the set) will produce a random subset of $\Z_{2n}$ where each element is chosen with probability $q$. If a sparse version of Schur's theorem holds in $\Z_{2n}$, then with high probability, any $2r$-colouring of this random set yields $cq^3n^2$ monochromatic sets $\{x, y, x+y\}$ for some constant $c>0$. 

Consider now an $r$-colouring of the original set $U_1 \cup U_2$ in $r$ colours $C_1, \dots, C_r$. This induces a colouring of $U_1\cup -U_2\subset\Z_{2n}$ with $2r$ colours $C_1,\dots,C_{2r}$: if $x\in U_1$ and is coloured with colour $C_i$ in $[n]$, then we continue to colour it with colour $C_i$, whereas if $x\in -U_2$ and $-x$ has colour $C_i$ in $[n]$, then we colour it with colour $C_{i+r}$. We have already noted that this colouring must contain many monochromatic sets $\{x, y, x+y\}$, and each one corresponds to a monochromatic set (either $\{x,y,x+y\}$ or $\{-x,-y,-(x+y)\}$) in the original colouring. 

The robust version of Schur's theorem can be deduced from one of the standard proofs, which itself relies on Ramsey's theorem for triangles and many colours.

\begin{theorem} \label{Schur}
Let $r$ be a positive integer. Then there exists an integer $n_0$ and a constant $c$ such that, if $n \geq n_0$, any $r$-colouring of $\{1, \dots, n\}$ contains at least $c n^2$ monochromatic triples of the form $\{x, y, x+y\}$.
\end{theorem}

We shall say that a subset $I$ of the integers is $r$-\textit{Schur} if for every $r$-colouring of the points of $I$ there is a monochromatic triple of the form $\{x, y, x+y\}$. The $r=2$ case of the following theorem was already known: it is a result of Graham, R\"odl and Ruci\'nski \cite{GRR96}. 

\begin{theorem} \label{RelativeSchur}
For every positive integer $r$ there exists a constant $C$ such that if $p\geq Cn^{-1/2}$, then the probability that $(\Z_n)_p$ is $r$-Schur is $1-o(1)$.
\end{theorem}

\begin{proof}
Let $X = \Z_n\setminus\{0\}$ and $S$ be the collection of subsets of $X$ of the form $\{x, y, x+y\}$ with all of $x$, $y$ and $x+y$ distinct. Since any two of $x$, $y$ and $x+y$ determine the third, it follows that $|S_i(a) \cap S_j(b)| \leq 1$ whenever $i,j\in\{1,2,3\}$, $i \neq j$, and $a, b \in X$. Therefore, $S$ has two degrees of freedom. Furthermore, each $S_i(a)$ has size $n-3$. By Theorem~\ref{Schur}, there exists a constant $c$ such that, for $n$ sufficiently large, any $r$-colouring of $\Z_n$ contains at least $c n^2$ monochromatic subsets of the form $\{x, y, x+y\}$. Applying Theorem~\ref{unconditionalcolouring}, we see that there exist positive constants $C$ and $c'$ such that, with probability $1-o(1)$ a random subset $U$ of $\Z_n$ chosen with probability $p = C n^{-1/2}$ satisfies the condition that, in any $r$-colouring of $U$, there are at least $c' p^3 n^2$ monochromatic subsets of the form $\{x, y, x+y\}$. In particular, $U$ is $r$-Schur. Once again, the result for larger probabilities follows easily.
\end{proof}

As we mentioned in the introduction, it is quite a bit harder to prove $0$-statements for colouring statements than it is for density statements. However, $0$-statements for partition regular systems have been considered in depth by R\"odl and Ruci\'nski \cite{RR97}, and their result implies that Theorem \ref{RelativeSchur} is sharp.

A far-reaching generalization of Schur's theorem was proved by Rado \cite{R41}. It is likely that our methods could be used to prove other cases of Rado's theorem, but we have not tried to do so here, since we would have to impose a condition on the configurations analogous to the strictly balanced condition for graphs and hypergraphs.

\subsection{The hypergraph removal lemma}

Rather than jumping straight into studying hypergraphs, we shall begin by stating a slight strengthening of the triangle removal lemma for graphs. This strengthening follows from its proof via Szemer\'edi's regularity lemma and gives us something like the ``robust" version we need in order to use our methods to obtain a sparse result. If $G$ is a graph and $X$ and $Y$ are sets of vertices, we shall write $G(X,Y)$ for the set of edges that join a vertex in $X$ to a vertex in $Y$, $e(X,Y)$ for the cardinality of $G(X,Y)$ and $d(X,Y)$ for $e(X,Y)/|X||Y|$. 

\begin{theorem} \label{triangleremoval}
For every $a>0$ there exists a constant $K$ with the following property. For every graph $G$ with $n$ vertices, there is a partition of the vertices of $G$ into $k\leq K$ sets $V_1,\dots,V_k$, each of size either $\lfloor n/k\rfloor$ or $\lceil n/k\rceil$, and a set $E$ of edges of $G$ with the following properties.
\begin{enumerate}
\item The number of edges in $E$ is at most $an^2$.
\item $E$ is a union of sets of the form $G(V_i,V_j)$.
\item $E$ includes all edges that join a vertex in $V_i$ to another vertex in the same $V_i$.
\item Let $G'$ be $G$ with the edges in $E$ removed. For any $h,i,j$, if there are edges in all of $G'(V_h,V_i)$, $G'(V_i,V_j)$ and $G'(V_h,V_j)$, then the number of triangles $xyz$ with $x\in V_h$, $y\in V_i$ and $z\in V_j$ is at least $a^3|V_h||V_i||V_j|/128$.
\end{enumerate}
\end{theorem}

In particular, this tells us that after we remove just a few edges we obtain a graph that contains either no triangles or many triangles. Let us briefly recall the usual statement of the dense triangle removal lemma and see how it follows from Theorem \ref{triangleremoval}.

\begin{corollary} \label{densetriangleremoval}
For every $a>0$ there exists a constant $c>0$ with the following property. For every graph $G$ with $n$ vertices and at most $cn^3$ triangles it is possible to remove at most $an^2$ edges from $G$ in such a way that the resulting graph contains no triangles. 
\end{corollary}

\begin{proof} 
Apply Theorem \ref{triangleremoval} to $a$ and let $c=a^3/200K^3$. Now let $G$ be a graph with $n$ vertices. Let $V_1,\dots,V_k$ and $E$ be as given by Theorem \ref{triangleremoval} and remove from $G$ all edges in $E$. If we do this, then by Condition 1 we remove at most $an^2$ edges from $G$. If there were any triangle left in $G$, then by Condition 4 there would have to be at least $a^3\lfloor n/k\rfloor^3/128>cn^3$ triangles left in $G$, a contradiction. This implies the result.
\end{proof}

Here now is a sketch of how to deduce a sparse triangle removal lemma from Theorem \ref{triangleremoval}. We begin by proving a sparse version of Theorem \ref{triangleremoval} itself. Given a random graph $U$ with edge probability $p \geq Cn^{-1/2}$, for sufficiently large $C$, let $H$ be a subgraph of $U$. Now use Corollary~\ref{unconditionalsetstructural} to find a dense graph $G$ such that the triangle density of $G$ is roughly the same as the relative triangle density of $H$ in $U$ (that is, if $H$ has $\alpha p^3n^3$ triangles, then $G$ has roughly $\alpha n^3$ triangles) and such that for every pair of reasonably large sets $X,Y$ of vertices the density $d_G(X,Y)$ is roughly the same as the relative density of $H$ inside $U(X,Y)$ (that is, the number of edges of $G(X,Y)$ is roughly $p^{-1}$ times the number of edges of $H(X,Y)$).

Now use Theorem \ref{triangleremoval} to find a partition of the vertex set of $G$ (which is also the vertex set of $H$) into sets $V_1,\dots,V_k$ and to identify a set $E_G$ of edges to remove from $G$. By Condition 2, $E_G$ is a union of sets of the form $G(V_i,V_j)$. Define $E_H$ to be the union of the corresponding sets $H(V_i,V_j)$ and remove all edges in $E_H$ from $H$. If it happens that $G(V_i,V_j)$ is empty, then adopt the convention that we remove all edges from $H(V_i,V_j)$. Note that because the relative densities in dense complete bipartite graphs are roughly the same, the number of edges in $E_H$ is at most $2ap n^2$. Let $G'$ be $G$ with the edges in $E_G$ removed and let $H'$ be $H$ with the edges in $E_H$ removed.

Suppose now that $H'$ contains a triangle $xyz$ and suppose that $x\in V_h$, $y\in V_i$ and $z\in V_j$. Then none of $G'(V_h,V_i)$, $G'(V_i,V_j)$ and $G'(V_h,V_j)$ is empty, by our convention above, so Condition 4 implies that $G'$ contains at least $a^3|V_h||V_i||V_j|/128$ triangles with $x\in V_h$, $y\in V_i$ and $z\in V_j$. Since triangle densities are roughly the same, it follows that $H'$ contains at least $a^3p^3|V_h||V_i||V_j|/256$ triangles.

Roughly speaking, what this tells us is that Theorem \ref{triangleremoval} transfers to a sparse random version. From that it is easy to deduce a sparse random version of Corollary \ref{densetriangleremoval}. However, instead of giving the full details of this, we shall prove (in a very similar way) a more general theorem, namely a sparse random version of the simplex removal lemma for hypergraphs, usually known just as the hypergraph removal lemma.

The dense result is due to Nagle, R\"{o}dl, Schacht and Skokan \cite{NRS06, RS04}, and independently  to the second author \cite{G07}. A gentle introduction to the hypergraph removal lemma that focuses on the case of 3-uniform hypergraphs can be found in \cite{G06}. The result is as follows.

\begin{theorem} \label{hypergraphremoval}
For every $\d > 0$ and every integer $k \geq 2$, there exists a constant $\e > 0$ such that, if $G$ is a $k$-uniform hypergraph containing at most $\e n^{k+1}$ copies of $K_{k+1}^{(k)}$, it may be made $K_{k+1}^{(k)}$-free by removing at most $\d n^k$ edges.
\end{theorem}

A \textit{simplex} is a copy of $K_{k+1}^{(k)}$. As in the case of graphs, where simplices are triangles, it will be necessary to state a rather more precise and robust result. This is slightly more complicated to do than it was for graphs. However, it is much \textit{less} complicated than it might be: it turns out not to be necessary to understand the statement of the regularity lemma for hypergraphs.

Let us make the following definition. Let $H$ be a $k$-uniform hypergraph, and let $J_1,\dots,J_k$ be disjoint $(k-1)$-uniform hypergraphs with the same vertex set as $H$. We shall define $H(J_1,\dots,J_k)$ to be the set of all edges $A=\{a_1,\dots,a_k\}\in H$ such that $\{a_1,\dots,a_{i-1},a_{i+1},\dots,a_k\}\in J_i$ for every $i$. (Note that if $k=2$ then the sets $J_1$ and $J_2$ are sets of vertices, so we are obtaining the sets $G(X,Y)$ defined earlier.) 

Now suppose that we have a simplex in $H$ with vertex set $(x_1,\dots,x_{k+1})$. For every subset $\{u,v\}$ of $[k+1]$ of size 2, let us write $J_{uv}$ for the (unique) set $J_i$ that contains the $(k-1)$-set $\{x_j:j\notin\{u,v\}\}$. Then for each $u$ the set $H(J_{u1},\dots,J_{u,u-1},J_{u,u+1},\dots,J_{u,k+1})$ is non-empty. We make this remark in order to make the statement of the next theorem slightly less mysterious. It is an analogue for $k$-uniform hypergraphs of Theorem \ref{triangleremoval}. For convenience, we shall abbreviate $H(J_{u1},\dots,J_{u,u-1},J_{u,u+1},\dots,J_{u,k+1})$ by $H(J_{uv}:v\in[k+1],v\ne u)$. (It might seem unnecessary to write ``$v\in [k+1]$" every time. We do so to emphasize the asymmetry: the set depends on $u$, while $v$ is a dummy variable.)

\begin{theorem} \label{simplexremoval}
For every $a>0$ there exists a constant $K$ with the following property. For every $k$-uniform hypergraph $H$ with vertex set $[n]$, there is a partition of $\binom {[n]}{k-1}$ into at most $K$ subsets $J_1,\dots,J_m$, with sizes differing by a factor of at most 2, and a set $E$ of edges of $H$ with the following properties.
\begin{enumerate}
\item The number of edges in $E$ is at most $an^k$.
\item $E$ is a union of sets of the form $H(J_{i_1},\dots,J_{i_k})$.
\item $E$ includes all edges in any set $H(J_{i_1},\dots,J_{i_k})$ for which two of the $i_h$ are equal.
\item Let $H'$ be $H$ with the edges in $E$ removed. Suppose that for each pair of unequal integers $u,v\in[k+1]$ there is a set $J_{uv}$ from the partition such that the hypergraphs $H'(J_{uv}:v\in[k+1],v\ne u)$ are all non-empty. Then the number of simplices with vertices $(x_1,\dots,x_{k+1})$ such that the edge $(x_1,\dots,x_{u-1},x_{u+1},\dots,x_{k+1})$ belongs to $H'(J_{uv}:v\in[k+1],v\ne u)$ for every $u$ is at least $(1/2)(a/4)^{k+1}c_Kn^{k+1}$,
where $c_K$ is a constant that depends on $K$ (and hence on $a$).
\end{enumerate}
\end{theorem}

Let us now convert this result into a sparse version.

\begin{theorem} \label{sparsesimplexremoval}
For every $a>0$ there exist constants $C$, $K$ and $\d$ with the following property. Let $U$ be a random $k$-uniform hypergraph with vertex set $[n]$, and with each edge chosen independently with probability $C n^{-1/k} \leq p \leq \d$. Then with probability $1-o(1)$ the following result holds. For every $k$-uniform hypergraph $F\subset U$, there is a partition of $\binom {[n]}{k-1}$ into at most $K$ subsets $J_1,\dots,J_m$, with sizes differing by a factor of at most 2, and a set $E_F$ of edges of $F$ with the following properties.
\begin{enumerate}
\item The number of edges in $E_F$ is at most $apn^k$.
\item $E_F$ is a union of sets of the form $F(J_{i_1},\dots,J_{i_k})$.
\item $E_F$ includes all edges in any set $F(J_{i_1},\dots,J_{i_k})$ for which two of the $i_h$ are equal.
\item Let $F'$ be $F$ with the edges in $E_F$ removed. Suppose that for each pair of unequal integers $u,v\in[k+1]$ there is a set $J_{uv}$ from the partition such that the hypergraphs $F'(J_{uv}:v\in[k+1],v\ne u)$ are all non-empty. Then the number of simplices with vertices $(x_1,\dots,x_{k+1})$ such that the edge $(x_1,\dots,x_{u-1},x_{u+1},\dots,x_{k+1})$ belongs to $F'(J_{uv}:v\in[k+1],v\ne u)$ for every $u$ is at least $(1/4)(a/8)^{k+1}c_Kp^{k+1}n^{k+1}$,
where $c_K$ is a constant that depends on $K$.
\end{enumerate}
\end{theorem}

\begin{proof}
We have essentially seen the argument in the case of graphs. To start with, let us apply Corollary \ref{unconditionalsetstructural} with $S$ as the set of labelled simplices, $f$ as $p^{-1}$ times the characteristic function of $F$, $\mathcal{V}$ as the collection of all sets of the form $K_n^{(k)}(J_1,\dots,J_k)$ where each $J_i$ is a collection of sets of size $k-1$ (that is, the set of ordered sequences of length $k$ in $[n]$ such that removing the $i$th vertex gives you an element of $J_i$), and $\e=(1/4)(a/8)^{k+1}c_K$, where $K$ and $c_K$ come from applying Theorem~\ref{simplexremoval} with $a/2$ rather than $a$. With this choice, $\epsilon$ will also be less than $a/2K^k$. 

Note that $\a_S=1/k$ in this case, and that the cardinality of $\mathcal{V}$ is at most $2^{kn^{k-1}}$, so the corollary applies. From that we obtain a hypergraph $H$ (with characteristic function equal to the function $h$ provided by the corollary) such that $p^{-(k+1)}$ times the number of simplices in $F$ is at least the number of simplices in $H$ minus $\e n^{k+1}$, and such that the number of edges in  $H(J_1,\dots,J_k)$ differs from $p^{-1}$ times the number of edges in $F(J_1,\dots,J_k)$ by at most $\e n^k$ for every $(J_1,\dots,J_k)$. 

We now apply Theorem \ref{simplexremoval} to $H$ with $a$ replaced by $a/2$. Let $E_H$ be the set of edges that we obtain and let $H'$ be $H$ with the edges in $E_H$ removed. 

Let $J_1,\dots,J_m$ be the sets that partition $\binom{[n]}{k-1}$, and remove all edges from $F$ that belong to a set $F(J_{i_1},\dots,J_{i_k})$ such that the edges of $H(J_{i_1},\dots,J_{i_k})$ belong to $E_H$ (including when $H(J_{i_1},\dots,J_{i_k})$ is empty). Let $E_F$ be the set of removed edges and let $F'$ be $F$ after the edges are removed.

Since $m\leq K$, there are at most $K^k$ $k$-tuples $(J_{i_1},\dots,J_{i_k})$. For each such $k$-tuple  the number of edges in $H(J_{i_1},\dots,J_{i_k})$ differs from $p^{-1}$ times the number of edges in $F(J_{i_1},\dots,J_{i_k})$ by at most $\e n^k$. Therefore, since $E_H$ and $E_F$ are unions of sets of the form $H(J_{i_1},\dots,J_{i_k})$ and $F(J_{i_1},\dots,J_{i_k})$, respectively, and since $|E_H|\leq an^k/2$, it follows that $|E_F|\leq (a/2+\e K^k)pn^k\leq apn^k$. This gives us Condition 1. Conditions 2 and 3 are trivial from the way we constructed $E_F$. So it remains to prove Condition 4.

Suppose, then, that for all $u, v \in [k+1]$ there is a set $J_{uv}$ such that there are edges in all of the hypergraphs $F'(J_{uv}:v\in[k+1],v\ne u)$ for $u=1,2,\dots,k+1$. Then there must be edges in all the hypergraphs $H'(J_{uv}:v\in[k+1],v\ne u)$ as well, or we would have removed the corresponding sets of edges from $F$. By Condition 4 of the dense result applied to $H$, it follows that $H'$ contains at least $(1/2)(a/8)^{k+1}c_Kn^{k+1}$ simplices, which implies that $H$ does as well, which implies that $F$ contains at least $((1/2)(a/8)^{k+1}c_K-\e)p^{k+1}n^{k+1}$ simplices, which gives us the bound stated.
\end{proof}

Now let us deduce the simplex removal lemma. This is just as straightforward as it was for graphs.

\begin{corollary} \label{sparsehyperremoval}
For every $a>0$ there exist constants $C$ and $c>0$ with the following property. Let $U$ be a random $k$-uniform hypergraph with vertex set $[n]$, and with each edge chosen independently with probability $p\geq Cn^{-1/k}$. Then with probability $1-o(1)$ the following result holds. Let $F$ be a subhypergraph of $U$ that contains at most $cp^{k+1}n^{k+1}$ simplices. Then it is possible to remove at most $apn^k$ edges from $F$ and make it simplex free.
\end{corollary}

\begin{proof}
Let $c=(1/8)(a/8)^{k+1} c_K$, where $c_K$ is the constant given by Theorem \ref{sparsesimplexremoval}, and apply that theorem to obtain a set $E_F$, which we shall take as our set $E$. Then $E$ contains at most $apn^k$ edges, so it remains to prove that when we remove the edges in $E$ from $F$ we obtain a hypergraph $F'$ with no simplices. 

Suppose we did have a simplex in $F'$. Let its vertex set be $\{x_1,\dots,x_{k+1}\}$. For each $\{u,v\}\subset [k+1]$ of size 2, let $J_{uv}$ be the set from the partition given by Theorem \ref{sparsesimplexremoval} that contains the $(k-1)$-set $\{x_i:i\notin\{u,v\}\}$. Then, as we commented before the statement of Theorem \ref{simplexremoval} (though then we were talking about $H$), for each $u$ the set $F'(J_{uv}:v\in[k+1],v\ne u)$ is non-empty. Therefore, by Theorem \ref{sparsesimplexremoval}, $F'$, and hence $F$, contains at least $(1/4)(a/8)^{k+1} c_Kp^{k+1}n^{k+1}$ simplices. By our choice of $c$, this is a contradiction.

This argument works for $C n^{-1/k} \leq p \leq \d$. However, since $\d$ is a constant, we may, for $p > \d$, simply apply the hypergraph removal lemma itself to remove all simplices.
\end{proof}

\subsection{The stability theorem}

As a final application we will discuss the stability version of Tur\'an's theorem, Theorem \ref{ApproxStab}. The original stability theorem, due to Simonovits \cite{S68}, is the following.

\begin{theorem} \label{SimStab}
For every $\d > 0$ and every graph $H$ with $\chi(H) \geq 3$, there exists an $\e > 0$ such that any $H$-free graph with at least $\left(1 - \frac{1}{\chi(H) - 1} - \e\right) \binom{n}{2}$ edges may be made $(\chi(H) - 1)$-partite by removing at most $\d n^2$ edges.  
\end{theorem}

Unfortunately, this is not quite enough for our purposes. We would like to be able to say that a graph that does not contain too many copies of $H$ may be made $(\chi(H) - 1)$-partite by the deletion of few edges. To prove this, we appeal to the following generalization of the triangle removal lemma.

\begin{theorem} \label{graphremoval}
For every $\d > 0$ and every graph $H$, there exists a constant $\e > 0$ such that, if $G$ is a graph containing at most $\e n^{v_H}$ copies of $H$, then it may be made $H$-free by removing at most $\d n^2$ edges.
\end{theorem} 

Combining the two previous theorems gives us the robust statement we shall need.

\begin{theorem} \label{RobustStab}
For every $\d > 0$ and every graph $H$ with $\chi(H) \geq 3$, there exists a constant $\e$ such that any graph with at most $\e n^{v_H}$ copies of $H$ and at least $\left(1 - \frac{1}{\chi(H) - 1} - \e\right) \binom{n}{2}$ edges may be made $(\chi(H) - 1)$-partite by removing at most $\d n^2$ edges.  
\end{theorem}

To prove Theorem \ref{ApproxStab}, the statement of which we now repeat, we will follow the procedure described at the end of Section \ref{clemma}. 

\begin{theorem} \label{RelativeStabAgain}
Given a strictly $2$-balanced graph $H$ with $\chi(H) \geq 3$ and a constant $\d > 0$, there exist positive constants $C$ and $\e$ such that in the random graph $G_{n,p}$ chosen with probability $p \geq C n^{-1/m_2(H)}$, where $m_2(H) = (e_H - 1)/(v_H - 2)$, the following holds with probability tending to 1 as $n$ tends to infinity. Every $H$-free subgraph of $G_{n,p}$ with at least $\left(1 - \frac{1}{\chi(H) - 1} - \e\right)p \binom{n}{2}$ edges may be made $(\chi(H)-1)$-partite by removing at most $\d p n^2$ edges. 
\end{theorem}

\begin{proof}
Fix $\d > 0$. An application of Theorem \ref{RobustStab} gives us $\e > 0$ such that any graph with at most $\e n^{v_H}$ copies of $H$ and at least $\left(1 - \frac{1}{\chi(H) - 1} - 2\e\right) \binom{n}{2}$ edges may be made $(\chi(H) - 1)$-partite by removing at most $\d n^2/2$ edges. 

Let $t = \chi(H) - 1$. Apply Corollary \ref{unconditionalsetstructural} with $S$ the set of all labelled copies of $H$ in $K_n$ and $\mathcal{V}$ the set of all vertex subsets of $\{1, \dots, n\}$. This yields constants $C$ and $\lambda$ such that, for $C n^{-1/m_2(H)} \leq p \leq \lambda$, the following holds with probability tending to 1 as $n$ tends to infinity. Let $G$ be a random graph where each edge is chosen with probability $p$. Let $\mu$ be its characteristic measure. Then, if $f$ is a function with $0 \leq f \leq \mu$, there exists a $\{0,1\}$-valued function $j$ such that $\E_{s \in S} f(s_1)\dots f(s_e) \geq \E_{s \in S} j(s_1) \cdots j(s_e) - \e$ and, for all $V \in \mathcal{V}$, $|\E_{x \in V} f(x) - \E_{x \in V} j(x)| \leq \eta \frac{|X|}{|V|}$, where $\eta = \min(\e, \d/2t)$. 

Let $A$ be a $H$-free subgraph of $G$ with $\left(1 - \frac{1}{\chi(H) - 1} - \e\right) p \binom{n}{2}$ edges and let $0 \leq f \leq \mu$ be $p^{-1}$ times its characteristic function. Apply the transference principle to find the function $j$, which is the characteristic measure of a graph $J$. The number of copies of $H$ in $J$ is at most $\e n^{v_H}$. Otherwise, we would have 
\[\E_{s \in S} f(s_1)\dots f(s_e) \geq \E_{s \in S} j(s_1) \cdots j(s_e) - \e > 0,\]
implying that $A$ was not $H$-free. Moreover, the number of edges in $J$ is at least $\left(1 - \frac{1}{\chi(H) - 1} - 2\e\right) \binom{n}{2}$. Therefore, by the choice of $\e$, $J$ may be made $(\chi(H) - 1)$-partite by removing at most $\d n^2/2$ edges. 

Let $V_1, \dots, V_t$ be the partite pieces. By transference, $|\E_{x \in V_i} f(x) - \E_{x \in V_i} j(x)| \leq \eta \frac{|X|}{|V_i|}$ for each $1 \leq i \leq t$. Therefore, if we remove all of the edges of $A$ from each set in $V_i$, we have removed at most
\[\sum_{i=1}^t \sum_{x \in V_i} f(x) \leq \sum_{i=1}^t \sum_{x \in V_i} j(x) + t \eta |X| \leq \left(\frac{\d}{2} + t\eta\right)n^2 \leq \d n^2.\]
Moreover, the graph that remains is $(\chi(H) - 1)$-partite, so we are done.

It only remains to consider the case when $p > \lambda$. However, as observed in \cite{K97}, for $p$ constant, the theorem follows from an application of the regularity lemma. This completes the proof.
\end{proof}

As a final note, we would like to mention that the method used in the proof of Theorem \ref{RelativeStabAgain} should work quite generally. To take one more example, let $K$ be the Fano plane. This is the hypergraph formed by taking the seven non-zero vectors of dimension three over the field with two elements and making $xyz$ an edge if $x + y + z = 0$. The resulting graph has seven vertices and seven edges. It is known \cite{DCF00} that the extremal number of the Fano plane is approximately $\frac{3}{4} \binom{n}{3}$. Since the Fano plane is strictly $3$-balanced, Theorem \ref{RelativeHyperTuran} implies that if $U$ is a random $3$-uniform hypergraph chosen with probability $p \geq C n^{-2/3}$, then, with high probability, $U$ is such that any subgraph of $U$ with at least $\left(\frac{3}{4} + \e\right) |U|$ edges contains the Fano plane.

Moreover, it was proved independently by Keevash and Sudakov \cite{KS05} and F\"uredi and Simonovits \cite{FS05} that the extremal example is formed by dividing the ground set into subsets $A$ and $B$ of nearly equal size and taking all triples that intersect both as edges. The stability version of this result says that, for all $\d > 0$, there exists $\e > 0$ such that any $3$-uniform hypergraph on $n$ vertices with at least $\left(\frac{3}{4} - \e\right) \binom{n}{3}$ edges that does not contain the Fano plane may be partitioned into two parts $A$ and $B$ such that there are at most $\d n^3$ edges contained entirely within $A$ or $B$. The same proof as that of Theorem \ref{RelativeStabAgain} then implies the following theorem.

\begin{theorem} \label{Fano}
Given a constant $\d > 0$, there exist positive constants $C$ and $\e$ such that in the random graph $G_{n,p}^{(3)}$ chosen with probability $p \geq C n^{-2/3}$, the following holds with probability tending to 1 as $n$ tends to infinity. Every subgraph of $G_{n,p}^{(3)}$ with at least $\left(\frac{3}{4} - \e\right) e(G)$ edges that does not contain the Fano plane may be made bipartite, in the sense that all edges intersect both parts of the partition, by removing at most $\d p n^3$ edges. 
\end{theorem}

\section{Concluding remarks} \label{Conclusion}

One question that the results of this paper leave open is to decide whether or not the thresholds we have proved are sharp. By saying that a threshold is sharp, we mean that the window over which the phase transition happens becomes arbitrarily small as the size of the ground set becomes large. For example, a graph property $\mathcal{P}$ has a sharp threshold at $\hat{p} = \hat{p}(n)$ if, for every $\e > 0$,
\[\lim_{n \rightarrow \infty} \mathbb{P} (G_{n,p} \mbox{ satisfies $\mathcal{P}$}) =
\left\{ \begin{array}{ll} 0, \mbox{ if } p < (1-\e) \hat{p},\\ 1,
\mbox{ if } p > (1+\e) \hat{p}.\end{array} \right.\]
Connectedness is a simple example of a graph property for which a sharp threshold is known. The appearance of a triangle, on the other hand, is known not to be sharp. A result of Friedgut \cite{F99} gives a criterion for judging whether a threshold is sharp or not. Roughly, this criterion says that if the property is globally determined, it is sharp, and if it is locally determined, it is not. This intuition allows one to conclude fairly quickly that connectedness should have a sharp threshold and the appearance of any particular small subgraph should not. 

For the properties that we have discussed in this paper it is much less obvious whether the bounds are sharp or not. Many of the properties are not even monotone, which is crucial if one wishes to apply Friedgut's criterion. Nevertheless, the properties do not seem to be too pathological, so perhaps there is some small hope that the sharpness of their thresholds can be proved. There has even been some success in this direction already. Recall that the threshold at which $G_{n,p}$ becomes $2$-colour Ramsey with respect to triangles is approximately $n^{-1/2}$. A difficult result of Friedgut, R\"odl, Ruci\'nski and Tetali \cite{FRRT06} states that this threshold is sharp. That is, there exists $\hat{c} = \hat{c}(n)$ such that, for every $\e > 0$,
\[\lim_{n \rightarrow \infty} \mathbb{P} (G_{n,p} \mbox{ is $(K_3, 2)$-Ramsey}) =
\left\{ \begin{array}{ll} 0, \mbox{ if } p < (1-\e) \hat{c} n^{-1/2},\\ 1,
\mbox{ if } p > (1+\e) \hat{c} n^{-1/2}.\end{array} \right.\]
Unfortunately, the function $\hat{c}(n)$ is not known to tend towards a constant. It could, at least in principle, wander up and down forever between the two endpoints. Nevertheless, we believe that extending this result to cover all (or any) of the theorems in this paper is important.

There are other improvements that it might well be possible to make. We proved our graph and hypergraph results for strictly balanced graphs and hypergraphs, while the results of Schacht \cite{S09} and Friedgut, R\"odl and Schacht \cite{FRS09} apply to all graphs and hypergraphs. On the other hand, our methods allow us to prove structural results such as the stability theorem which do not seem to follow from their approach. It seems plausible that some synthesis of the two approaches could allow us to extend these latter results to general graphs and hypergraphs in a tidy fashion.\footnote{In subsequent work, Samotij \cite{Sj14} showed how to adapt Schacht's method so that it also applies to structural statements such as Theorem~\ref{ApproxStab}. However, it still remains an open problem to extend the methods of this paper to all graphs and hypergraphs.}

In our approach, restricting to strictly balanced graphs and hypergraphs was very convenient, since it allowed us to cap our convolutions only at the very last stage (that is, when all the functions involved had sparse random support). In more general cases, capping would have to take place ``all the way down". It seems likely that this can be done, but that a direct attempt to generalize our methods would be messy. 

A more satisfactory approach would be to find a neater way of proving our probabilistic estimates. The process of capping is a bit ugly: a better approach might be to argue that with high probability we can say roughly how the modulus of an uncapped convolution is distributed, and use that in an inductive hypothesis. (It seems likely that the distribution is approximately Poisson.)

Thus, it seems that the problem of extending our methods to general graphs and hypergraphs and the problem of finding a neater proof of the probabilistic estimates go hand in hand.

\end{document}